\documentclass[11pt,a4paper]{amsart}
\usepackage[centering, margin=1in]{geometry}
\usepackage[latin1]{inputenc}
\usepackage[english]{babel}
\usepackage[T1]{fontenc}
\usepackage{amsmath}
\usepackage{amssymb}
\usepackage{amsthm}
\usepackage{array,booktabs}
\usepackage{graphicx}
\usepackage{hyperref}
\usepackage{euflag}
\usepackage{lmodern}
\usepackage{tcolorbox}
\usepackage{quoting}
\usepackage{enumerate}
\usepackage{mathtools}
\usepackage{csquotes}
\usepackage{stmaryrd}
\usepackage{cite}
\usepackage{tikz}
\usepackage{todonotes}
\usepackage{tikz-cd}
\usetikzlibrary{matrix,arrows,decorations.pathmorphing}
\usepackage{wrapfig}
\mathchardef\mhyphen="2D
\renewcommand{\Re}{\operatorname{Re}}

\def\C{\ensuremath\mathbb{C}}
\def\d{\ensuremath\mathrm{d}}
\def\A{\ensuremath\mathbb{A}}

\def\Z{\ensuremath\mathbb{Z}}
\def\Q{\ensuremath\mathbb{Q}}
\def\Oo{\ensuremath\mathcal{O}}
\def\N{\ensuremath\mathbb{N}}

\def\F{\ensuremath\mathbb{F}}
\def\tr{\ensuremath\mathrm{tr}}

\def\Tr{\ensuremath\mathrm{Tr}}
\def\Frob{\ensuremath\mathrm{Frob}}

\usepackage{hyperref}
\usepackage[final]{pdfpages}
\usepackage{verbatim}
\newtheorem{theorem}{Theorem}[section]
\newtheorem{definition}[theorem]{Definition}
\newtheorem{corollary}[theorem]{Corollary}
\newtheorem{lemma}[theorem]{Lemma}
\newtheorem{Conjecture}[theorem]{Conjecture}
\newtheorem{proposition}[theorem]{Proposition}

\theoremstyle{remark}
\newtheorem{remark}{Remark}[section]
\newtheorem{example}{Example}[section]

\def\eps{\ensuremath\varepsilon}
\def\rank{\text{\rm rank}}

\def\im {\text{\rm im}}

\def\Gal{\text{\rm Gal}}

\def\SL{\mathrm{SL}}
\def\PSL{\mathrm{PSL}}

\def\GL{\mathrm{GL}}

\def\LL{\mathcal{L}}

\def\ev{v}
\def\T{\underline{T}}
\def\a{{\underline{\alpha}}}
\def\b{{\underline{\beta}}}

\def\supp{\mathrm{supp}}

\def\modulo{\text{ \rm mod }}

\def\1{\mathbf{1}}
\def\cond{\text{\rm cons}}

\DeclareMathOperator{\disc}{disc}

\DeclareMathOperator{\Spec}{Spec}

\DeclareMathOperator{\Hom}{Hom}

\DeclareFontFamily{U}{wncy}{}
    \DeclareFontShape{U}{wncy}{m}{n}{<->wncyr10}{}
    \DeclareSymbolFont{mcy}{U}{wncy}{m}{n}
    \DeclareMathSymbol{\Sh}{\mathord}{mcy}{"58} 
\def\modulo{\text{ \rm mod }}

\def\cond{\text{\rm cond}}

\def\pmod{\text{ \rm mod }}
\numberwithin{equation}{section}

\numberwithin{equation}{section}

\begin{document}
\title[Horizontal $p$-adic $L$-functions]{Horizontal $p$-adic $L$-functions}
\author{Daniel Kriz and Asbj\o rn Christian Nordentoft}

\address{University of Milan, Dipartimento di Matematica ``Federigo Enriques'', Via Cesare Saldini 50, 20133 Milan (MI), Italy}

\address{University of Copenhagen, Department of Mathematical Sciences, Universitetsparken 5, 2100 Copenhagen \O, Denmark}

\email{\href{mailto:daniel.kriz@unimi.it}{daniel.kriz@unimi.it}}

\email{\href{mailto:nordentoft@math.ku.dk}{nordentoft@math.ku.dk}}

\date{\today}

\thanks{
{\euflag}This project has received funding from the European Union's Horizon 2020 research and innovation programme under the Marie Sk\l odowska-Curie grant agreement No 101034255, from the Independent Research Fund Denmark DFF-1025-00020B and from a public grant from the Fondation Math\'{e}matique Jacques Hadamard.}

\begin{abstract} 
We define new objects called \emph{horizontal $p$-adic $L$-functions} associated to $L$-values of twists of elliptic curves over $\mathbb{Q}$ by characters of $p$-power order and conductor prime to $p$. We study the fundamental properties of these objects and obtain applications to non-vanishing of finite order twists of central $L$-values, making progress toward conjectures of Goldfeld and David--Fearnley--Kisilevsky. For general elliptic curves $E$ over $\mathbb{Q}$ we obtain strong quantitative lower bounds on the number of non-vanishing central $L$-values of twists by Dirichlet characters of fixed order $d\equiv 2 \modulo 4$ greater than two. We also obtain non-vanishing results for general $d$, including $d = 2$, under mild assumptions. In particular, for elliptic curves with $E[2](\mathbb{Q}) = 0$ we improve on the previously best known lower bounds on the number of non-vanishing $L$-values of quadratic twists due to Ono. Finally, we obtain results on simultaneous non-vanishing of twists of an arbitrary number of elliptic curves with applications to Diophantine stability. 
\end{abstract}
\maketitle
\tableofcontents
\section{Introduction}
Let $E/\Q$ be an elliptic curve and let $K/\Q$ be an abelian extension. A classical question in arithmetic statistics asks to understand the behavior of $\rank_\Z E(K)$ as $K$ varies over various families of number fields (see e.g. \cite{MaRu_stab}). By the Birch and Swinnerton-Dyer (BSD)  Conjecture this is related to understanding the non-vanishing set 
\begin{equation}\label{eq:charset}\{ \chi\in\widehat{\Gal(K/\Q)}:  L(E,\chi,1)\neq 0\}.\end{equation} 
This paper is concerned with obtaining non-vanishing results for such twisted $L$-values. 
The two key examples of families of abelian extensions are the following:
\begin{enumerate}
\item \emph{Vertical families}: fix a prime $p$ and let $K$ vary over cyclotomic $p$-power extensions.
\item \emph{Horizontal families}: fix an integer $d\geq 2$ and let $K$ vary over cyclic extensions of order $d$.
\end{enumerate} 
The vertical families are the object of study in Iwasawa theory and it is a celebrated result of Mazur \cite{Mazur72} that if $E$ has ordinary good reduction at $p$ then the Mordell--Weil rank is uniformly bounded over all $p$-power cyclotomic extensions and, correspondingly, Rohrlich \cite{Rohrlich84} famously proved that only finitely many central $L$-values of $p$-power cyclotomic character twists of $E$ are zero. In the horizontal case when $d=2$, Goldfeld made his influential Fifty-Fifty Conjecture \cite{Golfeld5050} with recent substantial progress obtained by combining the works of Smith \cite{Smith} and the first named author \cite{Kriz}, in particular settling the conjecture (both the algebraic and analytic part) for the congruent number family.  Other progress on Goldfeld's conjecture was made using mod 3 congruences by Chao Li and the first author \cite{Kriz16}, \cite{KrizLi}, \cite{KrizLi2} and in work of Castella--Grossi--Lee--Skinner \cite{CastellaGrossiLeeSkinner} (see also Section \ref{sec:previousresults} below). 
When $d$ is equal to an odd prime $p$ there are precise conjectures due to David--Fearnley--Kisilevsky \cite{DaFeKi07} (see also  \cite{MazurRubin21}) which in particular predict that for $100\%$ of cyclic degree $p$ extension $K/\Q$ we have $\rank_\Z E(K)=\rank_\Z E(\Q)$ and correspondingly that $L(E,\chi,1)\neq 0$ for $100\%$ of order $p$ Dirichlet characters $\chi$. To the authors' knowledge no progress has been made toward the conjectures on non-vanishing of higher order twists outside the cases studied in the work of Fearnley--Kisilevsky--Kuwata\cite{FeKiKu12}. In particular, nothing seems to be known in the case of non-zero rank nor when $d$ is composite. This is due to the fact that many of the methods in the quadratic case $d=2$ (e.g. the theta correspondence, calculating the first moment using techniques from analytic number theory) all seem to fall short in the case $d>2$ (see Section \ref{sec:previousresults} below for more details).     

In this paper, we introduce a $p$-adic approach in the horizontal setting reminiscent of Mazur's vertical Iwasawa theoretic methods which allows us to obtain strong quantitative lower bounds for non-vanishing of twists of $L$-values by characters of (fixed) order $d$ in large generality. Our methods also apply to simultaneous non-vanishing and we obtain results which go beyond what was previously known, even in the quadratic case $d=2$. The results of this paper are obtained by associating to $E/\Q$ certain horizontal measures (\emph{horizontal $p$-adic $L$-functions}) and studying the character zeroes of such measures. We refer to Section \ref{sec:methodproof} for a more detailed overview of our approach.

\subsection{Non-vanishing Results}
To state our results, let $d\geq 2$ be an integer, $X\geq 1$ and define the following set of Dirichlet characters:
$$\mathcal{K}_{d}(X):=\{ \chi\modulo D: \text{primitive of order $d$}, D\leq X\},$$
which by Corollary \ref{cor:numberchar} below satisfies:
$$ \# \mathcal{K}_{d}(X) =(c_d+o(1))X (\log X)^{\sigma_0(d)-2},\quad \text{as }X\rightarrow \infty,$$
for some $c_d>0$, where $\sigma_0(d)=\sum_{h|d}1$ denotes the divisor function. Our first non-vanishing result is:
\newpage \begin{theorem}\label{thm:non-vanishing}
Let $E/\Q$ be an elliptic curve and let $d\geq 2$ be an integer such that one of the following holds: 
\begin{itemize}
\item $d\equiv 2\modulo 4$ and $d\geq 6$.
\item $d$ is even and the mod $2$ representation $\overline{\rho}_{E,2}$ is irreducible. 
    \item $d$ is odd, $L(E,1)\neq0$ and the mod $p$ representation $\overline{\rho}_{E,p}$ is irreducible for some prime $p|d$.
\end{itemize}
Then there exists a constant $\alpha=\alpha_{E,d}>0$ such that   
\begin{equation}\label{eq:quantbound} \#\{\chi\in \mathcal{K}_{d}(X): L(E,\chi,1)\neq0\}\gg_{E,d} \frac{X}{(\log X)^{1-\alpha}},\quad \text{as }X\rightarrow \infty.\end{equation}
\end{theorem}
We note that even in the first case $d=6$ treated by the above theorem, it was previously not known whether there were infinitely many non-vanishing order six twists. Specializing the above theorem to the quadratic case $d=2$  improves on a result of Ono \cite{Ono01} (see Remark \ref{rmk:Ono} below). 

Let $E_1,\ldots ,E_n$ be elliptic curves over $\Q$ of respective conductors $N_1,\ldots, N_n$.  We say $p$ is \emph{$(E_1,\ldots ,E_n)$-good} if there exists a positive proportion of \emph{orderly primes} $\ell$, which we define to be primes satisfying\footnote{This condition for $E/\Q$ coincides with \emph{Taylor--Wiles primes} as in \cite[p.\ 555]{TaylorWiles95} (i.e.\ the image of Frobenius has distinct eigenvalues mod $p$) but for non-trivial nebentypus \emph{orderly primes} as in Definition \ref{def:TW} (i.e.\ 1 is not an eigenvalue of the image of Frobenius mod $p$) are different. We emphasize that we do not use Taylor-Wiles patching or any patching methods in our discussion.}
\begin{equation}\label{eq:TWcond} \ell\equiv 1\modulo p,\quad (\ell,N_i)=1,\quad p\nmid \# E_i(\F_\ell),\quad i=1,\ldots,n,\end{equation} 
see Definitions \ref{def:TW} and \ref{def:TWjoint}. We then obtain the following simultaneous non-vanishing result:
\begin{theorem}\label{thm:simultnonv}
Let $E_1,\ldots, E_n$ be elliptic curves over $\Q$ such that $L(E_i,1)\neq 0$ for $i=1,\ldots, n$. Let $d\geq 2$ and assume that $p$ is $(E_1,\ldots ,E_n)$-good for some prime $p|d$. Then there exists a constant $\alpha=\alpha_{E_1,\ldots, E_n,d}>0$ such that  
\begin{align}
\#\{ \chi\in \mathcal{K}_d(X): L(E_1,\chi,1)\cdots L(E_n,\chi,1)\neq 0 \}\gg_{E_i,d}
\frac{X}{(\log X)^{1-\alpha}},\quad \text{as }X\rightarrow \infty.\end{align} 
\end{theorem}
By the work of Gross-Zagier \cite{GrossZagier86} and Kolyvagin \cite{Kolyvagin88}, we know that $L(E,1) \neq 0$ implies $\mathrm{rank}_{\Z}E(\Q) = 0$ for any elliptic curve $E/\Q$. Applying this to $E$ and its quadratic twists and combining it with Theorem \ref{thm:simultnonv} in the case $d=2$ (see Section \ref{sec:final} for details) yields the following ``simultaneous Diophantine stability'' result for quadratic extensions of $\Q$:
\begin{corollary}\label{cor:simultnonv}
  Let $E_1,\ldots, E_n$ be elliptic curves over $\Q$ such that $L(E_i,1)\neq 0$ for $i=1,\ldots, n$ and assume that $2$ is $(E_1,\ldots ,E_n)$-good. Then with $\alpha=\alpha_{E_1,\ldots, E_n,2}>0$ as in Theorem \ref{thm:simultnonv} we have
  \begin{align}
\#\{ K/\Q:[K\!:\!\Q]=2, |\disc(K)|\leq X, \rank_\Z E_i(K)=0, i=1,\ldots,n\}\gg_{E_i}
\frac{X}{(\log X)^{1-\alpha}},\end{align}
as $X\rightarrow \infty$.
\end{corollary}
\begin{remark}
For a non-CM elliptic curve $E/\Q$  and $p\geq 13$ it follows from a result of Zywina \cite[Prop.\ 1.13]{Zywina2015} that $p$ is always $E$-good. Furthermore, it follows from the works of Jones \cite[Corollary 1.3]{Jones} and Bhargava--Shankar \cite[Theorem 6]{BhaSha15}  that for one hundred percent of $n$-tuples of distinct elliptic curves $E_1,\ldots, E_n$ with analytic rank zero (ordered by naive height) it holds that $p$ is $(E_1,\ldots, E_n)$-good for \emph{all} primes $p$. In particular, the conclusions of Theorem \ref{thm:simultnonv} and Corollary \ref{cor:simultnonv} follow. This yields simultaneous non-vanishing results for an arbitrary number of elliptic curves which are new even in the quadratic case $d=2$. To the authors' knowledge the only previous result for simultaneous non-vanishing in this setting is the work of Munshi \cite{Munshi12} concerning quadratic twists of two elliptic curves. 
\end{remark} 
\begin{remark} \label{rmk:Ono}
In the case where $d=2^m$ is a power of two, the irreducibility of $\overline{\rho}_{E,2}$ is equivalent to $E$ not having non-trivial rational $2$-torsion, i.e. $E[2](\Q)=0$. In this case Theorem \ref{thm:non-vanishing} holds with 
$$\alpha_{E,2^m}=\begin{cases} 2m/3,& \im(\overline{\rho}_{E,2})\cong \Z/3,\\ 
m/3,& \im(\overline{\rho}_{E,2})\cong S_3,\\
 \end{cases}$$ 
 which improves on a result of Ono \cite{Ono01} (who in the case $d=2$ obtained an inexplicit value for $\alpha$ depending on $E$) and furthermore, extends it to twists of order equal to a general power of two. See Corollary \ref{cor:2^m} for details. 
\end{remark}
\begin{remark}
Theorems \ref{thm:non-vanishing} and \ref{thm:simultnonv} generalize to holomorphic newforms of even weight and we also obtain bounds on the $p$-adic valuation for $p|d$. See Corollaries \ref{cor:d=2(4)}, \ref{cor:2^m} and \ref{cor:simultnonvgen} for the exact statements. 
\end{remark}
\subsubsection{The mod $p$ Kurihara Conjecture and non-vanishing} 
As we will explain in Section \ref{sec:methodproof}, the proofs of the above statements all rely on some initial non-vanishing either of $L(E,1)$ or when $d$ is even of some quadratic twist obtained using \cite{FriedbergHoffstein95}. This is not available when $d$ is odd and $E/\Q$ has rank $\geq 1$. Instead, in the higher rank setting there is an arithmetic way to obtain initial non-vanishing via a conjecture of Kurihara \cite{Kurihara14}, which is a $\Q$-analogue of \emph{Kolyvagin's Conjecture} \cite{WZhang}. 

More precisely, let $p$ be a prime not dividing the Manin constant of $E$ and such that $E[p](\Q)=0$. We say that $E/\Q$ satisfies the \emph{mod $p$ Kurihara Conjecture} (Conjecture \ref{conj:Kurihara}) if there exists a finite number of distinct \emph{Kato primes} $q_1,\ldots, q_r$ (i.e. $q_i \equiv 1\modulo p$ and $p\mid \#E(\F_{q_i})$) such that the \emph{Kurihara number}, or \emph{Kato--Kolyvagin derivative}, is non-zero modulo $p$:
\begin{equation}\label{eq:Kolderiv}\sum_{a_1=1}^{q_1-1}\cdots \sum_{a_r=1}^{q_r-1}\left(\prod_{i=1}^ra_i\right) \left\langle \tfrac{\prod_{i=1}^r(\zeta_{q_i})^{a_i}}{q_1\cdots q_r} \right\rangle^+_E\not\equiv 0\modulo p,\end{equation}
where $\zeta_{q_i}$ denotes a generator of $(\Z/q_i)^\times$ for $i=1,\ldots, r$ considered as an element of $(\Z/q_1\cdots q_r)^\times$ by putting $\zeta_{q_i}\equiv 1\modulo \tfrac{q_1\cdots q_r}{q_i}$, and 
\begin{equation}\label{eq:modularsymbol}\langle \tfrac{a}{q} \rangle^+_E:=(\Omega^+_E)^{-1}\Re \left(-2\pi i\int_{a/q}^\infty f_E(z)dz\right)\in \Z_{(p)},\end{equation} 
denotes the (plus) modular symbol associated to $E$ with $\Omega^+_E$ the real N\'{e}ron period of $E$ (here $\Z_{(p)}$ denotes rational numbers with denominators prime to $p$). Note that the condition (\ref{eq:Kolderiv}) does not depend on the choice of $(\zeta_{q_i})_{1\leq i\leq r}$. Furthermore, the minimal such $r$ is conjectured to be (essentially) equal to the Selmer rank \cite{Kurihara14}. To state our results, let $\C_p$ be the $p$-adic complex numbers and let $\ev_p:\C_p\rightarrow \Q\cup\{\infty\}$ be the unique valuation with $\ev_p(p)=1$. Fix an embedding $\overline{\Q}\subset \C_p$.
\begin{theorem}\label{thm:Kolynonvan}
Assume that $E/\Q$ satisfies the mod $p$ Kurihara Conjecture (Conjecture \ref{conj:Kurihara}) with corresponding $r\geq 1$. For example, this holds when $E/\Q$ satisfies the assumptions of Theorem \ref{thm:BGCSthm}. Then there exists a constant $\alpha=\alpha_{p}>0$ (independent of $E$) such that
\begin{equation}\label{eq:kuriharaquant} \#\left\{\chi\in \mathcal{K}_{p}(X): \ev_p\left( L(E,\chi,1)/\Omega^+_E\right)\leq \tfrac{r}{p-1}\right\}\gg_{E,p} \frac{X}{(\log X)^{1-\alpha}},\quad \text{as }X\rightarrow \infty,\end{equation}
where $\Omega^+_E$ denotes the real N\'{e}ron period of $E$. 

Furthermore, if the mod $p$ Kurihara Conjecture holds (\ref{eq:Kolderiv}) for the Kato primes $q_1,\ldots, q_r$ and $r\leq p-2$. Then for a \emph{positive proportion} of order $p$ Dirichlet characters $\chi$ with conductor of the shape $\ell q_1\cdots q_r$ with $\ell$ prime we have \begin{equation}\label{eq:r=1Kuri}\tfrac{1}{p-1}\leq \ev_p\left(L(E,\chi,1)/\Omega^+_E\right)\leq  \tfrac{r}{p-1},\end{equation}
\end{theorem} 
Recently, Burungale--Castella--Grossi--Skinner \cite{BurungaleCastellaGrossiSkinner} have obtained  progress towards the mod $p$ non-triviality of  the Kolyvagin system attached to Kato's Euler system (Theorem \ref{thm:BGCSthm}). Using its consequences for the mod $p$ Kurihara Conjecture established by C. H. Kim \cite{Kim} (Theorem \ref{thm:Kimthm}) we conclude: 
\begin{corollary}\label{cor:Kolynonvan} Let $E/\Q$ be a fixed non-CM elliptic curve. Then for $100\%$ of integers $d\geq 2$ there exists $\alpha=\alpha_d>0$ such that
$$ \#\{\chi\in \mathcal{K}_{d}(X): L(E,\chi,1)\neq 0\}\gg_{E,d} \frac{X}{(\log X)^{1-\alpha}},\quad \text{as }X\rightarrow \infty.$$
When we restrict to $d=p$ prime, the same holds for $100\%$ of primes $p$.
\end{corollary}
\begin{remark}In \cite{FeKiKu12}, Fearnley--Kisilevsky--Kuwata studied the case corresponding to $r=0$, i.e. $L(E,1)\neq 0$ and $\ev_p(L(E,1)/\Omega_E^+)=0$, and obtained the analogue of Theorem \ref{thm:Kolynonvan} in this setting. Their elegant proof follows by combining the horizontal norm relations and the congruence $\chi\equiv 1\modulo p^{1/(p-1)}$ for $\chi$ a Dirichlet character of order $p$. This corresponds in our argument to the case where the horizontal $p$-adic $L$-function is invertible and thus has no zeroes, see Proposition \ref{prop:invertible}. \end{remark}


\subsection{Method of Proof}\label{sec:methodproof}
We will now describe the method of proof in the setting of Theorem \ref{thm:non-vanishing}. The starting point is a formula due to Birch and Stevens by which one can express the $L$-values $L(E,\chi,1)$ as a twisted sum of \emph{modular symbols} (Corollary \ref{cor:BirchStevens}). Using the horizontal norm relations for modular symbols (Section \ref{sec:normrelQ}), we associate to $E$ and a sequence of \emph{orderly} primes $\ell_1,\ell_2,\ldots$ (i.e. primes satisfying (\ref{eq:TWcond})) an element of the   \emph{pro-$p$ horizontal Iwasawa algebra}:
$$\nu_E\in \Lambda^\mathrm{hor}:=\Z_p\left\llbracket \prod_{n\in \N} \Z/p^{m_n}\right\rrbracket=\varprojlim_{N}\Z_p\left[\prod_{n\leq N} \Z/p^{m_n}\right] ,$$
where $p^{m_n}|\!| (\ell_n-1)$ for $n\geq 1$. Note that elements of $\Lambda^\mathrm{hor}$ can be identified with $\Z_p$-valued measures on the profinite group $\prod_{n\in \N} \Z/p^{m_n}$, see (\ref{eq:homcts}). The horizontal measure $\nu_E$ encodes the twisted $L$-values $L(E,\chi,1)$ (appropriately normalized) with $\chi$ a Dirichlet character of $p$-power order and conductor given by a product of orderly primes (Corollary \ref{cor:measure}). We refer to $\nu_E$ as a \emph{horizontal $p$-adic $L$-function of $E$} (Definition \ref{def:padicL}). Pushing forward along the canonical projections $\Z/p^{m_n}\twoheadrightarrow \Z/p$, we obtain an element of the \emph{digit algebra} 
$$\Lambda^\mathrm{dig}:=\Z_p\llbracket (\Z/p)^\N\rrbracket,$$ 
which now encodes the twisted $L$-values $L(E,\chi,1)$ as above with $\chi$ either the trivial character or of order \emph{exactly} $p$. This way, we have translated the non-vanishing of $L$-values of order $p$ twists of elliptic curves to a question about character zeroes of elements of the digit algebra. Our main structural result is that for any non-zero element of the digit algebra the minimal valuation is attained for a ``positive proportion'' of characters. 
\begin{theorem}\label{thm:nonvanishingPSintro}
Let $\nu\in \Z_p\llbracket (\Z/p)^\N \rrbracket$ be a non-zero horizontal measure. Then there exists a finite set $M_\nu$ of continuous $\C_p^\times$-valued characters of $(\Z/p)^\N$ such that the following holds: for any continuous character $\chi:(\Z/p)^\N\rightarrow \C_p^\times$ there exists $\chi_0\in M_\nu$ such that
$$ \ev_p(\nu(\chi\chi_0))=\min_{\chi'}\ev_p(\nu(\chi')).$$  
\end{theorem}
For a  general version for measures on pro-$p$ abelian groups we refer to Theorem \ref{thm:nonvanishingPS}. This serves as a substitute for the Weierstrass preparation theorem in the case of vertical measures, see also Section \ref{sec:horweier}. The result is essentially sharp as can be seen by considering an embedding $\Z_p[\Z/p]\hookrightarrow \Lambda^\mathrm{dig}$ (Example \ref{ex:inf}). Furthermore, in Theorem \ref{thm:horWP} below it is shown that in the presence of the mod $p$ Kurihara Conjecture (\ref{eq:Kolderiv}) of rank $r$ it holds that
$$\min_{\chi}\ev_p(\nu_E(\chi))\leq \frac{r}{p-1},$$
and in particular that $\nu_E$ (as well as the pushforward to the digit algebra) is non-zero.  
More generally, it follows that  $\nu_E$ is non-zero if either  $L(E,1)\neq 0$ or  $p=2$, using a classical result of Friedberg--Hoffstein \cite{FriedbergHoffstein95}. Now for $d=p$ prime the quantitative bounds (\ref{eq:quantbound}) and (\ref{eq:kuriharaquant})  follow from Theorem \ref{thm:nonvanishingPSintro} since the number of characters of order $p$ with conductor given by a product of orderly primes is $\gg X/(\log X)^{1-\alpha}$ where $\alpha>0$ is the natural density of the orderly primes (Lemma \ref{lem:rprimefactors}). When $d$ is composite one  applies Theorem \ref{thm:nonvanishingPSintro} (and its generalization Theorem \ref{thm:nonvanishingpm} to higher prime powers) recursively to all the prime divisors of $d$, using the key fact that $p$ being $(E\otimes \chi)$-\emph{good} is automatic when the order of $\chi$ is prime to $p$ (Corollary \ref{cor:TWpropa}).        Finally, the simultaneous non-vanishing results are obtained by applying the same arguments to the product of the horizontal measures $\nu_{E_1},\ldots, \nu_{E_n}$ associated to elliptic curves $E_1,\ldots, E_n$ at joint orderly primes. 

The proof of Theorem \ref{thm:nonvanishingPSintro} proceeds by applying Fourier theory to study the  geometry of $X:=\Spec \Lambda^\mathrm{dig}$. This is a non-noetherian and non-reduced scheme over $\Z_p$. Our main observation is that despite these geometric defects there is a strong restriction on the possible zero sets of non-zero global sections (i.e. horizontal measures like the one associated to an elliptic curve above). Note here that the scheme $X$ has a point corresponding to each continuous character of $(\Z/p)^\N$ which all specialize to the unique maximal ideal (Corollary \ref{cor:locallity}). The main ingredients include a \emph{discrete maximum modulus principle} (Proposition \ref{prop:discrete}) which implies that although $X$ is non-noetherian the locus of characters of \emph{non-minimal} valuation for a fixed (non-zero) global section is actually noetherian. Now the set $M_\nu$ in Theorem \ref{thm:nonvanishingPSintro} is obtained as a maximal subset of this locus. The conjecture of Kurihara also interacts nicely with the geometric picture which we utilize in the proof of Theorem \ref{thm:Kolynonvan}: on the special fiber one can pick (non-canonical) coordinates and identify the Kato--Kolyvagin derivatives (\ref{eq:Kolderiv}) of $E$ modulo $p$ with the actual derivatives of the infinite variable  power series over $\F_p$ corresponding to $\nu_E$ under pullback to the special fiber, see Remark \ref{rem:dualnumbers}.  
\begin{remark}
    Extending the algebraic side of Iwasawa theory to the non-noetherian setting has recently seen progress, see \cite{BandiniBarsLonghi14},  \cite{BurnsDaoudLiang24}, \cite{PopescuYian24}. The results obtained in Section \ref{sec:horziontalmeasures} of the present paper can be seen as an analytic counterpart to these efforts.
\end{remark}
\begin{remark}Finally, we remark that our methods should adapt well to other situations where one has horizontal norm relations. In forthcoming work we obtain non-vanishing results analogous to those stated above for central $L$-values and derivatives of central $L$-values of elliptic curves $E/\Q$ base changed to quadratic fields $K$ twisted by ring class characters over $K$. \end{remark}
\subsection{Previous Results}\label{sec:previousresults}
In the case of quadratic twists there are many methods for obtaining non-vanishing: half-integral weight modular forms, multiple Dirichlet series and the approximate functional equation. 
The infinitude of non-vanishing quadratic twists is a classical result of Friedberg--Hoffstein \cite{FriedbergHoffstein95} obtained by studying multiple Dirichlet series constructed using metaplectic methods. An alternative approach using tools from analytic number theory was initiated in the work of Murty--Murty  
\cite{MurtyMurty91}, see also \cite{Young06}, \cite{SoYo}, \cite{RaSo}. The proportion of non-vanishing was improved by Ono--Skinner \cite{OnoSkinner98},\cite{OnoSkinner98.2} and further in the case where the elliptic curve satisfies $E[2](\Q)=0$ by  Ono \cite{Ono01} using the Shimura correspondence and half-integral weight modular forms (i.e. also using metaplectic methods). See also the survey \cite{BJKOS99}. In the special case of congruent number curves and certain other elliptic curves with complex multiplication, Goldfeld's Fifty--Fifty Conjecture has been resolved by combining the counting results of Smith \cite{Smith} and the $p$-converse theorem following from the Iwasawa Main Conjecture for ramified primes in imaginary quadratic fields as formulated and proved by the first named author \cite{Kriz}. 
Lastly, there is an approach using \emph{congruences} as in the works of Chao Li and the first named author \cite{Kriz16}, \cite{KrizLi}, \cite{KrizLi2} and in the work of Castella--Grossi--Lee--Skinner \cite{CastellaGrossiLeeSkinner}. See also the work of Zhai \cite{Zhai} and Cai--Li--Zhai \cite{CaiLiZhai} for results under assumptions on the $2$-part of the BSD-formula. To the authors' knowledge, the above approaches have so far not been made to work for higher order characters outside of the following cases:   

\begin{enumerate}
\item \emph{Metaplectic methods}:  Bucur--Chinta--Frechette--Hoffstein \cite{BACF04} and Blomer--Goldmakher--Louvel \cite{BLomerGoldmakherLouvel14} study a generalization of the multiple Dirichlet series techniques in \cite{FriedbergHoffstein95} for order $d$ twists of, reps. $\GL_2(\A_K)$ and $\GL_1(\A_K)$ automorphic forms where $K$ is a number field containing $\zeta_d$ (a primitive $d$th root of unity). The assumptions on $K$, however, excludes the cases of $\Q$ and imaginary quadratic fields, except $\Q(i), \Q(\sqrt{-3})$. See also \cite{BruFriedHoff05}.  Lieman \cite{Lieman94} and She \cite{She99} generalized the methods of Friedberg--Hoffstein for cubic twist families of a CM elliptic curve, i.e. families of elliptic curves over $\Q$ which are isomorphic over a cubic extension of $\Q$. Note that these \emph{twists} are different from the twists by cubic characters we are considering in this paper. 
\item \emph{Approximate functional equation approach}: For higher order characters of elliptic curves calculating a first moment using approximate functional equation seems out of reach with current technology. In the simpler setting of Dirichlet $L$-function, Baier--Young \cite{BaierYoung} calculated  the first moment for Dirichlet characters of orders three, four and six. To the authors' knowledge there are no results beyond this, and a general obstruction seems to be the appearance of higher order Gau{\ss} sums. Another relevant work is that of Chinta \cite{Chinta02} which combines rationality of modular symbols (which implies that the Galois action preserves non-vanishing) with mollification to obtain non-vanishing for \emph{all} twists by order $d$ characters of prime conductor $q$ whenever $d\gg q^{7/8+\eps}$.  
\item \emph{Congruences}: As mentioned above, Fearnley--Kisilevsky--Kuwata \cite{FeKiKu12} obtained strong quantitative non-vanishing results for twists by order $p$ characters, with $p$ prime, in the case where $L(E,1)\neq 0$ and $\ev_p(L(E,1)/\Omega_E^+)=0$. Besides being restricted to curves of analytic rank $0$, the method applied in \cite{FeKiKu12} crucially depends on the initial algebraically normalized central $L$-value being a $p$-adic unit. Thus an iterated version of their argument does  not apply to twists of order $d$ when $d$ is not a prime power.
\end{enumerate}
 
For related results over function fields, we refer to \cite{CDLL22} and the references therein.

\subsubsection{Notation}
Throughout, let $\N = \Z_{> 0}$ denote the natural numbers. Let $\C_p=\widehat{\overline{\Q_p}}$ be the $p$-adic complex numbers and denote by $\ev_p:\C_p\rightarrow \Q\cup\{\infty\}$ the unique $p$-adic valuation such that $\ev_p(p)=1$ and by $\Oo_{\C_p}=\{x\in \C_p: \ev_p(x)\geq 0\}$ the valuation ring of $\C_p$. Given a finite abelian group $G$ and a commutative ring $R$, we denote the associated group ring and its elements by 
$$R[G]=\left\{\sum_{g\in G} r_g[g]: r_g\in R \right\}.$$
\noindent\textbf{Acknowledgements} The authors heartily thank Farrell Brumley, Christophe Cornut, Daniel Disegni, Marius Fischer, Giada Grossi and Akshay Venkatesh for helpful discussions and comments while preparing this work. The authors are especially grateful to Christophe Cornut for a careful reading of an earlier version of this text and for insightful suggestions that improved the presentation of the article.
\section{Horizontal {\it p}-adic Measures}\label{sec:horziontalmeasures}
In this section we develop a theory of  $p$-adic measures on pro-$p$ abelian groups with an emphasis on the \emph{horizontal} case. The goal is to understand the character zeroes of such measures and more general properties of the Fourier transforms. The archetypal example is the following space of measures that we will refer to as the \emph{digit algebra}: 
$$\Lambda^\mathrm{dig}:=\Z_p\llbracket (\Z/p)^\N\rrbracket=\varprojlim_{n\in \N} \Z_p[(\Z/p)^n],$$
where $(\Z/p)^\N:=\prod_{n\in \N}\Z/p$ denotes the pro-$p$ group given by the product of an infinite number of cyclic $p$-groups and the transition map for $m\leq n$ is the projection onto the first $m$ factors. When the horizontal measures come from arithmetic (e.g. induced by modular symbols or Heegner points), information about the character zeroes of such measures will give  information about the non-vanishing of periods (e.g. $L$-values and derivatives of $L$-values) in twist families, see Section \ref{sec:applicationstomodularforms}.  

We will start by recalling the case of the vertical Iwasawa algebra, as we will try to approximate this case in what follows. For a comprehensive account, we refer to \cite{WashingtonCyclotomicFields}, \cite{deShalit},  \cite{JacintoWilliamsPadicLfunctions}. By the Amice transform we can identify the Iwasawa algebra with a power series algebra: 
\begin{equation}\label{eq:usual}\Lambda:=\Z_p\llbracket \Z_p\rrbracket\cong \Z_p\llbracket T\rrbracket,\quad \nu\mapsto f_\nu (T),\end{equation} 
such that for any continuous (additive) character $\psi:\Z_p\rightarrow \C_p^\times$ it holds that
$$\nu(\psi)=f_\nu(\psi(1)-1).$$
For $\nu$ non-zero the Weierstrass preparation theorem now gives a unique representation
\begin{equation}\label{eq:f(T)WP} f_\nu(T)=p^\mu \left(T^\lambda+\sum_{0\leq n<\lambda} b_n T^n\right)u(T),\end{equation}
where $\mu\in \Z_{\geq 0}$, $\lambda\in \Z_{\geq 0}$, $p|b_n$ for $0\leq n<\lambda$, and $u(T)\in \Z_p\llbracket T\rrbracket^\times$ is an invertible power series (i.e. $u(0)\in \Z_p^\times$).  In particular, any  non-zero vertical measure has finitely many character zeroes. Furthermore, this
 implies that for $\nu$ non-zero there exists $N_\nu\geq 1$ such that for a character  $\psi:\Z_p\rightarrow \C_p^\times$ of order $p^m$ with $m\geq N_\nu$ we have
\begin{equation}\label{eq:WPformula}\ev_p(\nu(\psi))=\mu+\frac{\lambda}{p^m-p^{m-1}}.\end{equation}
Applying this to the (odd) Kubota--Leopold $p$-adic $L$-function at $s=0$, one obtains (one half of) the Iwasawa class number formula \cite[Section 7.3]{WashingtonCyclotomicFields}.  

For a horizontal measure 
$$\nu \in \Z_p \llbracket (\Z/p)^\N\rrbracket,$$ 
the story is a bit more complicated: Proposition \ref{prop:horamice} does yield a horizontal analogue of the Amice transform and we define horizontal versions of the $\mu$- and $\lambda$-invariants in Definition \ref{def:lambdamu}. But there does however not seem to be a direct replacement for Weierstrass preparation. In particular, there exist measures with \emph{infinitely} many character zeroes. 
\begin{example}[Infinitely many zeroes]\label{ex:inf} Let $(\Z/p)^N\hookrightarrow(\Z/p)^\N $ be an embedding of groups and consider the following element:

\begin{equation}\label{eq:diagonal} \nu:=\sum_{x\in (\Z/p)^N}[x]\in \Z_p[(\Z/p)^N]\subset \Z_p \llbracket (\Z/p)^\N\rrbracket, \end{equation}
Notice that $\nu$ is a zero-divisor and for a continuous character $\chi: (\Z/p)^\N\rightarrow \C_p^\times$ we have, by character orthogonality, that $\nu(\chi)\neq 0$ if and only if the restriction of $\chi$ to the image of $(\Z/p)^N$ is trivial. So if we order the characters according to the minimal $n$ for which they factor through the projection $(\Z/p)^\N \twoheadrightarrow (\Z/p)^n$ onto the first $n$ factors of the product, then the proportion of non-vanishing characters for $\nu$ is given by $1/p^N$. 
\end{example}

We conclude that the best one can hope for is that a ``positive proportion'' of characters specializations are non-zero. Our main result is that, for any horizontal measure, the \emph{minimal valuation} is attained for a ``positive proportion'' of characters, see Theorem  \ref{thm:nonvanishingPS} and Corollary \ref{cor:nonvanishingPS} below. Furthermore, under some assumptions, Theorem \ref{thm:horWP} gives a formula for the minimal valuation in terms of the $\mu$- and $\lambda$-invariants introduced in Definition \ref{def:lambdamu}, which can be seen as a horizontal replacement of the formula (\ref{eq:WPformula}). Finally, for measures vanishing at the trivial character the minimal valuation can be directly linked to the \emph{augmentation rank}, see Proposition \ref{prop:augmentation}.   
\subsection{Horizontal Iwasawa Algebras}\label{sec:horizontal case} 
Let $(G,+)$ be a profinite abelian group and $R$ a commutative ring. Then we define the associated \emph{Iwasawa algebra} as:
\begin{equation}\label{eq:Iwasawa}
R\llbracket G\rrbracket:= \varprojlim_{H\leq G\text{ open}} R[G/H],
\end{equation}  
where the inverse limit is taken with respect to the partially ordered set of open subgroups of $G$ with respect to inclusion. We give $R\llbracket G\rrbracket$ the inverse limit topology.\footnote{However, the topology on $R\llbracket G\rrbracket$ will play no role in our arguments.} When $R\subset \Oo_{\C_p}$ is a $p$-adically complete subring we can identify (\ref{eq:Iwasawa}) with the space of $R$-valued measures on $G$:
\begin{equation}
 \label{eq:homcts}   R\llbracket G\rrbracket\cong \mathrm{Hom}_\mathrm{cts}(\mathcal{C}(G,R),R),
\end{equation} 
where $\mathcal{C}(G,R)$ denotes the space of continuous function $G\rightarrow R$ equipped with the sup-norm topology (see e.g.\ \cite[Proposition 2.3]{JacintoWilliams25}). The identification (\ref{eq:homcts}) is an $R$-algebra isomorphism with the product on the right-hand side given by convolution, see \cite[Remark 2.5]{JacintoWilliams25}).  In fact, given $\nu\in R\llbracket G\rrbracket $ and a continuous function $\varphi\in \mathcal{C}(G,\Oo_{\C_p})$ we can define the integral of $\varphi$ with respect to $\nu$ which we will denote by
\begin{equation}\label{eq:integralwrtnu} \nu(\varphi):= \int_G \varphi(x) \, d\nu(x).  \end{equation} 
Taking $G=\Z_p$ and $R=\Z_p$ yields the usual vertical Iwasawa algebra in equation (\ref{eq:usual}). The goal of this section is to explore the structure of  general \emph{pro-$p$ Iwasawa algebras}, i.e.\ when $G$ is a pro-$p$ abelian group. Recall that this means that $G$ is a profinite group such that for any open subgroup $H$ the quotient $G/H$ is a finite abelian $p$-group for some (fixed) prime number $p$.  For the applications in this paper, we will put special emphasize on the ``horizontal'' case. To define this, let $(G_n)_{n\in \N}$ be a sequence of non-trivial finite abelian groups. The product group $\prod_{n \in \N}G_n$ has a canonical presentation
$$G_{\N} := \varprojlim_{\substack{A \subset \N\\\text{ finite}}}\prod_{n \in A}G_n$$
where the transition maps for $A' \subset A\subseteq \N$ are given by the canonical projections
\begin{equation}\label{eq:natrest}\pi_{A,A'} : \prod_{n \in A}G_n\twoheadrightarrow \prod_{n \in A'}G_n.\end{equation}
For each finite set $A\subset \N$, we give $\prod_{n \in A}G_n$ the discrete topology. Then $G_{\N}$ has the natural associated product topology, which is the coarsest topology such that each projection $G_{\N} \rightarrow \prod_{n\in A}G_n$, with $A\subset \N$ finite, is continuous. This realizes $G_{\N}$ as a profinite group. We then define 
\begin{equation}\label{eq:defhoralg}\Lambda^{\mathrm{hor}} :=R\llbracket G_\N \rrbracket= \varprojlim_{\substack{A \subset \N\\\text{finite}}}R\left[\prod_{n \in A}G_n\right],\end{equation}
suppressing the dependence on $R$ and $G_\N$ in the notation $\Lambda^{\mathrm{hor}}$. We refer to $G_\N$ as a \emph{horizontal profinite group} and call $R\llbracket G_\N \rrbracket$ the \emph{horizontal Iwasawa algebra (associated to $G_\N$ and $R$)} which we say is \emph{pro-$p$} if $G$ is a pro-$p$ abelian group and $R$ is a $p$-adically complete subring of $\Oo_{\C_p}$. We refer to the elements of $R\llbracket G_\N \rrbracket$ as \emph{horizontal measures}.

As rings, the horizontal Iwasawa algebras are quite badly behaved compared to the vertical Iwasawa algebra $\Z_p\llbracket \Z_p\rrbracket$ as the next lemma shows.  
\begin{lemma}\label{lem:non-int}
A horizontal Iwasawa algebra is not an integral domain and is non-noetherian.
\end{lemma}
\begin{proof}
To construct zero-divisors, let $n\in \N$ and consider $x\in G_n$ non-trivial of order $m\geq 2$. Then we have 
$$([x]-1)([(m-1)x]+\ldots+[x]+1)=0\in R[G_n]\subset R\llbracket G_\N \rrbracket,$$
and thus $[x]-1$ is a non-zero zero divisor, where $R[G_n]$ is considered as an $R$-submodule of $R\llbracket G_\N \rrbracket$ via the canonical embedding $G_n\hookrightarrow G_\N$. 

To construct an infinite strictly increasing sequence of ideals, for each $n\in \N$ let $x_n\in G_n$ be a non-trivial element. Then we get an increasing chain of ideals $I_1\subset I_2\subset \ldots$ where
$$I_n:= ([x_1]-1,\ldots, [x_n]-1).$$
We want to argue that $I_n\neq I_{n+1}$. To see this, we pick a continuous character $\chi\in \widehat{G}$ such that $\chi(x_i)=1$ for all $1\leq i\leq n$ but $\chi(x_{n+1})\neq 1$ (e.g.\ a character that factors through the first $n+1$ factors, trivial on the first $n$ factors and non-trivial on $x_{n+1}$). Then clearly $\nu(\chi)=0$ for all $\nu\in I_n$, but $([x_{n+1}]-1)(\chi)\neq 0$. This shows that $R\llbracket G_\N\rrbracket$ is indeed non-noetherian.
\end{proof}

In the next section we will explore some further ring theoretic properties of pro-$p$ Iwasawa algebras. For now we will give a general method for constructing horizontal measures using ``norm relations'' which is key in the construction of our key object of study, \emph{horizontal $p$-adic $L$-functions}, in Section \ref{sec:applicationstomodularforms}. 
\subsubsection{Constructing measures from norm relations}\label{sec:horconstr}
Suppose that for each finite subset $A \subset \N$ we are given an element 
$$\theta_A \in R\left[\prod_{n\in A}G_n\right],$$
 such that we have the following ``norm relation property'': there exists a sequence of horizontal measures $(\alpha_n)_{n\in \N}\subset R\llbracket G_\N \rrbracket$ such that for $A'\subset A \subset \N$ with $A$ finite we have that 
$$\pi_{A,A'}(\theta_A)=\left(\prod_{n\in A\setminus A'} \pi_{\N,A'}(\alpha_n)\right) \theta_{A'},$$
where $\pi_{A,A'}$ are the canonical projections as in (\ref{eq:natrest}).  Suppose that $\alpha_n$ is invertible for every $n \in \N $ (i.e. $\alpha_n \in R\llbracket G_\N \rrbracket^\times$). Then we obtain elements
$$\nu_A := \frac{\theta_A}{\prod_{n\in A} \pi_{\N,A}(\alpha_n)} \in R\left[\prod_{n\in A}G_n\right]$$
which are compatible under the canonical projections $\pi_{A,A'}$ as in (\ref{eq:natrest}). To see this compatibility, note that $\pi_{A,A'}\circ \pi_{\N,A}=\pi_{\N,A'}$ for any $A'\subset A \subset \N$. Hence
\begin{align*}\pi_{A,A'}(\nu_A) = \pi_{A,A'}\left(\frac{\theta_A}{\prod_{n\in A}\pi_{\N,A}(\alpha_n)}\right) = \frac{\pi_{A,A'}(\theta_A)}{\prod_{n\in A}\pi_{A,A'}(\pi_{\N,A}(\alpha_n))} &= \frac{\left(\prod_{n\in A\setminus A'}\pi_{\N,A'}(\alpha_n) \right) \cdot \theta_{A'}}{\prod_{n\in A}\pi_{\N,A'}(\alpha_n)}\\
&= \frac{\theta_{A'}}{\prod_{n\in A'}\pi_{\N,A'}(\alpha_n)} = \nu_{A'}.
\end{align*}
Thus we get an element 
$$\nu := \varprojlim_{\substack{A\subset \N\\ \mathrm{finite}}} \nu_A \in R\llbracket G_\N \rrbracket.$$
This construction will be key in obtaining horizontal measures in arithmetic settings, see Section \ref{sec:normrelQ}.
 
\subsection{Ring Theoretic Properties of Pro-{\it p} Iwasawa Algebras}
Let $p$ be a prime number. In this section, we record some general properties of pro-$p$ Iwasawa algebras. In particular, we  prove locality of pro-$p$ Iwasawa algebras and define a notion of ``twisting''.  

First of all we introduce some general notation: For $G$ finite, $\nu\in R[G]$ and $x\in G$ we put $\nu_x:=\nu(\delta_x)$ where $\delta_x:G\rightarrow \{0,1\}$ denotes the delta function at $x\in G$. In other words, we have 
$$\nu=\sum_{x\in G}\nu_x\, [x] .$$
We will denote the trivial character of $G$ by 
$$\1:G\rightarrow R,\quad  x\mapsto 1.$$ 

We start by giving a characterization of the invertible elements in pro-$p$ Iwasawa algebras.
\begin{proposition}\label{prop:invertible}Let $G$ be a pro-$p$ abelian group and $(R,\mathfrak{m}_R)$ a local ring with residue field $R/\mathfrak{m}_R$ of characteristic $p$. Then $\nu\in R\llbracket G\rrbracket^\times$ if and only if $\nu(\1)\notin \mathfrak{m}_R$.
\end{proposition}
\begin{proof}
If $\nu$ has an inverse $\nu^{-1}\in R\llbracket G\rrbracket$, then we get
$$\nu(\1)\nu^{-1}(\1)=1\in R.$$
This shows one implication. On the other hand, if for each open subgroup $H$ the projection $\nu_H\in R[G/H]$ is invertible (in $R[G/H]$), then the inverses $(\nu_H^{-1})_{H\leq G \text{ open}}$ define a compatible system since inverses are unique, which implies that $\nu$ is invertible. Thus we may reduce to the case where $G$ is a finite abelian $p$-group. We claim that any maximal ideal $\mathfrak{m}$ of $R[G]$ satisfies $\mathfrak{m}\cap R=\mathfrak{m}_R$: The containment $\subset$ is clear since $\mathfrak{m}\cap R^\times=\emptyset$ and $R$ is local. The other containment follows by maximality of $\mathfrak{m}$, since $\mathfrak{m}+(\mathfrak{m}_R)$ is a proper ideal of $R[G]$ as it does not contain $R^\times$ (as can be seen by looking at the intersection with $R$). This shows that we have an embedding $R/\mathfrak{m}_R\hookrightarrow R[G]/\mathfrak{m}$ and thus $F:=R[G]/\mathfrak{m}$ is a finite field of characteristic $p$ with an $R$-module structure. By the standard adjunction for group rings we have:
$$\Hom_{R\mathrm{-alg}}(R[G],F)\cong \Hom_{\mathrm{grp}}(G,F^\times)=\{\1\},$$
since $G$ is a $p$-group and $F^\times$ is a group of order coprime to $p$. This implies that $\mathfrak{m}$ contains the ideal generated by $\mathfrak{m}_R$ and the augmentation ideal of $R[G]$, which itself is a maximal ideal. We conclude that $(R[G],\mathfrak{m})$ is a local ring and $\nu\in \mathfrak{m}$ if and only if $\nu(\1)\in \mathfrak{m}_R$ which implies the wanted conclusion.
\end{proof}

\begin{corollary}\label{cor:locallity}
Let $G$ be a pro-$p$ abelian group  and $(R,\mathfrak{m}_R)$ a local ring with residual characteristic $p$. Then $R\llbracket G \rrbracket$ is a local ring with unique maximal ideal give by $I_\mathrm{aug}+(\mathfrak{m}_R)$, where $I_\mathrm{aug}=\{\nu\in R\llbracket G \rrbracket: \nu(\1)=0 \}$ denotes the augmentation ideal  of $R\llbracket G \rrbracket$.
\end{corollary}
\begin{proof}
By the previous proposition we see that non-units are closed under addition. This shows that $R\llbracket G \rrbracket$ has a unique maximal ideal, namely the non-units $I_\mathrm{aug}+(\mathfrak{m}_R)$.
\end{proof}

\subsubsection{Twisting of measures}
Let $R$ be a subring of $\C_p$ and $G$ a profinite abelian group. We want to make sense of the twist of an $R$-valued measure on $G$ by a finite order character.

\begin{lemma}\label{lem:twisting}
Let $\nu \in R\llbracket G\rrbracket$. Let $\chi_0:G\rightarrow \C_p^\times$   be a continuous finite-order character and denote by $R(\chi_0)$ the smallest ring containing $R$ and the image of $\chi_0$. Then there there exists a unique measure $\nu^*\in R(\chi_0)\llbracket G\rrbracket$ such that  
$$\nu^*(\chi)=\nu(\chi\chi_0),$$
for any continuous character $\chi: G\rightarrow \C_p^\times$.
\end{lemma}
\begin{proof}
Note firstly that $\mathrm{ker}(\chi_0)$ is an open subgroup of $G$ since $\chi_0$ is of finite order. For existence we consider
\begin{equation}
  \nu^*:=\varprojlim_{\substack{H\leq G\text{ open},\\H\subset \ker(\chi_0)}}\, \sum_{x\in G/H} \chi_0(x) \nu_x\, [x] \in R(\chi_0)\llbracket G\rrbracket,
\end{equation}
which is well-defined since the open subgroups of $G$ contained in any fixed open subgroup are cofinal (among the open subgroups of $G$) and $\chi_0$ is well-defined on $G/H$ since $H\subset \ker(\chi_0)$. The interpolation property now follows by approximating $\chi$ by locally constant functions and using the continuity of the measures $\nu$ and $\nu^*$. Uniqueness follows since any measure on $G$ is uniquely determined by its values on finite order characters by Fourier inversion.
\end{proof}

 We denote the measure in Lemma \ref{lem:twisting} by $\nu_{\chi_0}$ and refer to it as the \emph{twist of $\nu$ by $\chi_0$}.

\subsection{Fourier Theory of Pro-{\it p} Abelian Groups}
Throughout this section $(G,+)$ will denote a pro-$p$ abelian group (possibly finite) and $R\subset \Oo_{\C_p}$ will be a $p$-adically complete subring. For a set $M$ we will write $|M|:=\#M$ for its cardinality. We will denote by 
$$ \widehat{G}:=\Hom_\mathrm{cts}(G,\C_p^\times),$$
the multiplicative group of $\C_p^\times$-valued continuous characters of $G$. By integration as in (\ref{eq:integralwrtnu}) we can associate to $\nu\in \Oo_{\C_p}\llbracket G \rrbracket$ a map:
$$ \widehat{G} \rightarrow \Oo_{\C_p},\quad \chi\mapsto \nu(\chi), $$
which we refer to as the \emph{Fourier transform of $\nu$}. 
We will be most interested in the case where $G$ is a horizontal pro-$p$ abelian group, meaning  that there exists integers $m_n\geq 1, n\in \N$ and an isomorphism $G\cong \prod_{n\in \N} \Z/p^{m_n}$ of profinite groups. In this case the group  of characters admits the following simple description. 
\begin{lemma}\label{lem:finiteorder}Let $G$ be a horizontal pro-$p$ abelian group. Then any continuous character $\chi:G\rightarrow \C_p^\times$ factors through the quotient by an open subgroup of $ G$, i.e.\ it holds that $$\widehat{G}=\varinjlim_{H\leq G\text{ \rm open}} \widehat{G/H}.$$In particular, all characters have finite order.
\end{lemma}
\begin{proof}
Denote by $G_\mathrm{tor}\leq G$ the subgroup of elements of finite order. Since $G$ is a product of finite groups   the subgroup $G_\mathrm{tor}$ is dense in $G$. For any $\chi\in \widehat{G}$ we clearly have $\chi(G_\mathrm{tor})\subset \mu_{p^\infty}$, where $\mu_{p^\infty}\subset \C_p^\times$ 
denotes the subgroup of $p$-power roots of unity. So by continuity $\chi(G)\subset \overline{\mu_{p^\infty}}=\mu_{p^\infty}$ using that $\mu_{p^\infty}$ is discrete in the $p$-adic topology. In particular, the kernel of $\chi$ is open and so since any open subgroup of a profinite group is of finite index, we get the wanted conclusion. 
\end{proof}
Given a subgroup $M\leq \widehat{G}$ we put
\begin{equation}\check{M}:=G/\left(\bigcap_{\chi\in M}\ker(\chi)\right)=G/\langle x\in G: \chi(x)=1,\forall \chi\in M\rangle.\end{equation}
We will need the following basic duality. 
\begin{lemma}\label{lem:check}
    Let $G$ be a profinite group and let $M\leq \widehat{G}$ be a finite subgroup. Then the joint kernel $\bigcap_{\chi\in M}\ker(\chi)$ is an open subgroup of $G$ and there is a natural isomorphism $M\cong \widehat{\check{M}}$.
\end{lemma}
\begin{proof}
Clearly $\chi\in M$ has finite order which implies that $\ker(\chi)\leq G$ is open. By finiteness of $M$ the intersection is open as well. In particular, we conclude that $\check{M}$ is finite. 

To prove the isomorphism we start by observing that there is a natural injection $M\hookrightarrow \widehat{\check{M}},$ 
since all characters $\chi\in M$ factor through $\check{M}$. On the other hand since $M$ is finite abelian we can find characters $\chi_1,\ldots, \chi_n\in M$ of order $d_1,\ldots ,d_n$ such that the natural map $\prod_{i=1}^n \langle \chi_i\rangle\rightarrow M$ is an isomorphism. This means that $\bigcap_{\chi\in M}\ker(\chi)=\bigcap_{i=1}^n\ker(\chi_i)$ and so we obtain an isomorphism $\check{M}\cong \prod_{i=1}^n G/\ker(\chi_i)$. Using that $|G/\ker(\chi_i)|=d_i$, $|M|=d_1\cdots d_n$,  and that $\check{M}$ is finite, we conclude that
$ |\widehat{\check{M}}|=|\check{M}|= |M|$
which implies that the injection $M\hookrightarrow \widehat{\check{M}}$ is an isomorphism.
\end{proof}
We denote by 
$$e(G):=\min \left(\{e\in \Z_{\geq 1}: p^e x=0,\forall x\in G\}\cup\{\infty\}\right),$$
the $p$-exponent of a group $G$ (which might be infinite).

Finally, for a pro-$p$ measure $\nu\in R\llbracket G\rrbracket$ we denote the \emph{valuation of $\nu$} by
\begin{equation}\label{eq:valnu}\ev_p(\nu):=\inf_{\chi\in \widehat{G}_\mathrm{tor}} \ev_p(\nu(\chi)),\end{equation}
where $\widehat{G}_\mathrm{tor}\subset \widehat{G}$ denotes the subgroup of finite order (continuous) characters of $G$. Note that in the case of the vertical Iwasawa algebra ($R=\Z_p$ and $G=\Z_p$) the valuation $v_p(\nu)$ is equal to the $\mu$-invariant of $\nu$ by Weierstrass preparation  (\ref{eq:WPformula}). 
\subsubsection{Maximum modulus principle for finite group rings}
In this section we will prove a key Fourier theoretic lemma saying that under certain ``largeness'' and ``integrality'' assumptions the $p$-adic maximum of the Fourier transform is attained at a non-trivial character. Thinking of the non-trivial characters as lying on the perimeter of a circle with the trivial character at the center, the result becomes an incarnation of the maximum modulus principle.  
\begin{proposition}[Discrete maximum modulus principle]\label{prop:discrete}
Let $p$ be a prime. Let $(G,+)$ be a finite abelian group. 
Let $\nu\in \C_p[G]$ be a $\C_p$-valued measure on $G$ satisfying 
\begin{equation}\label{eq:keylemma}\ev_p(\nu(\mathbf{1}))<\ev_p(|G|)+\min_{x\in G}\ev_p(\nu_x)\end{equation}
Then there exists $\chi\in \widehat{G},\chi\neq \mathbf{1}$ such that 
$$\ev_p(\nu(\chi))\leq \ev_p(\nu(\mathbf{1})).$$ 

\end{proposition}
\begin{proof}
By orthogonality of characters we have for each $x\in G$:
\begin{align}\label{eq:Fourierinversion2}\sum_{\chi\in \widehat{G}}\nu(\chi)\overline{\chi}(x)=\sum_{y\in G} \nu_y \sum_{\chi\in \widehat{G}}\chi(y-x)=   |G|\nu_x.\end{align}
Assume for the sake of contradiction that 
$$\ev_p(\nu(\chi))>\ev_p(\nu(\mathbf{1})),\quad \forall \chi\in \widehat{G},\chi\neq \mathbf{1}.$$
Then we conclude by the strong triangle inequality that the $p$-adic valuation of the left-hand side of (\ref{eq:Fourierinversion2}) is equal to $\ev_p(\nu(\mathbf{1}))$. But the $p$-adic valuation of the right-hand side of (\ref{eq:Fourierinversion2}) is equal to $\ev_p(|G|)+\ev_p(\nu_x)$ which contradicts the assumption (\ref{eq:keylemma}). This finishes the proof. 
\end{proof}

\subsubsection{Structure theorem for character zeroes}\label{sec:positiveprop}
In the previous section we saw that for integral measures on finite abelian $p$-groups there will be non-trivial characters with minimal valuation whenever the group is sufficiently large. We would like to apply this local statement to obtain information about the global structure of $p$-adic measures on pro-$p$ abelian groups. As we saw above in Lemma \ref{lem:non-int}, the rings $R\llbracket G\rrbracket$ might be non-noetherian and thus the geometry of $\Spec R\llbracket G\rrbracket$ is hard to get a grip on. The key insight is that Proposition \ref{prop:discrete} however implies that the locus of characters with non-minimal valuation \emph{is} noetherian. This yields a surprisingly rigid structure on the Fourier transform of such $p$-adic measures.

\begin{theorem}[Structure theorem] \label{thm:nonvanishingPS}
Let $G$ be a pro-$p$ abelian  group and put $R=\Oo_{\C_p}$. Let $\nu\in R\llbracket G\rrbracket$ be a non-zero measure. Then for every $\eps>0$ there exists a finite subgroup $M_{\nu,\eps}\leq \widehat{G}$  such that the following holds: for any non-trivial finite subgroup $M\leq \widehat{G}$ there exists $\chi\in M,\chi\neq \1$ and $\chi_0\in M_{\nu,\eps}$ such that
\begin{equation}
\ev_p(\nu( \chi\chi_0))\leq \ev_p(\nu)+\eps,
\end{equation}  
where $\ev_p(\nu)$ is defined as in equation (\ref{eq:valnu}). Furthermore, if $\ev_p(\nu(\1))\leq v_p(\nu)+\eps$ then we can pick $M_{\nu,\eps}$ such that 
\begin{equation}\label{eq:Mbound}|M_{\nu,\eps}|\leq p^{\lfloor \ev_p(\nu(\1)) \rfloor}.\end{equation}
\end{theorem}
\begin{proof}
Let $\eps>0$ and consider firstly the case where $\ev_p(\nu(\1))\leq v_p(\nu)+\eps$. Let $\mathcal{X}$ denote the poset with respect to inclusion consisting of all finite subgroups $M\leq \widehat{G}$ with the property that 
\begin{equation}\label{eq:X} \ev_p(\nu(\chi))> \ev_p(\nu(\1)),\quad \forall \chi\in M,\chi\neq \1.\end{equation}
Note that $ \mathcal{X}$ is non-empty as the condition is vacuous for the trivial subgroup $M=\{\1\}$.  By Lemma \ref{lem:check}  we obtain for any finite subgroup $M\subset \widehat{G}$ a map $R\llbracket G\rrbracket\rightarrow R[\check{M}]$ and it holds that $|\check{M}|=|M|$. Applying Proposition \ref{prop:discrete} to the pushforward of $\nu$ along this map we conclude that for any finite subgroup $M\leq \widehat{G}$ satisfying 
$$\ev_p(|M|)=\ev_p(|\check{M}|)> \ev_p(\nu(\1)),$$ 
there exists some non-trivial character $\chi\in M$ such that $\ev_p(\nu(\chi))\leq \ev_p(\nu(\1))$. By applying the floor function we conclude that for all  $M\in \mathcal{X}$ it holds that 
\begin{equation}\label{eq:calXbound}\ev_p(|M|)\leq \lfloor \ev_p(\nu(\1))\rfloor.\end{equation}
Now let $M_1\leq M_2\leq \ldots$ be an ascending sequence of elements of $\mathcal{X}$. If the sequence does not stabilize then since $G$ is pro-$p$ we have that $\ev_p(|M_n|)\rightarrow \infty$ as $n\rightarrow \infty$ which contradicts (\ref{eq:calXbound}).  Thus we conclude that $\mathcal{X}$ is a noetherian poset. 

Since $\mathcal{X}$ is a non-empty noetherian poset there exists a maximal element $M_0\in \mathcal{X}$ and by the bound (\ref{eq:calXbound}) we have  $\ev_p(|M_0|)\leq \lfloor \ev_p(\nu(\1))\rfloor$. We claim that we can pick $M_{\nu,\eps}=M_0$ in the theorem. To see this, let $M\leq \widehat{G}$ be any finite subgroup. If  $M\subset M_0$ the conclusion is automatic by taking $\chi_0=\chi^{-1}$. So we may reduce to the case where $M_0\subsetneq M\cdot M_0$, in which case we conclude by maximality of $M_0$ that $M\cdot M_0\notin \mathcal{X}$. This means that there exists a non-trivial character $\chi'\in M\cdot M_0$ such that
 $$ \ev_p(\nu(\chi')) \leq \ev_p(\nu(\1)).$$
Note that since $M_0\in \mathcal{X}$ we must have  that $\chi'\notin M_0 $ (by the defining property (\ref{eq:X}) of $\mathcal{X}$). Thus writing $\chi'=\chi \chi_0$ with $\chi\in M$ and $\chi_0\in M_0$,  we conclude that $\chi\neq \1$. This finishes the proof in this case.
 
In the general case, let $\chi_0\in \widehat{G}$ be a finite order character such that $\ev_p(\nu(\chi_0))\leq \ev_p(\nu)+\eps$ (which exists by the definition (\ref{eq:valnu}) of the valuation $\ev_p(\nu)$). Consider the measure $\nu_{\chi_0}\in R(\chi_0)\llbracket G\rrbracket$ obtained by twisting $\nu$ by $\chi_0$ as in Lemma \ref{lem:twisting}. Then by the above we obtain a finite subgroup $M_{\nu_{\chi_0},\eps}\leq G$ with the desired property for $\nu_{\chi_0}$ and thus we can take $M_{\nu,\eps}=\langle \chi_0\rangle \cdot M_{\nu_{\chi_0},\eps}$ to finish the proof.     
\end{proof}
In the case where $G=(\Z/p)^\N$ and $R$ satisfies some mild assumption, e.g.\ for the \emph{digit algebra} $\Lambda^\mathrm{dig}=\Z_p\llbracket(\Z/p)^\N\rrbracket$ emphasized in the introduction, we obtain the following simplified statement.
\begin{corollary}\label{cor:nonvanishingPS}
Let $R$ be a $p$-adically complete subring of $\Oo_{\C_p}$ such that the $p$-adic valuation on $R$ is discrete. Suppose $R[1/p] \cap \Q_p(\mu_p) = \Q_p$. Let $\nu \in R\llbracket (\Z/p)^\N\rrbracket$ be a non-zero horizontal measure. Then the minimal valuation is attained for a {\lq\lq}positive proportion{\rq\rq} of characters, i.e. there exists a finite set $M_\nu$ of continuous characters of $(\Z/p)^\N$ such that the following holds: for any continuous character $\chi:(\Z/p)^\N\rightarrow \C_p^\times$ there exists $\chi_0\in M_\nu$ such that
$$ \ev_p(\nu(\chi\chi_0))=\ev_p(\nu).$$
\end{corollary}
\begin{proof}
We note that for  any character $\chi$ of $(\Z/p)^\N$ we have that $\ev_p(\nu(\chi))\in \frac{1}{p-1}v_p(R)$ which is discrete inside $\Q\cup\{\infty  
\}$ by the assumption on $R$, and so we can take $\eps=0$ in Theorem \ref{thm:nonvanishingPS}. 
By our condition  on $R$ we have a canonical identification $\mathrm{Gal}(R[1/p](\mu_p)/R[1/p]) = \mathrm{Gal}(\Q_p(\mu_p)/\Q_p)$ given by restriction. Since the coefficients of $\nu$ are in $R$, we conclude the following Galois equivariance:
\begin{equation}\label{eq:galoisinv}\nu(\chi^\sigma)=(\nu(\chi))^\sigma,\quad \sigma\in \Gal(\Q_p(\mu_p)/\Q_p),\end{equation}
showing in particular that the $p$-adic valuation is constant on Galois orbits of characters. Thus by applying the above theorem to the subgroup $M=\langle \chi \rangle=\{\1,\chi,\ldots,\chi^{p-1}\}$ and possibly applying an element of the Galois group, we obtain the desired statement.
\end{proof}
Note that the assumptions of Corollary \ref{cor:nonvanishingPS} are satisfied for $R = \Z_p$ and so Theorem \ref{thm:nonvanishingPSintro} from the introduction follows as a special case.
\begin{remark}
Corollary \ref{cor:nonvanishingPS} holds with $M_\nu=\{\1\}$ exactly if $\ev_p(\nu(\chi))$ is constant as a function on $\widehat{G}$. Note that by Proposition \ref{prop:invertible} a sufficient condition for this to be true is that  $\ev_p(\nu(\mathbf{1}))=0$ as this means that $\nu$ is invertible. In this case, this can also be seen directly using the congruence 
$$\ev_p(\nu(\mathbf{1}))\equiv \ev_p(\nu(\chi))\modulo p^{\frac{1}{p-1}},$$
when $\chi$ has order $p$ (this last observation is implicit in the arguments in \cite{FeKiKu12}).
\end{remark}
\subsubsection{Non-vanishing of maximal order characters}
We would like to have a version of Corollary \ref{cor:nonvanishingPS} for a general pro-$p$ group $G$ in the case $e(G)>1$. This seems to be a quite subtle question. In this section we will obtain some partial results in this direction. Below we denote by $$G[p]:=\{x\in G: px=0\},$$ 
the $p$-torsion of an abelian group $G$. Observe that for $G$ finite abelian  there is a natural identification:
\begin{equation}
\label{eq:Ghatp}\widehat{G/G[p]}\cong (\widehat{G})^p.
\end{equation} 
We start by considering the question of non-vanishing. 
\begin{proposition}\label{prop:nonvancond} Let $G$ be a finite abelian group and $R=\Oo_{\C_p}$. Let $\nu\in R[G]$ be a non-zero measure such that 
\begin{equation}\label{eq:nonvancond} \ev_p(\nu(\mathbf{1}))<v_p( G[p]).\end{equation}
Then there exists a character $\chi\in \widehat{G}\setminus (\widehat{G})^p$  such that
$$\nu(\chi)\neq 0.$$
\end{proposition}
\begin{proof}
Let $Y\subset G$ be a set of representatives for the quotient $G/G[p]$. Then we have by the identification (\ref{eq:Ghatp}) and character orthogonality
\begin{align*}
 \sum_{\chi\in \widehat{G}\setminus(\widehat{G})^p}\nu(\chi)\left(\sum_{y\in Y} \overline{\chi}(y)\right)
&= \sum_{y\in Y}\sum_{\chi\in \widehat{G}}\overline{\chi}(y) \nu(\chi) -\sum_{y\in Y}\sum_{\chi\in (\widehat{G})^p}\overline{\chi}(y)\nu(\chi) \\
&=|G|\sum_{y\in Y}\nu_y-\frac{|G|}{|G[p]|}\sum_{y\in Y}\sum_{x\in G[p]}\nu_{x+y}= |G|\sum_{y\in Y}\nu_y-\frac{|G|}{|G[p]|}\nu(\1). 
\end{align*}
By the assumption (\ref{eq:nonvancond}) we conclude that $\ev_p(\tfrac{|G|}{|G[p]|}\nu(\1))<\ev_p(|G|)$ and thus the above is non-vanishing, which finishes the proof.
\end{proof}
Under a slightly stronger assumption in place of (\ref{eq:nonvancond}), we can get an actual maximum modulus principle.
\begin{proposition}[Refined maximum modulus principle]\label{prop:1/pmi}
Let $G\cong \prod_{i=1}^n \Z/p^{m_i}$ and $R=\Oo_{\C_p}$. Let $\nu\in R[G]$ be a non-zero measure such that
\begin{equation}\label{eq:1/pmi}\ev_p(\nu(\mathbf{1}))<\sum_{i=1}^n \frac{1}{p^{m_i-1}}.\end{equation}
Then there exists a character $\chi\in \widehat{G}\setminus (\widehat{G})^p$  such that
$$\ev_p(\nu(\chi))\leq \ev_p(\nu(\mathbf{1})).$$
\end{proposition}
\begin{proof}
For $i=1,\ldots, n$ let $\chi_i$ be a generator of $\widehat{\Z/p^{m_i}}$ viewed as a character of $G$. Then we get by interchanging the sums that
\begin{align}\label{eq:conclude}
\sum_{0\leq a_1<p}\cdots \sum_{0\leq a_n<p} \nu(\chi_1^{a_1}\cdots \chi_n^{a_n})
= \sum_{x\in G} \nu_x \prod_{i=1}^n \left(\sum_{0\leq a_i<p}\chi_i(a_i x)  \right).
\end{align}
Observe that for any $1\leq i\leq n$ and $x\in G$ we have:
$$\ev_p\left(\sum_{0\leq a_i<p}\chi_i(a_i x)\right)=\begin{cases}
    1,& \chi_i(x)=1,\\
    \ev_p\left(\frac{\chi_i(px)-1}{\chi_i(x)-1}\right), & \chi_i(x)\neq 1.
\end{cases} $$
By standard facts about roots of unity it holds for $\chi_i(x)\neq 1$ that
$$\ev_p\left(\frac{\chi_i(px)-1}{\chi_i(x)-1}\right)\geq  \frac{1}{p^{m_i-1}-p^{m_i-2}}-\frac{1}{p^{m_i}-p^{m_i-1}}= \frac{1}{p^{m_i-1}},$$
and so by the strong triangle inequality we conclude from the assumption (\ref{eq:1/pmi}) that the right-hand side of (\ref{eq:conclude}) has valuation strictly bigger than $\ev_p(\nu(\1))$. Thus it follows that there exists $(a_1,\ldots, a_n)\not\equiv (0,\ldots, 0)\modulo p$ such that 
$$\ev_p(\nu(\chi_1^{a_1}\cdots \chi_n^{a_n}))\leq \ev_p(\nu(\1)), $$
and since clearly $\chi_1^{a_1}\cdots \chi_n^{a_n}\notin (\widehat{G})^p$ this finishes the proof. 
\end{proof}

Using the same type of argument as in the proof of Theorem \ref{thm:nonvanishingPS}, we obtain variants of Corollary \ref{cor:nonvanishingPS} in the horizontal case. For simplicity we will restrict the case where the $p$-adic valuation of $R$ is discrete. Recall from Section \ref{sec:horizontal case} that a profinite abelian group $G$ is  \emph{horizontal pro-$p$}  if there is an isomorphism $G\cong \prod_{n\in \N} \Z/p^{m_n}$ for some $m_n\geq1$.    
\begin{theorem}\label{thm:nonvanishing}
Let $G$ be a horizontal pro-$p$ abelian  group and let $R$ be a $p$-adically complete subring of $\Oo_{\C_p}$ with discrete $p$-adic valuation. Let $\nu\in R\llbracket G\rrbracket$ be a non-zero horizontal measure. Then there exists a finite subset $M_{\nu}\leq \widehat{G}$ such that the following holds: for any non-trivial finite subgroup $M\leq \widehat{G}$ there exists $\chi\in M\setminus M^p$ and $\chi_0\in M_{\nu}$ such that
\begin{equation}
\nu(\chi \chi_0)\neq 0.
\end{equation}
Furthermore, if $R[1/p] \cap \Q_p(\mu_p) = \Q_p$ then for  any character $\chi\in \widehat{G}$ there exists $\chi_0\in M_{\nu}$ such that
\begin{equation}
\nu(\chi \chi_0)\neq 0.
\end{equation}  
\end{theorem}
\begin{proof} First of all by Lemma \ref{lem:twisting} we may reduce to the case $\nu(\1)\neq 0$. Consider the poset  consisting of finite subgroups $M\leq \widehat{G}$ such that $\nu(\chi)=0$ for all $\chi\in M\setminus M^p$ with ordering given by inclusion. This poset is noetherian by Proposition \ref{prop:nonvancond} using that any ascending chain $M_1\leq M_2\leq \ldots$ will satisfy $v_p(M_n[p])\rightarrow \infty$  as $n\rightarrow \infty$ since $G$ is horizontal pro-$p$. Now the first claim follows by applying the same noetherian argument as in the proof of Theorem \ref{thm:nonvanishingPS}, noting that any $\chi\in \widehat{G}$ is of finite order by Lemma \ref{lem:finiteorder}. To prove the last claim we apply the above with $M=\langle \chi\rangle$. Note that $\chi'\in M \setminus M^p$ implies that $\chi$ and $\chi'$ have the same order and thus are Galois conjugate. Now the last claim follows from the Galois equivariance (\ref{eq:galoisinv}). We will skip the details.  
\end{proof}
If we restrict to $G$ with finite exponent we obtain the following strengthening:
\begin{theorem}\label{thm:nonvanishingpm}
Let $G$ be a horizontal pro-$p$ abelian  group of finite exponent $e(G)<\infty$ and let $R$ be a $p$-adically complete subring of $\Oo_{\C_p}$ with discrete $p$-adic valuation. Let $\nu\in R\llbracket G\rrbracket$ be a non-zero horizontal measure. Then there exists a finite subset $M_{\nu}\leq \widehat{G}$ such that the following holds: for any non-trivial finite subgroup $M\leq \widehat{G}$ there exists  $\chi\in M\setminus M^p$ and $\chi_0\in M_{\nu}$ such that
\begin{equation}
\ev_p(\nu(\chi \chi_0))=  \ev_p(\nu).
\end{equation}
Furthermore, if $R[1/p] \cap \Q_p(\mu_p) = \Q_p$ then 
for any character $\chi\in \widehat{G}$ there exists $\chi_0\in M_{\nu}$ such that
\begin{equation}
v_p(\nu(\chi \chi_0))=  \ev_p(\nu).
\end{equation}
  
\end{theorem}
\begin{proof}
We apply the same arguments as above to the poset consisting of finite subgroups $M\leq \widehat{G}$ such that $\ev_p(\nu(\chi))> \ev_p(\nu(\1))$ for all $\chi\in M\setminus M^p$. That this poset is noetherian follows from Proposition \ref{prop:1/pmi} using that $e(G)<\infty$ and the arguments above. The last claim follows by Galois equivariance. We will skip the details. 
\end{proof}
\begin{remark}
Note that the assumption (\ref{eq:nonvancond}) in Proposition \ref{prop:nonvancond} is necessary: Let $G$ be any abelian $p$-group and put  
\begin{equation}
\nu=\sum_{x\in G[p]}[x]\in \Z_p[G]. 
\end{equation}
Then $\nu(\chi)=0$ for all $\chi\in \widehat{G}\setminus (\widehat{G})^p$ and $\nu(\1)=|G[p]|$. 
\end{remark}
\begin{remark}
We believe that Proposition \ref{prop:nonvancond} should hold under the weaker assumption (\ref{eq:nonvancond}) which would allow one to prove that the conclusion of Theorem \ref{thm:nonvanishingpm} holds but without the condition that $G$ has finite exponent, i.e.\ $e(G)<\infty$.
\end{remark}
\subsection{Horizontal Amice Transform}\label{sec:horizontalAmice}
In this section we will study a horizontal version of the usual vertical Amice transform (\ref{eq:usual}). Let $R$ be a ring and let $\T=(T_n)_{n\in \N}$ be an infinite sequence of indeterminates.
Then we define the \emph{infinite dimensional polynomial algebra} over $R$ as the ring
\begin{equation}
R\llbracket \T \rrbracket:= \varprojlim_{n\in \N} R\llbracket T_1,\ldots, T_n\rrbracket, 
\end{equation}
where the transition maps for $m\leq n$ are given by sending $T_i\mapsto 0$ for $m<i\leq n$ and $T_i\mapsto T_i$ for $i\leq m$. We endow $R\llbracket T_1,\ldots, T_n \rrbracket$ with the standard $(T_1,\ldots, T_n)$-adic topology which induces an inverse limit topology on $R\llbracket\underline{T} \rrbracket$ which we will refer to as the \emph{$\underline{T}$-adic topology}. Note that we can represent the elements of $R\llbracket \T\rrbracket$ as power series in the indeterminates $(T_n)_{n\in \N}$ but that the $\underline{T}$-adic topology is different from the usual $(T_1,T_2,\ldots)$-adic one as in e.g.\ \cite{Nishimura}. More precisely, consider the indexing set
$$\mathcal{I}:=\{\a=(\alpha_n)_{n\in \N}: \alpha_n\in \Z_{\geq 0}, |\a|<\infty\},$$
where 
\begin{equation}
|\a|:=\sum_{n\in \N}\alpha_n.
\end{equation} 
We will denote the support of an index $\a\in \mathcal{I}$ by $$\supp(\a):=\{n\in \N: \alpha_n\neq 0\}.$$ 
Then each element of $R\llbracket \T\rrbracket$ can be identified with a unique power series of the form 
$$h(\T)=\sum_{\a\in \mathcal{I}} b(\a)\T^\a$$
where $b(\a)\in R$. 
In this representation the ring structure on $R\llbracket \T\rrbracket$ is given by the naive sum and product of power series. 
\begin{lemma}
For any sequence $(h_n(\T))_{n\in \N}\subset R\llbracket \T \rrbracket$, the following power series
\begin{equation}\label{eq:infsum} \sum_{n\in \N} T_n\cdot h_n(\T), \end{equation}     
defines an element of $R\llbracket \T \rrbracket$ by summing the coefficients.
\end{lemma}
\begin{proof}
This follows since for any $\a\in \mathcal{I}$ we have that only a finite number of summands in (\ref{eq:infsum}) can contribute to the coefficient of $\T^\a$, namely those corresponding to $n\in \supp(\a)$.
\end{proof}
We will realize horizontal pro-$p$ Iwasawa algebras as quotients of such infinite dimensional power series. Given an integer tuple $\underline{m}=(m_n)_{n\in \N}$ we define the following ideal of $R\llbracket \T \rrbracket$:
\begin{equation}\label{eq:I(m)}I(\underline{m}):=\left\{\sum_{n\in \N} ((T_n+1)^{p^{m_n}}-1)h_n(\T): h_n(\T)\in R\llbracket \T\rrbracket,n\geq 1\right\},\end{equation}
which is well-defined by the above lemma. Note that this is the $\underline{T}$-adic completion of the ideal generated  by $(T_n+1)^{p^{m_n}}-1$ for $n\in \N$. Fix an isomorphism $G\cong \prod_{n\in \N}\Z/p^{m_n}$ of profinite groups and define for $\a\in \mathcal{I}$ satisfying $\alpha_n=0,n>N$ the element:
\begin{align}\label{eq:Dr}
 D^\a:=\prod_{n\leq N} \left(\sum_{k=0}^{p^{m_n}-1} \binom{k}{\alpha_n}[k_{(n)}]\right) \in R\left[\prod_{n\leq N}\Z/p^{m_n}\right],  
 \end{align}
 where $k_{(n)}\in \prod_{j\leq N}\Z/p^{m_j}$ denotes the image of the element $(k\modulo p^{m_n})\in \Z/p^{m_n}$. For $\a=\underline{0}$ this is understood to be the element $D^0=[1]\in R[\langle 1\rangle]$. Notice that $D^\a$ depends on the choice of isomorphism $G\cong \prod_{n\in \N}\Z/p^{m_n}$ (see however Corollaries \ref{cor:independ} and \ref{cor:independ2}). When $\alpha_n\in\{0,1\}$ for all $n\in \N$ these correspond to the \emph{Kolyvagin derivatives} considered in \cite{WZhang} in the setting of theta elements of the base changes to an imaginary quadratic field of elliptic curves over $\mathbb{Q}$, see Section \ref{sec:KatoKoly} below for more details. In general, for a measure $D \in R[G]$ on a finite abelian group $G$ we get an associated function
\begin{equation}\label{eq:assfunc}\varphi_D: G \rightarrow R,\quad x\mapsto  D(\delta_x), \end{equation}
where $\delta_x:G\rightarrow R$ denotes the delta function at $x\in G$.
For a horizontal measure $\nu\in R\llbracket \prod_{n\in \N}\Z/p^{m_n}\rrbracket$ and $\alpha\in \mathcal{I}$ we define the \emph{$\a$th derivative of $\nu$} as
\begin{equation}\label{eq:nuDa}D^\a (\nu):=\nu (\varphi_{D^\a})\in R.\end{equation}
\begin{proposition}[Horizontal Amice transform]\label{prop:horamice}
Let $G$ be a horizontal pro-$p$ abelian group and fix an isomorphism $G\cong\prod_{n\in \N} \Z/p^{m_n}$. Let $R$ be a $p$-adically complete subring of $\Oo_{\C_p}$. Then we have an isomorphism of $R$-algebras
\begin{equation}
R\llbracket G \rrbracket \cong R\llbracket \T\rrbracket/I(\underline{m}),
\end{equation}
given by
\begin{equation}\label{eq:Amice}\nu\mapsto f_\nu(\T):=\sum_{\a\in \mathcal{I}}D^\a (\nu)\, \T^{\a} .\end{equation}
This isomorphism satisfies that for $\chi=(\chi_n)_{n\in \N}\in \widehat{G}$ we have
\begin{equation}\label{eq:AmiceI}f_\nu((\chi_n(1)-1)_{n\in \N})=\nu(\chi).\end{equation}
\end{proposition}
\begin{proof}
Note that for each $n\geq 1$ we have
\begin{align}\nonumber R\left[\prod_{1\leq i\leq n}\Z/p^{m_i}\right]&\cong R[T_1,\ldots, T_n]/(T_i^{p^{m_i}}-1: 1\leq i\leq n)\\
\nonumber&\cong R[ T_1,\ldots, T_n]/((T_i+1)^{p^{m_i}}-1: 1\leq i\leq n)\\
\label{eq:isoAmice}&\cong R\llbracket T_1,\ldots, T_n\rrbracket/((T_i+1)^{p^{m_i}}-1: 1\leq i\leq n), \end{align}
where the first isomorphism is given by mapping 
$$R\left[\prod_{1\leq i\leq n}\Z/p^{m_i}\right] \ni [(a_i\modulo p^{m_i})_{1\leq i\leq n}]\mapsto T_1^{a_1}\cdots T_n^{a_n},$$ 
the second is given by $T_i\mapsto T_i+1$, and the last is induced by the natural inclusion $R[T_1,\ldots, T_n]\subset R\llbracket T_1,\ldots, T_n \rrbracket$. 
 That this last map is an isomorphism follows by firstly showing this modulo $p^N$ using that $(T_i)^{p^{m_i}}\equiv 0 \modulo p$ in the quotient ring (\ref{eq:isoAmice}) and then letting $N\rightarrow \infty$ using that $R$ is $p$-adically complete. 
This implies that 
\begin{align}\label{eq:isoamicepre}
R\llbracket G \rrbracket= \varprojlim_{n\in \N} R\left[\prod_{1\leq i\leq n}\Z/p^{m_i}\right]\cong \varprojlim_{n\in \N} \left(R\llbracket T_1,\ldots, T_n\rrbracket/((T_i+1)^{p^{m_i}}-1: 1\leq i\leq n)\right),
\end{align}
where the transition maps for $m\leq n$ are given by mapping $T_i\mapsto 0$ for $m<i\leq n$ and $T_i\mapsto T_i$  for $i\leq m$ (which is well-defined since $(T_i+1)^{p^{m_i}}-1$ vanishes for $T_i=0$).
By this we get a surjective map 
\begin{equation}\label{eq:contrAmice}R\llbracket \T\rrbracket\twoheadrightarrow R\llbracket G \rrbracket,\end{equation}
with kernel given by the ideal 
$$\varprojlim_{n\in \N}((T_i+1)^{p^{m_i}}-1: 1\leq i\leq n),$$ 
which equals $I(\underline{m})$ as defined by the equation (\ref{eq:I(m)}) when identifying $R\llbracket \T\rrbracket$ with the space of infinite variable power series. This shows the desired isomorphism of $R$-algebras. Furthermore, one checks directly that (\ref{eq:Amice}) defines an inverse of (\ref{eq:isoamicepre}) since $\prod_{i\leq n} \binom{k_i}{\alpha_i}$ is exactly the $(\alpha_1,\ldots, \alpha_n)$th coefficient of $\prod_{i\leq n}(T_i+1)^{k_i}$. 
Finally, the interpolation property (\ref{eq:AmiceI}) follows from the choice of isomorphism (\ref{eq:isoAmice}).  
\end{proof}
\begin{remark}
The Amice transform $\nu\mapsto f_\nu$ can equivalently be expressed by the more familiar looking formula:
\begin{equation}
f_\nu(\T)=\int\limits_{x\in G} \prod_{n\in \N} (T_n+1)^{x_n}d\nu (x),
\end{equation}
where $x=(x_n)_{n\in \N}\in G\cong \prod_{n\in \N} \Z/p^{m_n}$ and the right-hand side makes sense as an element of $R\llbracket \T\rrbracket/I(\underline{m})$.
\end{remark}
\begin{remark}
A finitary version of the horizontal Amice transform was also employed by Darmon \cite[Sec.\ 3.1]{Darmon92} and Ota \cite[Prop. 3.3]{Ota} in the context of the \emph{Mazur--Tate Conjecture}.   
\end{remark}
\begin{remark}
As an immediate consequence of the horizontal Amice transform we get that $\ev_p(\nu(\1))=0$ implies that $\nu$ is invertible, yielding a different proof of Proposition \ref{prop:invertible} in the horizontal case. 
\end{remark}
\begin{example}[Sparse but infinite set of character zeroes]
The Amice transform yields a convenient way to construct horizontal measures with infinitely many character zeroes but of density $0$, e.g.:
\begin{equation}
\label{eq:exotic} \sum_{n\in \N} (-1)^n T_n.
\end{equation} 
\end{example}
\subsubsection{Iwasawa invariants of horizontal measures}
Note that as a consequence of (\ref{eq:Amice}) every element $f\in R\llbracket \T\rrbracket/I(\underline{m})$ has a unique representative 
\begin{equation}\label{eq:minimalrep}f(\T)=\sum_{\a\in \mathcal{I}:\alpha_n< p^{m_n}} b_f(\a)\T^\a,\end{equation} 
using that $D^\a(\nu)=0$ if $\alpha_n\geq p^{m_n}$ for some $n\in \N$. We write $b_\nu(\a):=b_{f_\nu}(\a)$.

\begin{definition}\label{def:lambdamu}Let $G$ be a horizontal pro-$p$ abelian group and fix an isomorphism $G\cong \prod_{n\in \N} \Z/p^{m_n}$. Let $R$ be a $p$-adically complete subring of $\Oo_{\C_p}$ with discrete $p$-adic valuation. Then we define the \emph{$\mu$-invariant} of a measure $\nu\in R\llbracket G\rrbracket$ as:
 $$ \mu(\nu):=\min\{\ev_p(b_\nu(\a)): \a\in \mathcal{I}\} .$$
 and the \emph{$\lambda$-invariant of $\nu$} as:
 $$ \lambda(\nu):=\min\{|\a|: \ev_p(b_\nu(\a))=\mu(\nu)\}. $$
 \end{definition}
Note that if $G=\Z_p$ the above formulas recover the usual notions of $\mu$- and $\lambda$-invariants in Iwasawa theory. As an immediate consequence of the horizontal Amice transform, we get that the $\mu$-invariant is independent of the choice of coordinates.
\begin{corollary}\label{cor:independ}
The $\mu$-invariant $\mu(\nu)$ is independent of the choice of isomorphism $G\cong \prod_{n\in \N} \Z/p^{m_n}$.
\end{corollary}
\begin{proof}
We see that $\mu(\nu)\geq n$ exactly if $f_\nu$ is contained in the ideal $p^nR\llbracket \T\rrbracket/I(\underline{m})$. But since the Amice transform is an isomorphism of $R$-algebras this means exactly that $\nu \in p^n R\llbracket G\rrbracket$.
\end{proof}
We end this section by recording some properties of the derivatives $D^\a (\nu)$ under pushforward.  
\begin{lemma}\label{lem:Dalphacongruence}
Let $G\cong \prod_{n\in \N} \Z/p^{m_n}$ and $R\subset \Oo_{\C_p}$ be as above. Consider the pushforward
$$ R\llbracket G\rrbracket\rightarrow R\llbracket  (\Z/p)^\N\rrbracket,\quad \nu\mapsto \rho(\nu),$$ along the canonical projection $\rho:\prod_{n\in \N} \Z/p^{m_n}\twoheadrightarrow  (\Z/p)^\N$. Then for $\nu\in R\llbracket G\rrbracket$ and $\a\in \mathcal{I}$ such that $\alpha_n\leq p-1$ for all $n\in \N$ it holds that
$$ D^\a (\nu)\equiv D^\a(\rho(\nu))\modulo p.$$
\end{lemma}
\begin{proof}
This follows since $\binom{k}{\alpha_n}\equiv \binom{k'}{\alpha_n}\modulo p$ whenever $k\equiv k'\modulo p$ and $\alpha_n<p$.
\end{proof}  
\begin{remark}\label{rem:dualnumbers}We note that reducing the co-domain of the Amice transform modulo $p$, we obtain a map 
$$\Z_p\llbracket (\Z/p)^\N\rrbracket\twoheadrightarrow \F_p\llbracket \T\rrbracket/J,\quad J=\left\{\sum_{n\in \N} (T_n)^p h_n(\T): h_n(\T)\in \F_p\llbracket \T\rrbracket\right\}, $$
where the right-hand side is an infinite dimensional version of the classical ``$p$-dual numbers''. This is exactly the pullback of the global sections to the special fiber of $\Spec \Z_p\llbracket (\Z/p)^\N\rrbracket$. In particular, the partial derivative with respect to any of the variables $T_n$ is well-defined on the special fiber and one sees immediately that for $\a\in \mathcal{I}$ with $\alpha_n<p$ for all $n\in \N$:
\begin{equation}D^\a (\nu)\equiv \frac{(\partial^\a f_\nu)(\underline{0})}{\a!}\modulo p.  \end{equation}
\end{remark}
\subsection{A Horizontal Substitute for Weierstrass Preparation}\label{sec:horweier}
In the case of the classical Iwasawa algebra the Weierstrass preparation theorem yields a canonical presentation (\ref{eq:f(T)WP}) of a power series $f(T)\in R\llbracket T\rrbracket$, with $R\subset \Oo_{\C_p}$ a $p$-adically complete subring with discrete $p$-adic valuation, in terms of its $\mu$- and $\lambda$-invariants. This implies in turn the precise formula (\ref{eq:WPformula}) for the $p$-adic valuation of $f(x)$ as $|x|_p\rightarrow 1$. We will now prove a horizontal analogue of this statement (cf.\ \cite{PassiVermani77} and \cite[Prop. 3.8]{TanBSD95}). Let  $\breve{\Z}_p$ denote the ring of integers of the $p$-adic completion of the maximal unramified extension of $\Q_p$.       
\begin{theorem}\label{thm:horWP}
Let $\nu\in \breve{\Z}_p\llbracket G\rrbracket $ be a non-zero horizontal measure where $G\cong (\Z/p^m)^\N$ for some $m\geq1$. Assume that $\lambda(\nu)\leq p-1$. Then it holds that
\begin{equation}\label{eq:Weierstrass}\ev_p(\nu)=\mu(\nu)+\frac{\lambda(\nu)}{p^m-p^{m-1}}. \end{equation}
Furthermore, if $\a\in \mathcal{I}$ satisfies $\ev_p(b_\nu(\a))=\mu(\nu)$ and $|\a|\leq p-2$ then there exists a character $\chi\in \widehat{G}$ which factors through  $\prod_{n\in \supp(\a)}\Z/p^m$ with non-trivial restriction to each factor such that 
\begin{equation}\label{eq:Weierstrass2}\ev_p(\nu(\chi))\leq \mu(\nu)+\frac{|\a|}{p^m-p^{m-1}}. \end{equation}
\end{theorem} 
For the proof, we will need the following elementary fact about polynomials.
\begin{lemma}\label{lem:restriction}
Let $P(T_1,\ldots, T_n)\in \overline{\F}_p[T_1,\ldots, T_n]$ be a non-zero polynomial of degree at most $p-2$. Then there exists $a_1,\ldots, a_n\in \F_p^\times$ such that $P(a_1,\ldots, a_n)\neq 0$.
\end{lemma}
\begin{proof} 
We will proceed by induction on the number of variables $n$. If $n=1$ this is clear by the condition on the degree and the fundamental theorem of algebra. Now assume that $P_{|T_1=a}$ vanish for all $a\in \F_p^\times$. Then since $\overline{\F}_p[T_2,\ldots, T_n][T_1]$ is a UFD we conclude that
$$ \prod_{a\in \F_p^\times}(T_1-a)=(T_1^{p-1}-1)\mid P(T_1,\ldots, T_n),$$
contradicting the assumption on the degree. Thus we can find $a\in \F_p^\times$ such that $P_{|T_1=a}$ is non-zero and so the conclusion follows by the induction hypothesis.  \end{proof}
\begin{proof}[Proof of Theorem \ref{thm:horWP}]
By considering $p^{-\mu(\nu)}\nu$ we may assume $\mu(\nu)=0$. If also $\lambda(\nu)=0$ we conclude  by Proposition \ref{prop:invertible} that $\nu$ is invertible and so (\ref{eq:Weierstrass}) follows in this case. So we may assume that $\lambda:=\lambda(\nu)\geq 1$. If we put $\underline{T}=((\zeta_{p^m})^{a_n}-1)_{n\in \N}$ with $a_n=0$ for $n$ sufficiently large and $\zeta_{p^m}$ a primitive $p^m$th root of unity, then for $\a\in \mathcal{I}$ and $b\in \breve{\Z}_p$ with $v_p(b)>0$ we have the lower bound
\begin{equation}
\label{eq:valuationlowerterms}v_p\left(b \underline{T}^\a  \right)\geq  1+ \frac{|\a|}{p^m-p^{m-1}}.
\end{equation}
Applying this to the representative (\ref{eq:minimalrep}) and recalling the definition of $\mu(\nu)$ and $\lambda(\nu)$  shows the inequality $\geq$ in (\ref{eq:Weierstrass}) since $\frac{\lambda(\nu)}{p^m-p^{m-1}}\leq 1$ by assumption. To prove the other inequality let $\underline{\beta}\in \mathcal{I}$ be an index such that $\ev_p(b_\nu(\underline{\beta}))=0$ and $|\b|=\lambda$. Assume, after a change of variables, that $\supp(\b)=\{1,\ldots, r\}$ with $1\leq r\leq p-1$ and let 
$$P(T_1,\ldots, T_r):=\sum_{\a\in \mathcal{I}: \supp (\a)\subset \{1,\ldots, r\},|\a|=\lambda} b_\nu(\a) \T^\a,$$
be the degree $\lambda$ homogeneous part of $(f_\nu)_{|T_i=0,i>r}$ (i.e. of the specialization of the representative (\ref{eq:minimalrep})). By the minimality of $|\b|$ and the assumption on $\supp(\b)$ we conclude that 
\begin{align*}f_\nu(T_1,\ldots, T_r,0,0,\ldots)& \equiv P(T_1,\ldots, T_r)\\
&\equiv T_1\ldots T_r Q(T_1,\ldots T_r) \modulo (p,\underline{T}^\a: \supp(\a)\subset \{1,\ldots ,r\}, |\a|=\lambda+1),\end{align*}
for some non-zero homogenous polynomial  
$Q(T_1,\ldots, T_r)\in \overline{\F}_p[T_1,\ldots, T_r]$  of degree $\lambda-r\leq  p-2$.

Consider now a formal thickening of the mod $p$ coordinate ring:
$$ \overline{\F}_p[T_1,\ldots, T_n] \hookrightarrow \overline{\F}_p[X][T_1,\ldots, T_r]/(X^{\lambda+1}),$$
and specialize $P$ to a power of the universal character in $X$, i.e. we put $T_i=(X+1)^{a_i}-1$ with $0\leq a_i\leq p^m-1$. Since $(X+1)^{a_i}-1\equiv a_iX\modulo X^2$ and $P$ is homogeneous of degree $\lambda$, we obtain
$$P((X+1)^{a_1}-1,\ldots, (X+1)^{a_r}-1)\equiv X^\lambda P(a_1,\ldots, a_r)\modulo X^{\lambda+1}.$$
Now putting $X=\zeta_{p^m}-1$ for a primitive $p^m$th root of unity $\zeta_{p^m}$ (i.e. specializing the universal character to an actual character of order $p^m$), we conclude by (\ref{eq:valuationlowerterms}) that
\begin{align}
\nonumber f_\nu((\zeta_{p^m})^{a_1}-1,\ldots,(\zeta_{p^m})^{a_r}-1 )&\equiv b_\nu(\underline{0})+(\zeta_{p^m}-1)^\lambda P(a_1,\ldots, a_r)\\
\label{eq:specializeuniversalchar}&\equiv b_\nu(\underline{0})+(\zeta_{p^m}-1)^\lambda a_1\cdots a_rQ(a_1,\ldots, a_r) \modulo p^{(\lambda+1)/(p^m-p^{m-1})}.
\end{align} 
If $v_p(b_\nu(\underline{0}))\leq   \frac{\lambda}{p^m-p^{m-1}}$ then the trivial character does the job by the preceding equation (note that under our assumptions this can only happen if $m=1$ and $\lambda=p-1$). Otherwise we can ignore this term. By Lemma \ref{lem:restriction}  we can find $a_1,\ldots, a_r\in \F_p^\times$ such that $Q(a_1,\ldots, a_r)\in \overline{\F}_p^\times$ which when plugged into (\ref{eq:specializeuniversalchar}) yields the inequality $\leq$ in (\ref{eq:Weierstrass}). This proves the first part. 

The second part (\ref{eq:Weierstrass2}) follows by the same arguments applying now Lemma \ref{lem:restriction} directly to the degree $|\b|$ homogenous part $P(T_1,\ldots,T_r)$ of $f_\nu
(T_1,\ldots, T_r,0,\ldots)$ (which applies by the stronger assumption $|\b|\leq p-2$).
\end{proof}
\begin{corollary}\label{cor:independ2}
Let $\nu\in \breve{\Z}_p\llbracket G\rrbracket $ be a horizontal measure and assume that there is an isomorphism $G\cong(\Z/p)^\N$ such that $\lambda(\nu)\leq p-1$. Then the $\lambda$-invariant does not depend on the choice of isomorphism $G\cong (\Z/p)^\N$. 
\end{corollary} 
\begin{proof}
This follows from the formula (\ref{eq:Weierstrass}) in view of Corollary \ref{cor:independ}.
\end{proof}
\begin{remark}
The condition $\lambda(\nu)\leq p-1$ is necessary for equality (\ref{eq:Weierstrass}) to hold by considering $f_\nu(\T)=p+T_1\cdots T_r$ for $r=\lambda(\nu)\geq p$. 
\end{remark}

\subsection{Kato--Kolyvagin Derivatives}\label{sec:KatoKoly}

A particularly important case of the above is when there exists $\a\in \mathcal{I}$ such that $\ev_p(b_\nu(\a))=\mu(\nu)$ satisfying $|\a|=|\supp(\a)|$ (i.e. $\alpha_n\in \{0,1\}$ for all $n\in \N$) as these are what appear in the conjectures of Kolyvagin and Kurihara, see e.g.\ \cite{BurungaleCastellaGrossiSkinner}. In this case we refer to $D^{\a}$ as a \emph{Kato--Kolyvagin derivative} and we will use the simplified notation
\begin{equation}\label{eq:defDr}D^r:= D^\a\quad \text{where }\alpha_i=1\text{ for }1\leq i\leq r\text{ and } \alpha_i=0\text{ for }i>r. \end{equation} 
In the presence of Kato--Kolyvagin derivatives of minimal valuation, we can bound the minimal valuation $\ev_p(\nu)$ without any restriction on the size of $\lambda(\nu)$.
\begin{proposition}\label{prop:kolyvagin}
Let $R$ be a $p$-adically complete subring of $\Oo_{\C_p}$ with discrete $p$-adic valuation and let $\nu\in R\llbracket (\Z/p)^\N \rrbracket$ be a non-zero horizontal measure. Assume that $\ev_p(D^r(\nu))=\mu(\nu)$ for some $r\geq 1$. Then there exists a character $\chi\in \widehat{(\Z/p)^\N}$ factoring through $\prod_{1\leq i\leq r}\Z/p$ such that  
\begin{equation}\label{eq:KatoKolyUB}\ev_p(\nu(\chi))\leq \mu(\nu)+\frac{r}{p-1}.\end{equation}
\end{proposition}
\begin{proof}
We may clearly reduce to the case $\mu(\nu)=0$. Let $$P(T_1,\ldots, T_r)=f_\nu(\T)_{|T_n=0,n>r},$$ 
be the restriction of the minimal representative (\ref{eq:minimalrep}) of $f_\nu$ to the first $r$ variable. By the principle of inclusion-exclusion we see that the coefficient of $T^j$ for $0\leq j \leq r$ of the polynomial  
\begin{equation}\label{eq:PIEproof}\sum_{I\subset \{1,\ldots,r\}} (-1)^{|I|} P_{|T_i=0,i\in I, T_i=T,i\notin I},\end{equation}
is exactly 
$$\sum_{k_1+\ldots +k_r=j} b_\nu((k_i)_{1\leq i\leq r})\sum_{ I\subset \{i: k_i= 0\} }(-1)^{|I|}= \begin{cases} 0,& 0\leq j\leq r-1,\\
 b_\nu(\a),& j=r.\end{cases} $$
Putting $T=\zeta-1$ with $\zeta$ a primitive $p$th root of unity, the valuation of (\ref{eq:PIEproof}) becomes $r/(p-1)$. This yields the wanted conclusion by the interpolation property (\ref{eq:AmiceI}) of the horizontal Amice transform.
\end{proof}
\begin{remark}
Note that for the measure $\nu\in \Z_p\llbracket (\Z/p)^\N\rrbracket$ with $f_\nu(\T)=T_1^pT_2-T_1T_2^p$ it holds that $\mu(\nu)=0$, $\lambda(\nu)=p+1$ but $v_p(\nu)\geq \frac{p+2}{p-1}$ by \cite[Prop.\ 3.8]{TanBSD95}. 
\end{remark}
\subsection{Augmentation Rank}
Using the Amice transform and a result of Tan \cite[Prop. 3.8]{TanBSD95} we can give an alternative description of the valuation $\ev_p(\nu)$ of a horizontal measure, providing a link to the conjecture of Mazur and Tate \cite{MazurTate} (see Remark \ref{rem:augrank} below). 

Let $R$ be a ring and $G$ a profinite abelian group. We define the \emph{augmentation ideal} of $R\llbracket G\rrbracket$ as: 
$$I=I_\mathrm{aug}(R\llbracket G\rrbracket):=\ker(\1)=\{\nu\in R\llbracket G\rrbracket: \nu(\1)=0 \}.$$
For $r\geq 0$, we define the \emph{$r$th augmentation ideal} as:
\begin{equation}\label{eq:Ir}I_{r}=I_r(R\llbracket G\rrbracket) :=\varprojlim_{H\leq G\text{ open}} \left(I_\mathrm{aug}(R[G/H])\right)^r. \end{equation} 
 with the convention that $I_{0}=R\llbracket G\rrbracket$. For $G\cong \prod_{n\in \N}\Z/p^{m_n}$ and $r\geq 1$ then under the horizontal Amice transform from Proposition \ref{prop:horamice} one can easily verify that
\begin{equation}\label{eq:augmentation}I_{r}=\left\{ \sum_{\a\in \mathcal{I}: |\a|=r} \T^\a h_\a(\T): h_\a(\T)\in R\llbracket \T \rrbracket   \right\}+I(\underline{m}). \end{equation}
Note that $I_r$ is the $\underline{T}$-adic completion of $I^r$ and so for $G$ horizontal pro-$p$ we have  $I^{r}\subsetneq I_r$ (e.g.\ $\sum_{n\in \N} T_n T_{n+1}\cdots T_{n+r-1} \in I_r\setminus I^r$). 

We define the \emph{augmentation rank} of a non-zero measure $\nu\in R\llbracket G\rrbracket$ as:
\begin{equation}  r_\mathrm{aug}(\nu):=\max\{r\geq 0: \nu \in I_{r} \}. \end{equation} 
\begin{proposition}\label{prop:augmentation}
Let $\nu\in \Z_p\llbracket (\Z/p)^\N\rrbracket$ be a non-zero measure such that $\nu(\1)=0$. Then it holds that 
\begin{equation}\label{eq:augrank}r_\mathrm{aug}(\nu)=(p-1)\ev_p(\nu). \end{equation}
\end{proposition}
\begin{proof} The inequality $\leq$ follows directly from the formula (\ref{eq:augmentation}). Now let $r=r_\mathrm{aug}(\nu)$. Using the representation (\ref{eq:augmentation}) we see that we can find a representation of $f_\nu(\T)$ in $\Z_p\llbracket \T\rrbracket$ of the form  $P(\T)+Q(\T)$ with $P(\T) $ homogeneous of degree $r$ and $Q(\T)\in I_{r+1}$. We claim that there exists $\underline{a}\in \mathcal{I}$ such that $P(\underline{a})\not\equiv 0\modulo p$. If not then it follows by \cite[Prop. 3.8]{TanBSD95} that $P(\T)\in I_{r+1}$ and so $\nu\in I_{r+1}$ contrary to the assumption (note that \emph{loc.\ cit.} considers groups rings over finite $p$-primary groups but this implies the infinite dimensional version by the definition (\ref{eq:Ir})). Now given such $\underline{a}\in\mathcal{I}$ we specialize $P(\T)$ to the universal character $((X+1)^{a_n}-1)_{n\in \N}$ as in the proof of Theorem \ref{thm:horWP}. Putting $X=\zeta_p-1$ with $\zeta_p$ a primitive $p$th root of unity and  we conclude
$$\ev_p(f_\nu((\zeta_p)^{\underline{a}}-\underline{1}))=\frac{r}{p-1}+\ev_p(P(\underline{a}))=\frac{r}{p-1},$$
where $(\zeta_p)^{\underline{a}}-\underline{1}\in \prod_{n\in \N}\Oo_{\C_p}$ denotes the element with $n$th entry equal to $(\zeta_p)^{a_n}-1$. This proves the inequality $\geq$ in (\ref{eq:augrank}) and so we are done.  
\end{proof}
\begin{remark}
For general horizontal Iwasawa algebras the situation is less clear. For instance for $G\cong (\Z/p^m)^\N$ with $m>1$ and $\nu_1,\nu_2\in \Z_p\llbracket G\rrbracket$ satisfying $f_{\nu_1}(\T)=pT_1$ and $f_{\nu_2}(\T)=T_1$, it holds that $r_\mathrm{aug}(\nu_1)=r_\mathrm{aug}(\nu_2)=1$ but $\ev_p(\nu_1)=1+\frac{1}{p^m-p^{m-1}}$ and $\ev_p(\nu_2)=\frac{1}{p^m-p^{m-1}}$. The key feature here being that $r_\mathrm{aug}(pT_1)=p$ for $m=1$ but $r_\mathrm{aug}(pT_1)=1$ for $m>1$, see \cite[Section 3]{TanBSD95} for more details. 
\end{remark}
\section{Horizontal Norm Relations for Modular Symbols}\label{sec:normrelQ}
In this section we will prove the horizontal norm relations for modular symbols of integral weight $k$ and general nebentypus and use them to construct \emph{theta elements}. This will be key in constructing horizontal measures from elliptic curves and more generally holomorphic modular forms. While these general norm relations and their proof are almost certainly known to experts (cf. for example \cite[Section 1.3]{MazurTate}), the authors have not been able to locate a complete reference and hence for completeness we provide a detailed proof here. In order to provide a uniform treatment we will be working with ``additive twist $L$-series'' and thus our treatment, although equivalent, will differ slightly from the usual approach, see e.g. \cite[Sec. 10]{Williamsnotes}.
\subsection{Additive Twist {\it L}-series}
In the first part of this section, there is no extra effort in working with general newforms  for $\GL_2/\Q$, meaning either a holomorphic newform, a Hecke--Maa{\ss} newform or an Eisenstein series (for a definition in this generality see e.g. \cite[Sec. 4+6]{DuFrIw02}). Let $f$ be a Hecke newform for $\GL_2/\Q$ of integral non-negative weight $k$, level $N$ and nebentypus $\epsilon_f\modulo N$ with Fourier coefficients $a_f(n),n\geq 1$ normalized so that the Ramanujan Conjecture is $|a_f(n)|\leq d(n)n^{(k-1)/2}$ (notice that this normalization differs from that in \cite[Sec. 4]{DuFrIw02}). In all cases, the following proofs show that the norm relations for additive twist $L$-series are a formal consequence of the Hecke relations \cite[Eq. (6.5)]{DuFrIw02};
\begin{align}\label{eq:Heckerel}a_f(m)a_f(n)=\sum_{d|(m,n)}\epsilon_f(d) d^{k-1}a_f(\tfrac{mn}{d^2}),\\
a_f(mn)=\sum_{d|(m,n)}\mu(d)\epsilon_f(d)d^{k-1}a_f(\tfrac{m}{d}) a_f(\tfrac{n}{d}),\end{align}
where $\mu$ denotes the M\"{o}bius function. Note that what is sometimes referred to as {\lq\lq}trivial central character and level $N${\rq\rq} corresponds to $\epsilon_f$ equal to the principal character modulo $N$. 

Given $f$ as above and $x\in \Q$, we define the associated  \emph{additive twist $L$-series} as the analytic continuation of   
\begin{equation} L(f,x,s):=\sum_{n\geq 1}\frac{a_f(n)e(nx)}{n^{s}},\quad  \Re s\gg_f 1, \end{equation}
where $e(x)=e^{2\pi i x}$. For a proof of analytic continuation see e.g. \cite[Proposition 5.8]{DrNo22}. The additive twist $L$-series satisfy functional equations relating $(s,\tfrac{a}{q})\leftrightarrow (k-s,-\tfrac{\overline{a}}{q})$ where $(a,q)=1$ and $a\overline{a}\equiv 1\modulo q$. We will thus refer to the value at $s=k/2$ as the \emph{central value}. We have the following general horizontal norm relations. We note the following formula (\ref{eq:hornormrel1}) takes a very pleasant form when $q_2$ is square-free (see Corollary \ref{cor:normrelationstheta} below).    
\begin{proposition}[Horizontal norm relations]\label{prop:norm-relations}
Let $f$ be a Hecke newform of weight $k\in \Z_{\geq 0}$ and nebentypus $\epsilon_f$, and let $q_1,q_2\geq 2$ be two coprime integers. Then we have for all $a_0\in (\Z/q_1)^\times$ and $s\in \C$ the following relation holds:
\begin{align}\label{eq:hornormrel1} \sum_{\substack{a\in (\Z/q_1q_2)^\times:\\ a\equiv a_0\modulo q_1}}L(f,\tfrac{a}{q_1q_2},s)
=  \sum_{d_1d_2d_3=q_2} (d_1)^{k-2s}\mu(d_1)\epsilon_f(d_1)\mu(d_2)(d_3)^{1-s}a_f(d_3) L(f,\tfrac{d_1\overline{d_2}a_0}{q_1},s), \end{align}
where $d_2\overline{d_2}\equiv 1\modulo q_1$. Here $\mu$ denotes the M\"{o}bius function.
\end{proposition}
\begin{proof} Note that by the Chinese Reminder Theorem there is a set theoretic bijection 
$$ \{a\in (\Z/q_1q_2)^\times: a\equiv a_0\modulo q_1\}\leftrightarrow \{a_0q_2\overline{q_2}+q_1b: b\in (\Z/q_2)^\times\}, $$
where $\overline{q_2}$ is such that $q_2\overline{q_2}\equiv 1\modulo q_1$. Using this, we get for $\Re s\gg_f 1$ by absolute convergence that 
\begin{align}\label{eq:interchangesums} \sum_{\substack{a\in (\Z/q_1q_2)^\times:\\ a\equiv a_0\modulo q_1}}L(f,\tfrac{a}{q_1q_2},s) = \sum_{n\geq 1} \frac{a_f(n)e(\tfrac{na_0\overline{q_2}}{q_1})}{n^{s}}\sum_{b\in (\Z/q_2)^\times}e(\tfrac{nb}{q_2}).   \end{align}
The inner sum can be easily evaluated (e.g apply \cite[Lemma 3]{Sh75} with the trivial character) as 
$$\sum_{b\in (\Z/q_2)^\times}e(\tfrac{nb}{q_2})= \sum_{d| ( n,q_2)}d\, \mu(\tfrac{q_2}{d}).$$
Plugging this into (\ref{eq:interchangesums}), interchanging the sums and applying the Hecke relations (\ref{eq:Heckerel}) we arrive at
\begin{align} 
\sum_{d\mid q_2} d \mu\left(\tfrac{q_2}{d}\right) \sum_{n\geq 1} \frac{a_f(nd)}{(nd)^s} e\left(\frac{nda_0\overline{q_2}}{q_1}\right)
=  \sum_{d\mid q_2} d^{1-s} \mu\left(\tfrac{q_2}{d}\right) \sum_{\delta|d} a_f(\tfrac{d}{\delta}) \mu(\delta)\epsilon_f(\delta) \delta^{k-1-s}L(f,\tfrac{\delta d a_0\overline{q_2}}{q_1},s),
\end{align}
Now the result follows by analytic continuation after doing the change of variables $d_1=\delta,d_2=\tfrac{q_2}{d},d_3=\tfrac{d}{\delta}$. 
\end{proof}
\begin{remark}
A different proof of the norm relations for modular symbols can be given by observing that for $q_2$ prime we have the following relation in $\Delta_0$ (defined as in (\ref{eq:Delta0}) below);
$$\sum_{\substack{a\in (\Z/q_1q_2)^\times:\\ a\equiv a_0\modulo q_1}} \{\tfrac{a}{q_1q_2},i\infty\}= T_{q_2} \{\tfrac{a_0}{q_1},i\infty\}-\{\tfrac{q_2a_0}{q_1},i\infty\}-\{\tfrac{\overline{q_2}a_0}{q_1},i\infty\}. $$ 
The above proof of Proposition \ref{prop:norm-relations} has the advantage of being true for any $\GL_2$-newform (including non-cohomological forms).
\end{remark}
\subsubsection{The Birch--Stevens formula}Key to our applications is a formula due to Birch and Stevens which yields a relation between additive and multiplicative twists. Recall that we define for a Dirichlet character $\chi\modulo q$ the $L$-function of $f$ twisted by $\chi$ as  
\begin{equation}\label{eq:twistedLvalue}L(f,\chi,s):= \sum_{n\geq 1} \frac{a_f(n)\chi(n)}{n^{s}},\quad \Re s\gg_f 1,\end{equation}
and elsewhere by analytic continuation. We denote by $f\otimes \chi$ the newform corresponding to the automorphic representation $\pi_f\otimes \chi$, where $\pi_f$ denotes the automorphic representation corresponding to the newform $f$ of level $N$. The level of $f\otimes \chi$ divides $Nq$. The Fourier coefficients of the twisted newform $f\otimes \chi$ satisfies for $\ell$ prime not dividing $Nq$ that
\begin{equation}\label{eq:twistedform} a_{f\otimes \chi}(\ell)=a_f(\ell)\chi(\ell),\quad \epsilon_{f\otimes \chi}(\ell)= \epsilon_f(\ell)\chi(\ell)^2,
\end{equation}
and if $(N,q)=1$ then $L(f\otimes \chi,s)=L(f,\chi,s)$.
 Note that we always have $L(f\otimes \chi,k/2)=0\Leftrightarrow L(f,\chi,k/2)=0$.  
We have the following general formula.
 
\begin{proposition}[Birch--Stevens formula]\label{cor:BirchStevens} 
Let $f$ be a Hecke newform of weight $k\in \Z_{\geq 0}$ and nebentypus $\epsilon_f$. Let $\chi \modulo q$ be a primitive Dirichlet character. Then we have
\begin{align}\label{eq:BS}
\sum_{a\in (\Z/q)^\times} \overline{\chi(a)}L(f,\tfrac{a}{q},s)=\tau(\overline{\chi}) L(f, \chi,s), \quad s\in \C,
\end{align}
where $\tau(\overline{\chi})=\sum_{a\in (\Z/q)^\times}\overline{\chi}(a)e(a/q)$ is a Gau{\ss} sum.
\end{proposition}
\begin{proof}
For $\Re s\gg_f 1$ we can by absolute convergence interchange the two sums on the left-hand side of (\ref{eq:BS}). Now we apply the classical formula
$$\sum_{a\in (\Z/q)^\times}\overline{\chi(a)}e(na/q)=\chi(n)\tau(\overline{\chi}),$$
using that $\chi$ is primitive. This yields the wanted for $\Re s\gg_f 1$ and elsewhere by analytic continuation. 
For details see e.g. \cite[Prop. 8.1]{DrNo22}.
\end{proof}

\subsection{Period Polynomials}
When $f$ is a holomorphic eigenform of weight even weight $k\in 2\Z_{\geq 0}$ and level $N$, one can write the central value $L(f,\tfrac{a}{q},k/2)$ in terms of the \emph{period polynomials} associated to $f$ as we will now recall. We refer to \cite[Sec. 10]{Williamsnotes} for details. For any nebentypus, such $f$ defines a modular form of weight $k$ with \emph{trivial} multiplier for the group 
$$\Gamma_1(N)=\{\gamma\in \PSL_2(\Z):\gamma\equiv \begin{psmallmatrix}1&\ast\\0&1 \end{psmallmatrix}\modulo N\}.$$ 
Denote by $\Delta_0$ the degree zero divisors of $\mathbf{P}^1(\Q)$, i.e. the free abelian group generated by 
\begin{equation}\label{eq:Delta0}\{r_1,r_2\}:=[r_1]-[r_2]\in \Delta_0,\quad r_1,r_2\in \mathbf{P}^1(\Q),\end{equation}
equipped with the $\GL_2(\Q)$-action given by
$$\gamma\{r_1,r_2\}=\{\gamma r_1,\gamma r_2\}.$$
Denote by $V_{k-2}(\C)$ the set of polynomials in one variable of degree at most $k-2$ with coefficients in $\C$ equipped with the $\GL_2(\Q)$-action given by
$$[P|\begin{psmallmatrix}a&b\\c&d\end{psmallmatrix}](X)=(-bX+a)^{k-2}P\left( \frac{dX-c}{-bX+a}\right). $$ 
We define the set of (\emph{$\C$-valued}) \emph{modular symbols of weight $k$ and level $N$ (and general nebentypus)} as the set of $\Gamma_1(N)$-homomorphisms
\begin{equation}
\mathrm{Symb}_{k,N}(\C):=\mathrm{Hom}_{\Gamma_1(N)}(\mathbf{P}^1(\Q), V_{k-2}(\C)).
\end{equation}
The Hecke algebra $\mathcal{H}_N$ of level $N$ acts naturally on $\mathrm{Symb}_{k,N}(\C)$ (see e.g. \cite[Sec. 5]{PaPo13}) and this action commutes with the {\lq\lq}archimedian Hecke operator{\rq\rq} $\iota$ given by the action of the matrix $\begin{psmallmatrix}-1 &0\\0& 1 \end{psmallmatrix}$ (which normalizes $\Gamma_1(N)$). We have a natural injection of the vector space of cusp forms (holomorphic and anti-holomorphic) of weight $k$ and level $N$: 
\begin{equation}\label{eq:ES1}\mathcal{S}_k(\Gamma_1(N))\oplus \overline{\mathcal{S}_k(\Gamma_1(N))}\hookrightarrow \mathrm{Symb}_{k,N}(\C),\end{equation}
given by 
$$f\mapsto \phi_f,\quad \phi_f(\{r_1,r_2\})= \int_{r_1}^{r_2}(zX+1)^{k-2}f(z)dz,$$
and similarly for anti-holomorphic forms. One can check by hand that (\ref{eq:ES1}) is Hecke equivariant, and it is a classical fact of Eichler--Shimura that the map (\ref{eq:ES1}) defines an isomorphism of Hecke-modules onto the subspace of \emph{cuspidal symbols} (see \cite[Th. 5.17]{PaPo13}). 

Now let $f$ be a (holomorphic and Hecke-normalized) Hecke newform of even weight $k\in 2\Z_{\geq 0}$ and level $N$ and let $\pm$ be a sign. Then the following symbols are diagonalized by the algebra generated by $\mathcal{H}_N$ and $\iota$:
$$\phi_f^\pm:= \frac{1\pm (-1)^{k/2-1} \iota}{2}\phi_f\in \mathrm{Symb}_{k,N}(\C).$$
Denote by 
$$K_f:=\Q(a_f(n):n\geq 1),$$ 
the Hecke field of $f$ and by $\Oo_f$ the ring of integers of $K_f$. Then we have the following classical integrality result due to Manin and Shimura, see \cite[Theorem 1]{Shimura77}. For a more modern treatment consult \cite[Proposition 5.11]{PaPo13}. 
\begin{theorem}[Manin,\,Shimura]\label{thm:rational}
Let $f\in \mathcal{S}_k(\Gamma_1(N))$ be a Hecke newform of weight $k\in \Z_{\geq 0}$ and level $N$. Then there exists periods $\Omega^\pm\in \C^\times$ such that for all $r_1,r_2\in \mathbf{P}^1(\Q)$ the coefficients of the polynomial
\begin{equation}\label{eq:normalizedpoly} \frac{\phi_f^\pm}{\Omega^\pm}(\{r_1,r_2\})(X)\in V_{k-2}(\C), \end{equation}
are contained in the ring of integers $\mathcal{O}_{f}$ of the Hecke field $K_f$ of $f$.
\end{theorem}
The following result relates the coefficients of (\ref{eq:normalizedpoly}) with additive twist $L$-series. 
\begin{lemma}\label{lem:periodpoly}
We have for all integers $0\leq m\leq k-2$ that 
\begin{align}\label{eq:periodpol}
 \int_{a/q}^{i\infty} f(z) z^m dz=\sum_{j=0}^m \binom{m}{j} \left(\tfrac{a}{q}\right)^{m-j} (-2\pi i)^{-j-1} \Gamma(j+1)L(f,\tfrac{a}{q} ,j+1).
\end{align}
\end{lemma}
\begin{proof}
This follows directly by calculating the contour integral by writing $z=a/q+iy$, inserting the Fourier expansion of $f$ and interchanging sum and integral. For details see the proof of \cite[Lem. 2.2]{Nordentoft20.2}.  
\end{proof} 
\begin{corollary}
Let $\Omega^\pm$ be as in Theorem \ref{thm:rational}. For $0\leq j\leq k-2$, we have 
$$\frac{\Gamma(k-1)\Gamma(j+1)}{2(2\pi i)^{j+1}\Omega^\pm}\left(L(f,\tfrac{a}{q} ,j+1)\pm(-1)^{k/2+j-1} L(f,-\tfrac{a}{q} ,j+1)\right)\in \tfrac{1}{q^j} \Oo_f.$$
\end{corollary}
\begin{proof}
The coefficient of $X^{m}$ for $0\leq m\leq k-2$ in  the polynomial $\phi_f^\pm(\{\tfrac{a}{q},i\infty\})$ is equal to
\begin{equation}\label{eq:sinceexcept} \binom{k-2}{m}\frac{\int_{a/q}^{i\infty} f(z) z^{m} dz\pm (-1)^{k/2-1} \int_{-a/q}^{i\infty} f(z) z^{m} dz}{2}. \end{equation}
Thus by Theorem \ref{thm:rational} and Lemma \ref{lem:periodpoly} applied with $m=0$, we see that indeed
$$\frac{1}{4\pi i\Omega^\pm}\left(L(f,\tfrac{a}{q} ,1)\pm(-1)^{k/2-1} L(f,-\tfrac{a}{q} ,1)\right)\in \mathcal{O}_{f}.$$ 
Now the  rationality result  follows inductively for all $0\leq m\leq k-2$ by applying Lemma \ref{lem:periodpoly} and the expression (\ref{eq:sinceexcept}).
\end{proof}
\subsection{Theta Elements} We will now restrict to the case where the weight $k\in 2\Z_{\geq 0}$ is even and consider the central value $s=k/2$ (which is critical since $k$ is even) and construct associated \emph{theta elements}, as considered in e.g. \cite[Def. 3.1]{MazurRubin21}, satisfying norm relations and an interpolation property. 

For each choice of sign $\pm$, we define a \emph{plus/minus period of $f$} to be a complex number $\Omega^\pm_f\in (2\pi i)^{k/2}\Omega^\pm K_f$ such that 
\begin{equation}\label{eq:normaddtwist} L_f^\pm(\tfrac{a}{q}):=\frac{1}{\Omega_f^\pm}\left(L(f,\tfrac{a}{q} ,k/2)\pm L(f,-\tfrac{a}{q} ,k/2)\right)\in \tfrac{1}{q^{k/2-1}}\Oo_{f}, \end{equation} 
for all reduced fractions $\tfrac{a}{q}$ with $(q,N)=1$. In the case where $f=f_E$ is of weight $2$ corresponding to an elliptic curve $E/\Q$ via modularity, we can express the additive twist $L$-series in terms of modular symbols:
$$L(f_E,\tfrac{a}{q},1)= -2\pi i \int_{a/q}^{i\infty} f_E(z)dz.$$
In this case, we define the following plus/minus periods of $f_E$: 
\begin{equation}\label{eq:neron}\Omega_{f_E}^\pm:=\frac{2\Omega_E^\pm}{c_E\cdot (\#E_\mathrm{tor}(\Q))},\end{equation} 
where $c_E\in \Z$ denotes the Manin constant of $E$ and $\Omega_E^\pm$ denotes, respectively, the real and imaginary N\'{e}ron period of $E$. That these periods satisfy the integrality condition (\ref{eq:normaddtwist}) follows from \cite[Eq. (29)]{Man72} (note that there it is assumed that the Manin constant is $\pm 1$), see also the very enlightening discussion in Chapter 6 of \emph{loc.\ cit.}. This implies that in terms of the (plus/minus) modular symbols (\ref{eq:modularsymbol}) defined in the introduction we have: 
\begin{equation}\label{eq:addtwistMS} \left\langle \tfrac{a}{q} \right\rangle_E^\pm=\frac{1}{2\Omega_{E}^\pm}(L(f_E,\tfrac{a}{q},1)\pm L(f_E,-\tfrac{a}{q},1))=\frac{L_{f_E}^\pm(\tfrac{a}{q})}{c_E\cdot(\#E_\mathrm{tor}(\Q))}. \end{equation}
 
Let $\ell_1,\ldots, \ell_n$ be distinct primes not dividing $N$. For a subset $A\subset \{1,\ldots,n\}$ we put $$L_A:=\prod_{i\in A}\ell_i,\qquad L_n:=L_{\{1,\ldots,n\}}=\prod_{i=1}^n\ell_i.$$
Let $R=\Oo_f\left[\tfrac{1}{L_n^{k/2-1}}\right]$. Then we get an element of the group ring $R[(\Z/L_n)^\times]$ defined from the additive twist $L$-series (\ref{eq:normaddtwist});
\begin{equation}\label{eq:normrelationstheta} \theta_{f,L_n}^\pm :=\sum_{a\in (\Z/L_n)^\times} L_f^\pm(\tfrac{a}{L_n}) [a], \end{equation}
which by the above satisfies the following properties:
\begin{corollary}\label{cor:normrelationstheta}
Let $\ell_1,\ldots, \ell_n$ be distinct primes not dividing the level $N$ of $f$ and let $\theta_{f,L_n}^\pm$ be as above. Let $\chi\modulo L_n$ be a primitive Dirichlet character satisfying $\chi(-1)=\pm 1$. Then we have the following interpolation property:
 $$\theta_{f,L_n}^\pm(\chi)= \frac{\tau(\overline{\chi})L(f,\chi,k/2)}{\Omega_f^\pm}.$$ 
Furthermore, for a subset $A\subset \{1,\ldots,n\}$ the following norm relation with respect to the natural projection $ \pi_A:(\Z/L_n)^\times\twoheadrightarrow (\Z/L_A)^\times$  holds:
 \begin{equation}\label{eq:finalnormrel} \pi_A(\theta_{f,L_n}^\pm)=\prod_{i\in \{1,\ldots, n\}\setminus A}\biggr(a_f(\ell_i)\ell_i^{-(k-2)/2}-[\ell_i]-\epsilon_f(\ell_i)[\overline{\ell_i}]\biggr) \theta_{f,L_A}^\pm,\end{equation}
where $[\ell_i]\in \Z[(\Z/L_A)^\times]$ denotes the basis element corresponding to $(\ell_i\modulo L_A)\in  (\Z/L_A)^\times$ and $[\overline{\ell_i}]$ denotes its inverse.
\end{corollary}
\begin{proof}
The interpolation property follows directly from the Birch--Stevens formula (Proposition \ref{cor:BirchStevens}) and the definition of $L^\pm_f$ as in (\ref{eq:normaddtwist}). The norm relation follows by applying the general norm relations for additive twist $L$-series as in Proposition \ref{prop:norm-relations}. By induction we may reduce to the case $A=\{1,\ldots, n-1\}$. In this case the norm relations (\ref{eq:hornormrel1}) imply that for $a_0\in (\Z/L_{n-1})^\times $ we have
\begin{align} 
\pi_{\{1,\ldots, n-1\}}(\theta_{f,L_n}^\pm)(\delta_{a_0}) & =\sum_{\substack{a\in (\Z/L_{n})^\times:\\ a\equiv a_0\modulo L_{n-1}}}L^\pm_f(\tfrac{a}{L_n})\\ & = a_f(\ell_n) \ell_n^{-(k-2)/2} L^\pm_f(\tfrac{a_0}{L_{n-1}})-L^\pm_f(\tfrac{\overline{\ell_n} a_0}{L_{n-1}})-\epsilon_f(\ell_n) L^\pm_f(\tfrac{\ell_n a_0}{L_{n-1}})\\ & =\left((a_f(\ell_n)\ell_n^{-(k-2)/2}-[\ell_n]-\epsilon_f(\ell_n)[\overline{\ell_n}])\theta_{f,L_{n-1}}^\pm\right)(\delta_{a_0}),
\end{align}
where $\delta_{a_0}$ denotes the indicator function of $a_0\in (\Z/L_{n-1})^\times$, as wanted.
\end{proof} 
\section{Orderly Primes}\label{sec:TW}
In this section we will study a certain congruence condition on the Fourier coefficients of holomorphic modular forms related to the Euler factors appearing in the norm relations (\ref{eq:finalnormrel}) being invertible. The relevant condition turns out to be exactly that the image of Frobenius is fix-point free or equivalently that 1 is not a root of the Hecke polynomial.

It what follows, we let $f$ be a holomorphic newform of even weight $k$, level $N$, nebentypus $\epsilon_f$, and Hecke field $K_f$ with ring of integers $\Oo_f$. Let $p$ be a rational prime and let $\lambda|p$ be a place of $K_f$ above $p$ and denote by $\mathcal{O}_{f,\lambda}$ the valuation ring of the completion $K_{f,\lambda}$ of $K_f$ at $\lambda$. 
\begin{definition}[Orderly primes]\label{def:TW}
   Let $m\geq 1$ be an integer. We define the set of \emph{orderly primes for $f$ modulo $\lambda$ of order $m$} as: 
\begin{equation}\label{eq:TWdef}
\mathcal{P}_\mathrm{or}(f;\lambda,m):=\{\ell\text{ prime}: \ell\nmid N, p^m|\ell-1, a_f(\ell)-1-\epsilon_f(\ell)\in (\mathcal{O}_{f,\lambda})^\times\}. 
\end{equation}
We say that $\lambda^m$ is \emph{$f$-good} if $\mathcal{P}_\mathrm{or}(f;\lambda,m)$ has positive density among all primes and that $p^m$ is \emph{$f$-good} if $\lambda^m$ is $f$-good for some place $\lambda|p$ of $K_f$. If $f=f_E$ corresponds to an elliptic curve $E/\Q$ via modularity we put $\mathcal{P}_\mathrm{or}(E;p,m):=\mathcal{P}_\mathrm{or}(f_E;p,m)$ and say that $p^m$ is \emph{$E$-good} if it is $f_E$-good. For brevity, we write $\mathcal{P}_\mathrm{or}(f;\lambda)=\mathcal{P}_\mathrm{or}(f;\lambda,1)$ and $\mathcal{P}_\mathrm{or}(E;\lambda)=\mathcal{P}_\mathrm{or}(E;\lambda,1)$.
\end{definition}
 Let $G_\Q=\Gal(\overline{\Q}/\Q)$ denote the absolute Galois group of $\Q$ and denote by 
$$\rho_{f,\lambda}: G_{\Q}\rightarrow \GL_2(\Oo_{f,\lambda}),$$ 
the $\lambda$-adic representation associated with $f$ as defined by Deligne (for details consult e.g. \cite{Shimura}). Denote by $\F_\lambda$ the residue field of $\mathcal{O}_{f,\lambda}$ and let  
$$\overline{\rho}_{f,\lambda}: G_{\Q}\rightarrow \GL_2(\F_\lambda),$$
denote the corresponding residual representation characterized up to semi-simplification by the properties
\begin{align}\Tr(\overline{\rho}_{f,\lambda}(\Frob_\ell))=a_f(\ell), \quad \det(\overline{\rho}_{f,\lambda}(\Frob_\ell))=\epsilon_f(\ell)\ell^{k-1}\quad \text{in }\F_\lambda, \end{align}
for primes $(\ell, pN)=1$. Note that if the nebentypus is trivial and $(p-1,k-1)=1$, then the restriction of the determinant map $\det:\GL_2(\F_\lambda)\rightarrow \F_\lambda^\times $ to $\im( \overline{\rho}_{f,\lambda})$ maps surjectively onto $\F_p^\times\subset \F_\lambda^\times$. 

First of all we notice that the set of orderly primes modulo $\lambda|p$ essentially does not change upon twisting by a $p$-power order character. Recall the definition from (\ref{eq:twistedform}) of the twisted newform $f\otimes \chi$ with $\chi $ a Dirichlet character.
 \begin{lemma}\label{lem:twistingTW} Let $\chi$ be a Dirichlet character of order $p^j$ for some $j\geq 1$. Let $\lambda|p$ be a place of $K_f(\chi)$ above $p$ and let $\lambda_1,\lambda_2$ be the places of respectively $K_f$ and $K_{f\otimes \chi}$ such that $\lambda|\lambda_i,i=1,2$. Then for every $m\geq 1$, the two sets $\mathcal{P}_\mathrm{or}(f;\lambda_1,m)$ and $\mathcal{P}_\mathrm{or}(f\otimes \chi;\lambda_2,m)$ differ by finitely many primes.
 \end{lemma}
 \begin{proof}
 We observe that $\chi(\ell)\equiv 1\modulo \lambda$ since $\lambda$ lies above $p$.  Combining this with the expressions (\ref{eq:twistedform}) we get: 
$$a_{f\otimes \chi}(\ell)\equiv a_{f}(\ell)\mod \lambda,\quad  \epsilon_{f\otimes \chi}(\ell)\equiv \epsilon_{f}(\ell)\mod \lambda.$$ 
Thus a prime $\ell$ can belong to one set of orderly primes and not the other only if $\ell|N_{f\otimes \chi}N_f$. This yields the wanted conclusion. 
 \end{proof}

Under certain mild assumption we obtain the following criterion in terms of the Galois representation.
\begin{lemma}\label{lem:groupTW}
Let $f$ be a newform of even weight $k$ with Hecke field $K_f$. Let $p$ be a (rational) prime and let $\lambda|p$ be a place of $K_f$ above $p$. 
Assume that $(k-1,p-1)=1$ and that $\epsilon_f\equiv 1\modulo \lambda$.  Then the set of orderly primes $\mathcal{P}_\mathrm{or}(f;\lambda)$ is infinite if and only if  $\im (\overline{\rho}_{f,\lambda})\cap \SL_2(\F_\lambda)$ is not contained in the set of  unipotent matrices. If $\mathcal{P}_\mathrm{or}(f;\lambda)$ is infinite then it has positive density among all primes given by:
$$\frac{\#\{g\in \im (\overline{\rho}_{f,\lambda})\cap \SL_2(\F_\lambda): \tr(g)\neq 2 \}}{\#\im (\overline{\rho}_{f,\lambda})}.$$   
\end{lemma}  
\begin{proof}
Since  $(p-1,k-1)=1$ and $\epsilon_f\equiv 1\modulo \lambda$, it follows that a prime $\ell\nmid Np$ satisfies $\det(\overline{\rho}_{f,\lambda}(\Frob_\ell))=1$ iff $\ell\equiv 1\modulo p$. Thus by the Chebotarev Density Theorem the primes $\mathcal{P}_\mathrm{or}(f;\lambda)$ have positive (natural) density unless all elements of $\im( \overline{\rho}_{f,\lambda})\cap \SL_2(\F_\lambda)$ are unipotent matrices since these are exactly the matrices with determinant equal to $1$ and trace equal to $2$. Conversely, if $\im (\overline{\rho}_{f,\lambda})\cap \SL_2(\F_\lambda)$ is contained in the set of unipotent matrices then $\mathcal{P}_\mathrm{or}(f;\lambda)$ is finite.
\end{proof}
\subsection{Group Theoretic Results}  
The above considerations allow us to reduce (in certain cases) the question of infinitude of orderly primes to a purely group theoretic statement. In this section we will prove certain useful group theoretic results. Given an element of a finite group $g\in G$, we will denote by $|g|=\min\{n\geq 1: g^n=1\}$ it's order. We will be relying on Goursat's Lemma in the following form, see \cite[Theorem 5.5.1]{Hall18}.
\begin{lemma}[Goursat's Lemma]\label{lem:goursat}
Let $G_1,G_2$ be finite groups. The set of subgroups $H\leq G_1\times G_2$ which project surjectively onto each of the two factors are in one-to-one correspondence with triples $(N_1,N_2,\phi)$ where $N_i\leq G_i$ is a normal subgroup and $\phi:G_1/N_1\cong  G_2/N_2$ is an isomorphism. Under this correspondence we have $N_1\times N_2\leq H$.  
\end{lemma}
From this we will prove the needed results.
\begin{lemma} \label{lem:grouporder}
Let $p$ be a prime and $q=p^f$ with $f\geq 1$. Consider an element  $g\in \GL_2(\F_{q})$. Then $p$ dividing $|g|$ implies that $g$ is of the form $\lambda u$ where $\lambda\in \F_{q}^\times$ is a scalar and $u$ is unipotent. In particular, if $g\in \SL_2(\F_{q})$ and $p$ divides $|g|$ then $g=\pm u$ with $u$ unipotent and $g$ has trace $\pm 2$. 
\end{lemma}
\begin{proof}
Let $G$ be as in the statement. Let $\lambda_1,\lambda_2\in \F_{p^{2f}}$ be the eigenvalues of $g$. When $\lambda_1=\lambda_2=\lambda$ we get the desired conclusion by the Jordan normal form. Furthermore, if $g\in \SL_2(\F_{q})$ then this implies $\lambda^2=1$ and thus $\lambda\in\{\pm 1\}$ (note that in this case $|g|\in\{1,p,2p\}$). If $\lambda_1\neq \lambda_2$, we get by the Jordan normal form that the order of $g$ is the least common multiple of the order of $\lambda_1$ and $\lambda_2$ in the multiplicative group $\F_{p^{2f}}^\times$. This implies that the order of $g$ is coprime to $p$. This yields the wanted conclusion.  
\end{proof}
\begin{lemma} \label{lem:groupborel}
Let $p$ be a prime and $q=p^f$ with $f\geq 1$. Let $G$ be a subgroup of $\GL_2(\F_{q})$ such that the determinant map surjects onto $\F_p^\times\subset \F_{q}^\times$ and such that all elements of $G\cap \SL_2(\F_{q})$ are unipotent.  Then $G$ is contained in a Borel subgroup of $\GL_2(\F_q)$.
\end{lemma} 
\begin{proof}
Given two unipotent elements $g_1,g_2\in\SL_2(\F_{q})$ which are not contained in the same unipotent we can pick a basis (a priori over $\overline{\F}_p$) such that $g_1=\begin{psmallmatrix} 1 & x\\ 0& 1\end{psmallmatrix}$ and $g_2=\begin{psmallmatrix} 1 & 0\\ y& 1\end{psmallmatrix}$ with $xy\neq 0$. Now it follows that $g_1g_2=\begin{psmallmatrix} 1+xy & 1+x\\ 1+y& 1\end{psmallmatrix}$ is not unipotent. Thus we conclude that $G \cap \SL_2(\F_{q})$ is contained in a fixed unipotent subgroup $U=\{\begin{psmallmatrix}1& \ast \\ 0 & 1 \end{psmallmatrix}\}$.

If $G$ contains an element $g=\begin{psmallmatrix} a & b\\ c & d \end{psmallmatrix}$ with $c\neq 0$, meaning that $g$ is not contained in the Borel containing $U$, then for any $u=\begin{psmallmatrix} 1 & x\\ 0 & 1 \end{psmallmatrix}\in U$ the lower-left entry of $gug^{-1}\in \SL_2(\F_{q})$ is given by $-c^2x\det(g)^{-1}$. Thus we conclude by the above that $G \cap \SL_2(\F_{q})=\{I\}$.  In this case, by the First Isomorphism Theorem, the determinant defines an isomorphism $G\cong \F_p^\times$. Let $G$ be generated by $g\in \GL_2(\F_{q})$ with eigenvalues $\lambda_1,\lambda_2\in \overline{\F}_p$. Since $g$ has order $p-1$ we conclude that  
$$\lambda_1^{p-1}=\lambda_2^{p-1}=1,$$
which implies that $\lambda_1,\lambda_2\in \F_p$, and consequentely, by the Jordan normal form, $g$ is conjugate over $\GL_2(\F_{q})$ to an upper triangular matrix. Thus $G=\langle g\rangle$ is contained in a Borel in this case as well.       
\end{proof}
\begin{lemma}\label{lem:grouppm}
Let $p$ be a prime and $n\geq 1$ an integer. For $i=1,\ldots, n$ let $q_i=p^{f_i}$ be a power of $p$ with $f_i\geq 1$. Let $G$ be a subgroup of $\prod_{i=1}^n\SL_2(\F_{q_i})$ containing an element such that each component has trace not equal to $2$. Let $H$ be a subgroup of $G\times \Z/p^m$ which projects surjectively to each of the factors $G$ and $\Z/p^m$. Then $H$ contains an element of the form $(g_1,\ldots, g_n,0)$ where $\tr (g_i)\neq 2$ for $i=1,\ldots, n$.
\end{lemma}
\begin{proof}
By Goursat's Lemma above we see that $H$ contains $N\times\{0\}$ where $N$ is a normal subgroup of $G$ and $G/N\cong \Z/p^{m'}$ with $m'\leq m$. Given an element $g\in G$ with $(|g|,p)=1$ we see that $((g,0)N)^{|g|}=N$ implies that $(g,0)\in N$. Now let $g=(g_1,\ldots, g_n)\in G$ be such that $\tr(g_i)\neq 2$. Put $q=\max\{q_1,\ldots, q_n\}$. Then $(|g^{q}|,p)=1$ since $q_i$ divides $\# \SL_2(\F_{q_i})=(q_i^2-1)(q_i^2-q_i)$ exactly once. Thus we conclude that $g^{q}\in N$. Assume that there is some $1\leq i\leq n$ such that $\tr((g_i)^{q})=2,i=1,\ldots, n$. Then the eigenvalues $\lambda_1,\lambda_2\in \F_{q^2}$ of $g_i$ satisfy 
$\lambda_1^{q}=\lambda_2^{q}=1$ (since $\det(g_i)=1$) which implies that $(\lambda_1-1)^{q}=(\lambda_2-1)^{q}=0$ and thus $\lambda_1=\lambda_2=1$. This contradicts the assumption on $g=(g_1,\ldots, g_n)$ which yields the desired conclusion.
\end{proof}
\begin{lemma}\label{lem:grouplem2}
Let $p$ be prime and $d\geq 1$ a positive integer such that $(d,p)=1$. Let $G$ be a finite group and let $H$ be a subgroup of $G\times \Z/d\times \Z/p^m$ which projects surjectively onto both $G\times \Z/d$ and $\Z/p^m$. Then $H$ contains $\{1_G\}\times \Z/d\times \{0\}$. 
\end{lemma}
\begin{proof}
By Goursat's Lemma we see that $H$ contains $N\times\{0\}$ where $N$ is a normal subgroup of $G\times \Z/d$ and $(G\times \Z/d)/N\cong \Z/p^{m'}$ with $m'\leq m$. Since $(1_G,x)\in G\times \Z/d$ has order coprime to $p$ we conclude as above that $(1_G,x)\in N$ which yields the desired conclusion.  
\end{proof}

\subsection{Criteria for Infinitude of Orderly Primes}
From the above we will give conditions which imply infinitude of orderly primes. Throughout we will fix a (rational) prime $p$ and a holomorphic newform $f$ of even weight $k$, level $N$, nebentypus $\epsilon_f$ and with Hecke field $K_f$, as well as a place $\lambda|p$ of $K_f$ above $p$. \begin{corollary}\label{cor:galpm}
Assume that $(p-1,k-1)=1$ and that $\epsilon_f\equiv 1\modulo \lambda$. If $\lambda$ is $f$-good, then $\lambda^m$ is $f$-good for all $m\geq 1$. Furthermore, for $\lambda|2$ we have
\begin{equation}\label{eq:TW2}\mathrm{d}(\mathcal{P}_\mathrm{or}(f;\lambda,m))\geq \frac{\mathrm{d}(\mathcal{P}_\mathrm{or}(f;\lambda))}{2^{m-1}},\end{equation}
where  $\mathrm{d}(\mathcal{P}_\mathrm{or}(f;\lambda,m))$ denotes the (natural) density of the orderly primes for $f$ modulo $\lambda$ of order $m$.
\end{corollary}
\begin{proof}
By Lemma \ref{lem:groupTW}, $\mathcal{P}_\mathrm{or}(f;\lambda)$ having positive density is equivalent to $\im(\overline{\rho}_{f,\lambda})\cap \SL_2(\F_\lambda)$ containing some non-unipotent element (i.e. an element with trace not equal to $2$). Consider the Galois representation 
$$\rho:=\overline{\rho}_{f,\lambda}\oplus \chi_{\mathrm{cyc},p^m}:G_\Q\rightarrow \GL_2(\F_\lambda)\times (\Z/p^m)^\times,$$  
where $\chi_{\mathrm{cyc},p^m}$ denotes the  cyclotomic character of conductor $p^m$. Then for a prime $\ell$ with $(\ell,Np)=1$ and an integer $1\leq j\leq m$ we have
$$\rho(\Frob_\ell)\in \SL_2(\F_\lambda)\times ((1+p^j\Z)/p^m\Z)^\times,$$ if and only if $\ell\equiv 1\modulo p^j$. Thus the result follows from Lemma \ref{lem:grouppm} since we have a natural isomorphism $((1+p^j\Z)/p^m\Z)^\times\cong \Z/p^{m-j}$.

Now assume that $p=2$ and let 
$$H=\im(\overline{\rho}_{f,\lambda}\oplus \chi_{\mathrm{cyc},2^m})\subset \SL_2(\F_\lambda)\times (\Z/2^m)^\times.$$
Let $N_1\leq \im(\overline{\rho}_{f,\lambda})$ and $N_2\leq (\Z/2^m)^\times$ be such that $N_1\times N_2\leq H$ as in Goursat's Lemma. Then since $\im(\overline{\rho}_{f,\lambda})/N_1\cong (\Z/2^m)^\times/N_2$, we conclude that $N_1$ contains all elements of $\im(\overline{\rho}_{f,\lambda})$ with odd order. By Lemma \ref{lem:grouporder} an element of $\SL_2(\F_\lambda)$ has even order exactly if it is unipotent which implies (by the Chebotarev Density Theorem) that 
$$\#\{g\in N_1:  \tr(g)\neq 0\}= \#\{g\in \im(\overline{\rho}_{f,\lambda}):\tr(g)\neq 0\}=\mathrm{d}(\mathcal{P}_\mathrm{or}(f;\lambda))\cdot \#\im(\overline{\rho}_{f,\lambda}).$$ 
This yields (by the Chebotarev Density Theorem) that 
$$ \mathrm{d}(\mathcal{P}_\mathrm{or}(f;\lambda,m))= \frac{\#\{g\in N_1: \tr(g)\neq 0\}}{\#\im(\overline{\rho}_{f,\lambda}\oplus \chi_{\mathrm{cyc},2^m})} \geq \frac{\mathrm{d}(\mathcal{P}_\mathrm{or}(f;\lambda))\cdot \#\im(\overline{\rho}_{f,\lambda})}{\#\im(\overline{\rho}_{f,\lambda})\cdot \#\im(\chi_{\mathrm{cyc},2^m})}=\frac{\mathrm{d}(\mathcal{P}_\mathrm{or}(f;\lambda))}{2^{m-1}}, $$
as desired. 
\end{proof}
Secondly we obtain that irreducibility of the residual representation suffices for the existence of infinitely many orderly primes.
\begin{corollary}\label{cor:irred}
Assume that $(p-1,k-1)=1$, $\epsilon_f\equiv 1\modulo \lambda$, and that the residual representation of $f$ at $\lambda$ is irreducible. Then for $m\geq 1$ we have that $\mathcal{P}_\mathrm{or}(f;\lambda,m)$ has positive density among all primes.    
\end{corollary}
\begin{proof}
This follows directly from  Lemma \ref{lem:groupTW}, Lemma \ref{lem:groupborel} and Corollary \ref{cor:galpm}.
\end{proof}
In the reducible case we can give a precise characterization without any further assumptions.
\begin{proposition} 
Assume that $\overline{\rho}_{f,\lambda}$ is reducible with semi-simplification equal to $\F_\lambda[\chi_{1}]\oplus \F_\lambda[\chi_{2}]$ for some characters $\chi_{1},\chi_{2}:G_\Q\rightarrow \F_\lambda^\times$ satisfying $\chi_{1}\chi_{2}=\epsilon_f(\chi_{\mathrm{cyc},p})^{k-1}$.

Then $\mathcal{P}_\mathrm{or}(f;\lambda)$ has positive density if and only if neither $\chi_1,\chi_2$ is a power of the cyclotomic character $\chi_{\mathrm{cyc},p}$. Furthermore, if $(N,p)=1$ then $\mathcal{P}_\mathrm{or}(f;\lambda)$ has positive density if and only if 
$$\{\chi_1,\chi_2\}\notin \{\{1,(\chi_{\mathrm{cyc},p})^{k-1}\epsilon_f\}, \{(\chi_{\mathrm{cyc},p})^{k-1},\epsilon_f\}\}.$$ 
\end{proposition}
\begin{proof}
For $\ell\equiv 1 \modulo p$ we have 
\begin{equation}\label{eq:reduTW}1+\epsilon_f(\ell)-\Tr(\overline{\rho}_{f,\lambda}(\Frob_\ell))=(1-\chi_{1}(\ell))(1-\chi_{2}(\ell))\quad \text{in }\F_\lambda.\end{equation}
Thus we see that if neither  $\chi_{1},\chi_{2}$ factor through $\Gal(\Q(\zeta_p)/\Q)$, then for a positive proportion of primes $\ell\equiv 1\modulo p$ the above expression (\ref{eq:reduTW}) is a $\lambda$-adic unit. If on the other hand, $\chi_{1}$ or $\chi_{2}$ factor through $\Gal(\Q(\zeta_p)/\Q)$ then (\ref{eq:reduTW}) is never a $\lambda$-adic unit.  This implies the first statement. Furthermore, if we assume that $(N,p)=1$ then it follows from a result of Deligne (see the introduction of \cite{Gross90}) that at most one of the conductors of $\chi_{1},\chi_{2}$ is divisible by $p$. Combining this with the fact that $\chi_1\chi_2=\epsilon_f(\chi_{\mathrm{cyc},p})^{k-1}$ yields the desired conclusion by the previous result. 
\end{proof}
In particular we see from the above that there do exist cases where the residual representation $\overline{\rho}_{f,\lambda}$ at $\lambda$ is reducible and $\mathcal{P}_\mathrm{or}(f;\lambda)$ has positive density. 

\subsection{Joint Orderly Primes}
Let $p^m$ be a prime power and let $f_1,\ldots, f_r$ be newforms of level $N_i$, with nebentypus $\epsilon_i$ and with Hecke field $K_i$, respectively. Denote by $\Oo$ the ring og integers of compositum of Hecke fields $K_{1}\cdots K_{r}$ and consider a place $\lambda|p$ of $K_{1}\cdots K_{r}$  above $p$. Then we can consider the set of ``joint orderly primes''.
\begin{definition}[Joint orderly primes]\label{def:TWjoint}
We define the set of \emph{joint orderly primes for $f_1,\ldots, f_r$ modulo $\lambda$ of order $m$} as;
\begin{align}\mathcal{P}_\mathrm{or}(f_1,\ldots, f_r;\lambda,m):=\{\ell\text{ prime}: p^m|\ell-1,\ell\nmid N_i, a_{f_i}(\ell)- 1-\epsilon_{i}(\ell)\in (\Oo_{\lambda})^\times \forall i\}.\end{align}
We say that $\lambda^m$ is \emph{$(f_1,\ldots, f_r)$-good} if $\mathcal{P}_\mathrm{or}(f_1,\ldots, f_r;\lambda,m)$ has positive density among all primes and that $p^m$ is \emph{$(f_1,\ldots, f_r)$-good} if $\lambda^m$ is $(f_1,\ldots, f_r)$-good for some $\lambda|p$.
\end{definition}

First of all, we have have the following generalization of Corollary \ref{cor:galpm}.
\begin{corollary}\label{cor:galpmjoint}
Let $f_1,\ldots, f_r$ be holomorphic forms with respective even weights $k_i$, nebentypus $\epsilon_i$ and Hecke fields $K_i$. Let $\lambda|p$ be a place of $K_1\cdots K_r$ above $p$. Assume that $(p-1,k_i-1)=1$ and that $\epsilon_i\equiv 1\modulo \lambda$ for $i=1,\ldots,r$. If $\lambda$ is $(f_1,\ldots, f_r)$-good then $\lambda^m$ is $(f_1,\ldots, f_r)$-good for all $m\geq 1$.
\end{corollary}
\begin{proof}
    The corollary follows by the same argument as in the proof of  Corollary \ref{cor:galpm} where now we are applying Lemma \ref{lem:grouppm} to the Galois representation $\oplus_{i=1}^r \overline{\rho}_{f_i,\lambda_i}\oplus \chi_{\mathrm{cyc},p^m}$ where $\lambda_i$ is a place of $K_i$ such that $\lambda|\lambda_i$. We will skip the details.
\end{proof}

We will now fix newforms $f_1,\ldots, f_r$, twist them by some finite-order character $\chi$ and then consider joint orderly primes for the twisted forms $f_1\otimes \chi,\ldots, f_r\otimes \chi$. We start by the following lemma.  
\begin{lemma}\label{lem:disjointGalrep}
Let $G_1,G_2$ be finite groups. For $i=1,2$, let $\rho_i: G_{\Q}\rightarrow G_i$ be Galois representations  with disjoint ramifications. Then $\im (\rho_1\oplus \rho_2)=\im (\rho_1)\times \im(\rho_2)$. 
\end{lemma}
\begin{proof}
For $i=1,2$, let $K_i$ be the number field cut out by $\rho_i$, meaning that 
\begin{equation}\label{eq:galoisdisj}\Gal(K_i/\Q)\cong G_\Q/\ker(\rho_i)\cong \im (\rho_i),\quad i=1,2.\end{equation}
Recall that a prime $\ell$ divides the discriminant of $K_i$ exactly if $\rho_i$ is ramified at $\ell$. Thus we conclude that $(\disc(K_1),\disc(K_2))=1$ and thus $K_1\cap K_2=\Q$ (since $\disc(K_1\cap K_2)$ divides the discriminant of both fields $K_1,K_2$ and the class number of $\Q$ is one). By Galois Theory we have that 
$$\im (\rho_1\oplus \rho_2)\cong G_\Q/\ker(\rho_1\oplus \rho_2)= G_\Q/\ker(\rho_1)\cap\ker(\rho_2)\cong \Gal(K_1K_2/\Q).$$
Since $K_1,K_2$ are linearly disjoint we conclude that $\Gal(K_1K_2/\Q)\cong \Gal(K_1/\Q)\times \Gal(K_2/\Q)$. This yields the desired conclusion by the isomorphisms (\ref{eq:galoisdisj}).  
\end{proof}
We are now ready to prove a key result which shows that the orderly condition is satisfied under mild assumptions when considering twists of fixed forms. 
\begin{corollary}\label{cor:TWpropa}
Let $f_1,\ldots, f_r$ be newforms of even weight and level $N_i$, respectively. Let $\chi$ modulo $D$ be a primitive Dirichlet character of order $d\geq 2$.  Let $p$ be a prime number such that $(D,N_1\cdots N_r p)=1$ and $(d,p)=1$. 
Then $p^m$ is $(f_1\otimes \chi,\ldots ,f_r\otimes \chi)$-good for all $m\geq 1$.
\end{corollary}
\begin{proof}
Let $\lambda$ be a place of $K_{f_1\otimes \chi}\cdots K_{f_r\otimes \chi}$ above $p$ and for each $i=1,\ldots, r$ let $\lambda_i$ be a place of $K_{f_i}$ above $\lambda$. Put $N=\mathrm{lcm}(p^m, N_1,\ldots, N_r)$ and consider the Galois representation 
$$\rho:=\oplus_{i=1}^r\overline{\rho}_{f_i,\lambda_i}\oplus \chi_{\mathrm{cyc},N}\oplus \chi:G_\Q\rightarrow \prod_{i=1}^r\GL_2(\F_{\lambda_i})\times (\Z/N)^\times \times \Z/d,$$ 
where $\F_{\lambda_i}$ denote the residual field at $\lambda_i$ and as above $\chi_{\mathrm{cyc},N}$ denotes the cyclotomic character of conductor $N$.   Recall that the conductors of $\overline{\rho}_{f_i,\lambda_i}$ and $\chi$ are dividing, respectively, $pN_i$ and $D$. 
Thus we conclude by the assumption $(D,N_1\cdots N_rp)=1$ and Lemma \ref{lem:disjointGalrep} that $$\im(\rho)=\im(\oplus_{i=1}^r\overline{\rho}_{f_i,\lambda_i}\oplus \chi_{\mathrm{cyc},N})\times \Z/d.$$
In particular, the image of $\rho$ contains an element which is the identity in all but the last component and a generator of $\Z/d$ in the last component. It follows from the Chebotarev Density Theorem that for a positive proportion of primes $\ell$ we have 
$$a_{f_i}(\ell)\equiv 2\modulo \lambda_i,i=1,\ldots,r,\qquad  \chi(\ell) \text{ a primitive $d$th root of unity},\qquad \ell\equiv 1\modulo p^m.$$
Then by the above and the formulas (\ref{eq:twistedform}) it holds for a positive proportion of primes $\ell$ that 
$$ a_{f_i\otimes \chi}(\ell)-1-\epsilon_{f_i\otimes \chi}(\ell)\equiv 2\chi(\ell)-1-\chi(\ell)^2\equiv -(\chi(\ell)-1)^2 \modulo \lambda,\, i=1,\ldots, r.$$
Since $\prod_{\zeta\in \mu_d\setminus\{1\}}(1-\zeta)=d$ and $(d,p)=1$, it follows that $\chi(\ell)-1$ is a $p$-adic unit. This implies that 
$$a_{f_i\otimes \chi}(\ell)-1-\epsilon_{f_i\otimes \chi}(\ell)\in (\Oo_{\lambda})^\times,\, i=1,\ldots, r,$$
where $\Oo$ denotes the ring of integers of $K_{f_1\otimes \chi}\cdots K_{f_r\otimes \chi}$. 
This yields the desired conclusion.
\end{proof}

\section{Applications to Modular Forms}\label{sec:applicationstomodularforms}
In this section we will construct horizontal measures associated to holomorphic modular forms using the horizontal norm relations (Section \ref{sec:normrelQ}) at orderly primes (Section \ref{sec:TW}). Then we will apply general facts about horizontal measures (Section \ref{sec:horziontalmeasures}) to obtain various applications to non-vanishing of central $L$-values.
\subsection{Horizontal {\it p}-adic {\it L}-functions}\label{sec:applicationsQsection}
Let $p$ be a (rational) prime and let $f$ be a (holomorphic)  newform of even weight $k$ and level $N$ with Hecke field $K_f$. Let $\lambda|p$ be a place of $K_f$ above $p$ which defines an embedding $K_f\subset \C_p$ and let $R\subset \Oo_{\C_p}$ denote the valuation ring of the completion $K_{f,\lambda}$ of $K_f$ at $\lambda$.  Let $\mathcal{L}=(\ell_n)_{n\in \N}$ be a sequence of distinct primes congruent to $1$ modulo $p$ and  coprime to $N$. For a finite subset $A\subset \N$ we write
$$L_A:=\prod_{n\in A}\ell_n.$$
Recall that given a sign $\pm$ and a finite subset $A\subset \N$, we defined in equation (\ref{eq:normrelationstheta}) (using the normalized additive twist $L$-series as defined in equation (\ref{eq:normaddtwist})) a group ring element, referred to as a \emph{theta element} in \cite{MazurRubin21}, denoted by: 
$$\theta_{f,L_A}^\pm \in R\left[\prod_{n\in A}(\Z/\ell_n)^\times\right],$$ 
which by Corollary \ref{cor:normrelationstheta} satisfy the norm relations (\ref{eq:finalnormrel}). For $n\in \N$  let $m_n=\ev_p(\ell_n-1)$ and pick a  projection:
\begin{equation}\label{eq:biroot}\rho_{\ell_n}:(\Z/\ell_n)^\times\twoheadrightarrow \Z/p^{m_n}.\end{equation}
For $A\subseteq \N$ we write 
$$\rho_A:=\prod_{n\in A} \rho_{\ell_n}: \prod_{n\in A} (\Z/\ell_n)^\times\twoheadrightarrow \prod_{n\in A}\Z/p^{m_n},$$ 
and denote by
\begin{equation}\label{eq:tildetheta}\tilde{\theta}_{f,A}^\pm:= \rho_A(\theta_{f,L_A}^\pm)\in R\left[\prod_{n\in A}\Z/p^{m_n}\right],\end{equation}
the group ring element obtained by pushing forward the theta element along $\rho_A$. Observe that if $p>2$ then $\tilde{\theta}_{f,A}^-=0$ since all Dirichlet characters $\chi$ of odd order are even, i.e. $\chi(-1)=1$. 

Recall that for $A'\subset A\subset \N$ we denote by 
$$\pi_{A,A'}: \prod_{n\in A} (\Z/\ell_n)^\times\twoheadrightarrow \prod_{n\in A'} (\Z/\ell_n)^\times,\quad \tilde{\pi}_{A,A'}: \prod_{n\in A}\Z/p^{m_n}\twoheadrightarrow \prod_{n\in A'}\Z/p^{m_n}$$ 
the canonical projections. 

Then we have the following commutative diagram 
\begin{equation}\label{eq:digitdefinition}\begin{tikzcd}[column sep = large]
\prod_{n\in A'}(\mathbb{Z}/\ell_n)^\times  \arrow{r}{\rho_{A'}} & \prod_{n\in A'}\Z/p^{m_n} \\
\prod_{n\in A}(\mathbb{Z}/\ell_n)^\times \arrow{u}{\pi_{A,A'}} \arrow{r}{\rho_A} & \prod_{n\in A}\Z/p^{m_n}\arrow{u}{\tilde{\pi}_{A,A'}}
  \end{tikzcd}
  \end{equation}
Thus since pushforward defines a ring homomorphism, it follows from Corollary \ref{cor:normrelationstheta} that our new elements also satisfy norm relations;
\begin{align}
\tilde{\pi}_{A,A'}(\tilde{\theta}_{f,A}^\pm)&=\tilde{\pi}_{A,A'}(\rho_A(\theta_{f,A}^\pm))=\rho_{A'}(\pi_{A,A'}(\theta_{f,A}^\pm)) \\
&=\rho_{A'}\left(\left(\prod_{n\in A\setminus A'}(a_f(\ell_n)\ell_n^{-(k-2)/2}-[\ell_n]-\epsilon_f(\ell_n)[\overline{\ell_n}])\right)\theta_{f,L_{A'}}^\pm\right)\\ 
&\label{eq:pro_p_norm}=\left(\prod_{n\in A\setminus A'}(a_f(\ell_n)\ell_n^{-(k-2)/2}-[\sigma_{\ell_n,A'}]-\epsilon_f(\ell_n)[-\sigma_{\ell_n,A'}])\right)\tilde{\theta}_{f,A'}^\pm,\end{align}
where 
\begin{align}\label{eq:sigmaellnold}\sigma_{\ell_n,A'}:=\rho_{A'}((\ell_n\modulo \ell_{n'})_{n'\in A'})\in \prod_{n'\in A'}\Z/p^{m_{n'}}.\end{align} 
For each $n\in \N$  we denote by     $\sigma_{n}\in \prod_{i\in \N}\Z/p^{m_i}$ the element  characterized by the following property:
\begin{equation}\label{eq:sigmaelln}\tilde{\pi}_{\N,\{i\}}(\sigma_n) =\begin{cases}
    \rho_{\ell_{i}}(\ell_n \modulo \ell_{i}),& i\neq n,\\
    0,& i=n.
\end{cases}\end{equation} 
These are ``global lifts'' of the group ring elements appearing in the norm relations (\ref{eq:pro_p_norm}) in the sense that  $\tilde{\pi}_{\N,A}(\sigma_n)=\sigma_{\ell_n,A}$ for all finite subsets $A\subset \N$. Note here that this property does not depend on the $n$th component of $\sigma_n$ (which we have put equal to $0$). We obtain an element of $\mathrm{Hom}_\mathrm{cts}(\mathcal{C}(\prod_{i\in \N} \Z/p^{m_i},R), R)$  given by evaluation at $\sigma_n$ which in view of the identification  (\ref{eq:homcts}) defines an element of the horizontal Iwasawa algebra that we will denote by $[\sigma_n]\in  R\left\llbracket \prod_{i\in \N} \Z/p^{m_i} \right\rrbracket$.  
\begin{lemma}\label{lem:welldefinedness}
Let $\ell_n\in \mathcal{P}_\mathrm{or}(f;\lambda)$ be an orderly prime for $f$ modulo $\lambda$. Then 
\begin{equation}\label{eq:normrelinverting}a_f(\ell_n)\ell_n^{-(k-2)/2}-[\sigma_{n}]-\epsilon_f(\ell_n)[-\sigma_{n}]\in R\left\llbracket \prod_{i\in \N} \Z/p^{m_i} \right\rrbracket,\end{equation} 
is invertible, i.e. belongs to $R\left\llbracket \prod_{i\in \N} \Z/p^{m_i} \right\rrbracket^\times$. 
\end{lemma}
\begin{proof}Since $\ell_n\equiv 1\modulo p$ we have 
\begin{align*}
\biggr(a_f(\ell_n)\ell_n^{-(k-2)/2}-[\sigma_{n}]-\epsilon_f(\ell_n)[-\sigma_{n}]\biggr)(\1)&=a_f(\ell_n)\ell_n^{-(k-2)/2}-1-\epsilon_f(\ell_n)\\
&\equiv a_f(\ell_n)-1-\epsilon_f(\ell_n)\modulo p,  \end{align*}
which is a $\lambda$-unit by the 
 definition of orderly primes (see Definition \ref{def:TW}). Thus the invertibility follows directly from Proposition \ref{prop:invertible}.  \end{proof} 
We arrive at the following compatible system.
\begin{corollary}\label{cor:measure2}
Let $\pm$ be a sign and $r\geq 0$ an integer. Let $(\ell_n)_{n\in \N}$ be a sequence of distinct primes congruent to $1$ modulo $p$ and such that $\ell_n\in \mathcal{P}_\mathrm{or}(f,\lambda)$ for $n\geq r+1$. Then for $\{1,\ldots,r\}\subset A\subset \N$ the following group ring elements are well-defined:
\begin{equation}\label{eq:compatsys}\nu^\pm_{A}:=\left(\prod_{n\in A\setminus \{1,\ldots,r\}}\biggr(a_f(\ell_n)\ell_n^{-(k-2)/2}-[\sigma_{\ell_n,A}]-\epsilon_f(\ell_n)[-\sigma_{\ell_n,A}]\biggr)^{-1}\right)\tilde{\theta}_{f,A}^\pm\in R\left[\prod_{n\in A}\Z/p^{m_n}\right],\end{equation}
with $\sigma_{\ell_n,A}$ as in (\ref{eq:sigmaelln}) and $\tilde{\theta}_{f,A}^\pm$ as in (\ref{eq:tildetheta}). Furthermore, they satisfy 
$$\pi^\pm_{A,A'}(\nu^\pm_{A})=\nu^\pm_{A'},\quad \{1,\ldots,r\}\subset A'\subset A\subset \N.$$ 
\end{corollary}
\begin{proof}
That the elements are well-defined follows directly from Lemma \ref{lem:welldefinedness} above. The compatibility under the projection maps follows from the norm relations (\ref{eq:pro_p_norm}) and the general construction in Section \ref{sec:horconstr}.   
\end{proof}

By the above corollary we conclude that for $A\subset \N$ and $A\cup\{1,\ldots,r \}\subset B\subset \N$ the following element does not depend on $B$:
\begin{equation}\label{eq:nuAgen}\nu^\pm_A:=\tilde{\pi}_{B,A}(\nu^\pm_{B})\in R\left[\prod_{n\in A}\Z/p^{m_n}\right],\end{equation}
with $\nu^\pm_{B}$ as in (\ref{eq:compatsys}).
We summarize all of the above in the following key definition. 
\begin{definition}[Horizontal $p$-adic $L$-functions]\label{def:padicL}
   Let $f$ be a holomorphic newform of level $N$, even weight $k$ and with Hecke field $K_f$. Let $p$ be a prime number  and let $\lambda|p$ be a place of $K_f$ above $p$ which is $f$-good in the sense of Definition \ref{def:TW}. Let $R$ denote the valuation ring of the completion $K_{f,\lambda}$ of $K_f$ at $\lambda$. Let $r\geq 0$ be an integer and let $\mathcal{L}=(\ell_n)_{n\in \N}$ be a sequence of distinct primes not dividing $N$, congruent to $1$ modulo $p$, and such that $\ell_n\in \mathcal{P}_\mathrm{or}(f,\lambda)$ for $n\geq r+1$. For $n\in \N$ put   $m_{n}=\ev_p(\ell_n-1)$.
   
   Then we define the \emph{horizontal $p$-adic $L$-function of $f$ associated to $(\mathcal{L},r,\pm)$} as the pro-$p$ horizontal measure 
\begin{equation}\label{eq:horpadicL}\nu^\pm_{f,\mathcal{L},r}:=\varprojlim_{A\subset \N\,\mathrm{finite}}\nu^\pm_{A} \in  \Lambda^\mathrm{hor}=R\left\llbracket \prod_{n\in \N} \Z/p^{m_n} \right\rrbracket  ,
\end{equation}
  with $\nu_A^\pm$ given by (\ref{eq:compatsys}) and (\ref{eq:nuAgen}). For brevity, we denote $\nu_{f,\mathcal{L},r}:=\nu^+_{f,\mathcal{L},r}$. 
\end{definition}
Note that we are suppressing the dependence on the choice of projections (\ref{eq:biroot}) in the notation (see also Remark \ref{rem:abstract}). We emphasize that $\nu_{f,\mathcal{L},r}^-=0$ for $p>2$ since any Dirichlet character $\chi$ of odd order is even, i.e. $\chi(-1)=1$.   

The horizontal $p$-adic $L$-functions defined above interpolate $L$-values of twists of $f$ by characters of $p$-power order and with conductor coprime to $p$. The exact interpolation property is in terms of certain $L$-values modified at the primes dividing the conductor of the character. Let $\chi$ be a Dirichlet character with conductor $D_\chi$ dividing $\prod_{n\in \N}\ell_n$ and sign $\pm$, i.e. $\chi(-1)=\pm 1$. Then we define the \emph{modified $L$-value of $f$ twisted by $\chi$}: 
\begin{align}\nonumber L^\ast_f(\chi)&:=\left(\prod_{1\leq i\leq r: \ell_i\nmid D_\chi}\biggr(a_f(\ell_i)\ell_i^{-(k-2)/2}-\chi(\ell_i)-\epsilon_f(\ell_i)\overline{\chi}(\ell_i)\biggr)\right)\\
\label{eq:interpolationfinal} &\times \left(\prod_{i\geq r+1: \ell_i|D_\chi}\biggr(a_f(\ell_i)\ell_i^{-(k-2)/2}-\chi^{(i)}(\ell_i)-\epsilon_f(\ell_i)\overline{\chi}^{(i)}(\ell_i)\biggr)\right)^{-1} \tau(\overline{\chi})L(f,\chi,k/2)/\Omega_f^\pm,\end{align} 
   where $\chi^{(i)} \modulo \tfrac{D_\chi}{\ell_i}$ denotes the Dirichlet character obtained by restricting $\chi$ to a character modulo $\tfrac{D_\chi}{\ell_i}$ and the twisted $L$-function $L(f,\chi,s)$ is given by (\ref{eq:twistedLvalue}). Here we are suppressing the dependence on $\{\ell_1,\ldots, \ell_r\}$ in the notation. 
\begin{corollary}[Interpolation]\label{cor:measure}
Let $f$ be a newform of even weight $k$, level $N$ and nebentypus $\epsilon_f$. Let $\nu^\pm_{f,\mathcal{L},r}\in R\left\llbracket \prod_{n\in \N} \Z/p^{m_n} \right\rrbracket $ be a horizontal $p$-adic $L$-function of $f$ as in Definition \ref{def:padicL}. Let $\chi$ be a Dirichlet character of $p$-power order with conductor dividing $\prod_{n\in \N}\ell_n$. Write $\chi=\tilde{\chi}\circ \rho_\N$ in terms of the projection $\rho_\N: \prod_{n\in \N}(\Z/\ell_n)^\times \twoheadrightarrow \prod_{n\in \N}\Z/p^{m_n}$ and a character  $\tilde{\chi}:\prod_{n\in \N}\Z/p^{m_n}\rightarrow \C_p^\times$. Let $\pm$ denote the sign of $\chi$, i.e. $\chi(-1)=\pm 1$. Then we have \begin{equation}\label{eq:interpolationmeasure}\nu^\pm_{f,\mathcal{L},r}(\tilde{\chi})=L^\ast_f(\chi).\end{equation} 
\end{corollary}
\begin{proof}
Let $\chi \modulo D,\tilde{\chi}$ be as in the statement of the corollary. Denote by $\chi'$ the extension of $\chi$ to a (possibly non-primitive) character of conductor $D'=\mathrm{lcm}(D,\prod_{i=1}^r\ell_i)$.  By  the definition of the horizontal $p$-adic $L$-function in terms of the elements (\ref{eq:compatsys}) we have
$$\nu^\pm_{f,\mathcal{L},r}(\tilde{\chi})=\left(\prod_{i\geq r+1: \ell_i|D}(a_f(\ell_i)\ell_i^{-(k-2)/2}-\chi^{(i)}(\ell_i)-\epsilon_f(\ell_i)\overline{\chi}^{(i)}(\ell_i))\right)^{-1} \theta^\pm_{f,D'}(\chi').$$
Now the result follows directly by applying the norm relations from $D'$ to $D$ as in (\ref{eq:finalnormrel}) of Corollary \ref{cor:normrelationstheta} to $\theta^\pm_{f,D'}(\chi')$ and applying the Birch--Stevens formula (Proposition \ref{cor:BirchStevens}) to $\theta^\pm_{f,D}(\chi)$. Here we are using the fact that evaluation at a character defines a ring homomorphism.  
\end{proof}
\begin{remark}\label{rem:abstract} The horizontal $p$-adic $L$-functions defined above all depend on the choice of projections (\ref{eq:biroot}) which is suppressed in the notation.
Alternatively, one can define the horizontal $p$-adic $L$-function in a canonical way as an element of $R\llbracket G_{\mathcal{L}}\rrbracket$ with $G_{\mathcal{L}}=\Gal(\Q(\mu_{\mathcal{L}})/\Q)^{(p)}$ where $\Q(\mu_{\mathcal{L}})$ denotes the composition of all cyclotomic extensions of conductor $\prod_{i\leq n} \ell_i$ with $n\in \N$ and $G^{(p)}$ denotes the maximal pro-$p$ quotient of a pro-finite group $G$. In this paper, we have however chosen to work with the concrete ``coordinates'' $\Z/p^{m_n},n\in \N$ in alignment with the use of the coordinate $\Z_p$ in vertical Iwasawa theory.  
\end{remark}
\subsubsection{Products of measures} Let  $s\geq 1$ be an integer and let $f_1,\ldots, f_s$ be $s$ (not necessarily distinct) newforms of level $N_i$, respectively. Denote by $K=K_{f_1}\cdots K_{f_s}$ the compositum of the Hecke fields and put $N=N_1\cdots N_s$. Let $p$ be a prime and let $\lambda|p$ be place of $K$ above $p$ and let $R$ be the valuation ring of the completion $K_\lambda$. Assume that $\lambda$ is $(f_1,\ldots, f_s)$-good in the sense of Defintion \ref{def:TWjoint}.  Let $r\geq 0$ be an integer and let $\mathcal{L}=(\ell_n)_{n\in \N}$ be a sequence of distinct primes not dividing $pN$ and congruent to $1$ modulo $p$ such that $\ell_n\in \mathcal{P}_\mathrm{or}(f_1,\ldots, f_s;\lambda)$ for $n\geq r+1$. As above, for $n\in \N$  put  $m_{n}=\ev_p(\ell_n-1)$. 
We define the following horizontal measure obtained by taking the product of the horizontal $p$-adic $L$-functions for each $f_1,\ldots, f_s$ associated to $(\mathcal{L},r,\pm)$ as in Definition \ref{def:padicL}: 
 \begin{equation}
 \label{eq:productmeasure} \nu^\pm_{f_1,\ldots, f_s;\mathcal{L},r}:=\prod_{i=1}^s \nu^\pm_{f_i,\mathcal{L},r}\in R\left\llbracket \prod_{n\in \N} \Z/p^{m_n} \right\rrbracket .
 \end{equation} 
As above for brevity, we put $\nu_{f_1,\ldots, f_s;\mathcal{L},r}=\nu^+_{f_1,\ldots, f_s;\mathcal{L},r}$. We have the following interpolation formula.
 \begin{corollary}\label{cor:productmeasure}Let notation be as above.
Let $\chi$ be a Dirichlet character of $p$-power order with conductor dividing $\prod_{n\in \N}\ell_n$. Write $\chi=\tilde{\chi}\circ \rho_\N$ in terms of the projection $\rho_\N: \prod_{n\in \N}(\Z/\ell_n)^\times \twoheadrightarrow \prod_{n\in \N}\Z/p^{m_n}$ and a character  $\tilde{\chi}:\prod_{n\in \N}\Z/p^{m_n}\rightarrow \C_p^\times$. Let $\pm$ denote the sign of $\chi$, i.e. $\chi(-1)=\pm 1$. Then the horizontal measure 
$\nu^\pm_{f_1,\ldots, f_s;\mathcal{L},r}\in R\left\llbracket \prod_{n\in \N} \Z/p^{m_n} \right\rrbracket $ as in (\ref{eq:productmeasure})
satisfies the interpolation formula
\begin{equation}\label{eq:interpoltation}
\nu^\pm_{f_1,\ldots, f_s;\mathcal{L},r}(\tilde{\chi})=\prod_{i=1}^s L^\ast_{f_i}(\chi),
\end{equation} 
with $L^\ast_{f_i}(\chi)$ given by (\ref{eq:interpolationfinal}).
 \end{corollary} 
 \begin{proof}
 This follows directly from (\ref{eq:interpolationmeasure}) since evaluation at a character of $\prod_{n\in \N}\Z/p^{m_n}$ defines a ring homomorphism $R\left\llbracket \prod_{n\in \N} \Z/p^{m_n} \right\rrbracket\rightarrow \Oo_{\C_p}$. 
 \end{proof}
\subsection{Quantitative Propagation of Non-vanishing}
In this section, we will apply the general results on non-vanishing for horizontal measures from Section \ref{sec:horziontalmeasures} to the measures constructed in the previous section which will yield non-vanishing results for $L$-values twisted by finite order characters. Let $d\geq 2$ be an integer and define the relevant set of characters as follows:
\begin{align}\mathcal{K}_d:=\{\chi\modulo D: \text{primitive Dirichlet character of order $d$}\},\end{align}
 and for $X\geq 1$ put
 \begin{align}
\mathcal{K}_d(X):=\{(\chi\modulo D)\in \mathcal{K}_d: D\leq X\}.
\end{align}
The key for obtaining strong quantitative non-vanishing is  Theorem \ref{thm:nonvanishing} which yield that for a non-zero horizontal pro-$p$  measure a ``positive proportion'' of characters are non-vanishing. To translate this to an explicit quantitative estimate we will need to count Dirichlet characters of fixed order. This is quite standard following \cite{Serre76} but for the convenience of the reader we provide details.  
\subsubsection{Counting Dirichlet characters of fixed order} 
For a subset of primes $\mathcal{A}$ and integers $h,d\geq 1$ such that $h|d$ we define
\begin{equation}\label{eq:Ah}\mathcal{A}_{h,d}:=\{\ell\in \mathcal{A}: (\ell-1,d)=h\}.\end{equation}
When a subset of primes $\mathcal{A}$ has natural density among all primes we denote it by: 
$$ \d(\mathcal{A}):=\lim_{X\geq 1} \frac{\#\{\ell\in \mathcal{A}: \ell\leq X\}}{\#\{\ell \text{ prime}: \ell\leq X\}}. $$
For a prime $p$ and integers $k,m$ such that $0\leq k\leq m$ and $m\geq 1$ we define
$$g(p^m,p^k):=\begin{cases} 1,& k=m\\ p^{m-k}-p^{m-k-1},& 0<k<m\\ p^{m}-2p^{m-1},& k=0\end{cases}$$
and extend it multiplicatively to all pairs $h|d$ by: 
\begin{equation}\label{eq:CRT}g(d,h):=\prod_{p|d}g(p^{\ev_p(d)},p^{\ev_p(h)}).\end{equation} 
\begin{lemma}\label{lem:countingchar}
Let $\mathcal{P}$ denote the set of primes and let $h,d\geq 1$ be integers such that $h|d$. Then $\mathcal{P}_{h,d}$ has natural density given by 
$$ \d(\mathcal{P}_{h,d})=\frac{g(d,h)}{\varphi(d)}.   $$ 
\end{lemma}
\begin{proof}
By the Prime Number Theorem for arithmetic progressions we see that the density of $\mathcal{P}_{h,d}$ is exactly
$$\frac{\#\{a\in (\Z/d)^\times: (a-1,d)=h\}}{\varphi(d)}.$$
When $d,h$ are powers of the same prime one can easily check that the numerator is indeed given by the function $g$ above. By the Chinese Reminder Theorem it follows that the numerator is given by (\ref{eq:CRT}) for general $h|d$. \end{proof}
\begin{lemma}\label{lem:rprimefactors}
Let $\mathcal{A}$ be a subset of the set of primes and let $d\geq 2$ be an integer. Assume that for each $h|d$ the set $\mathcal{A}_{h,d}$ contains a subset of natural density $\alpha_{h,d}$ among all primes and that $\alpha_{d,d}>0$. Put 
\begin{equation}\label{eq:formulaalpha}\alpha=\sum_{h|d} (h-1)\alpha_{h,d}.\end{equation}
Then there exists a constant $c_{\mathcal{A},d}>0$ such that
\begin{align}\label{eq:rprimefactors}\#\left\{(\chi \modulo D)\in \mathcal{K}_d(X): D|\prod_{\ell\in\mathcal{A}}\ell\right\}\geq  (c_{\mathcal{A},d}+o_{\mathcal{A},d}(1)) X(\log X)^{\alpha-1},\end{align}
as $X\rightarrow \infty$. Furthermore, the inequality (\ref{eq:rprimefactors}) holds with equality if for all $h|d$ the subset $\mathcal{A}_{h,d}$ has natural density among all primes and we put $\alpha_{h,d}=\d(\mathcal{A}_{h,d})$ in (\ref{eq:formulaalpha}). 
\end{lemma}
\begin{proof} The proof is standard following Serre \cite{Serre76} and we will simply sketch how it goes. By possibly going to subsets we may assume that $\mathcal{A}_{h,d}$ has natural density exactly $\d(\mathcal{A}_{h,d})$ and then show that equation (\ref{eq:rprimefactors}) holds with equality. One sees that the generating function for the number of characters of order \emph{dividing} $d$ with conductor a product of primes in $\sqcup_{h|d}\mathcal{A}_{h,d}$ is given by
\begin{equation}
L(s)=\sum_{n\geq 1} \frac{a_n}{n^s}=\prod_{h|d} \prod_{\ell\in \mathcal{A}_{h,d}}(1+(h-1)\ell^{-s}),\quad \Re s>1.
\end{equation} 
Now it follows as in \cite{Serre76} that 
$$L(s)=\exp\left(-\left(\sum_{h|d} (h-1)\d(\mathcal{A}_{h,d})\right)\log (s-1)+\Phi(s)  \right),\quad \Re s>1,$$
where $\Phi(s)$ is holomorphic for $\Re s\geq 1$ and thus by a Tauberian argument that 
\begin{align}\label{eq:nmbofchar}\sum_{n\leq X} a_n= (c+o(1))X(\log X)^{\Sigma_{h|d} (h-1)\d(\mathcal{A}_{h,d})-1},\quad \text{as }X\rightarrow \infty,\end{align}
for some constant $c>0$. Thus we are reduced to showing that the number of characters as above with order less than $d$ has a smaller order of growth. We apply the same argument as above to each $h|d,h\neq d$ and by the assumption $\d(\mathcal{A}_{d,d})>0$ we see that the number of characters with order dividing $d$ with $h<d$ is of strictly smaller order of growth than (\ref{eq:nmbofchar}). This yields the wanted inequality with equality if all the sets $\mathcal{A}_{h,d}$ have natural density. 
\end{proof}
\begin{corollary}\label{cor:numberchar}
Let $d\geq 2$ be an integer. Then there exists a constant $c_d>0$ such that 
$$ \#\mathcal{K}_{d}(X)=(c_d+o(1))X(\log X)^{\sigma_0(d)-2},\quad \text{as }X\rightarrow \infty,$$
where $\sigma_0(d)=\sum_{h|d}1$ denotes the divisor function.
\end{corollary}
\begin{proof}
Note that there is a finite number of Dirichlet characters $\chi_1,\ldots, \chi_N$ of order dividing $d$ such that $\cond(\chi_i)|d^\infty$. Thus we can reduce to considering those with $(\cond(\chi),d)=1$ (by summing up the contribution from each $\chi_i$). By combining Lemmas \ref{lem:countingchar} and \ref{lem:rprimefactors} we see that $\alpha$ as in equation (\ref{eq:formulaalpha}) is given by
\begin{align*}\alpha&=\frac{1}{\varphi(d)}\sum_{h|d}g(d,h)(h-1)=\frac{1}{\varphi(d)}\left(\prod_{p^m|\!|d}\left(\sum_{k=0}^m g(p^m,p^k)p^k\right)-\prod_{p^m|\!|d}\left(\sum_{k=0}^m g(p^m,p^k)\right)\right).\end{align*}
Now we observe that 
$$ \sum_{k=0}^m g(p^m,p^k)p^k=p^{m}-2p^{m-1}+p^m+\sum_{k=1}^{m-1}(p^{m-k}-p^{m-1-k})p^k=(m+1)(p^{m}-p^{m-1}),  $$
and similarly 
$$\sum_{k=0}^m g(p^m,p^k)=p^m-2p^{m-1}+1+\sum_{k=1}^{m-1}(p^{m-k}-p^{m-1-k})=p^{m}-p^{m-1}.$$
Inserting this into the above yields $\alpha=\sigma_0(d)-1$ as desired.
\end{proof}
\subsubsection{Propagating non-vanishing}
We are now ready to prove the main quantitative propagation of non-vanishing result underlying all of our applications stated in the introduction.
\begin{theorem}[Propogation of non-vanishing]\label{thm:propa}
Let $f_1,\ldots, f_n$ be newforms of even weight  $k_i$, level $N_i$ and Hecke field $K_{f_i}$, respectively. Let $p^m$ be a prime power and let $\lambda|p$ be a place  of $K_{f_1}\cdots K_{f_n}$ above $p$. Assume that $\lambda^m$ is $(f_1,\ldots,f_n)$-good and put
\begin{equation}\label{eq:alphaformula}\alpha=\sum_{k=1}^m (p^k-p^{k-1})\cdot \d(\mathcal{P}_\mathrm{or}(f_1,\ldots,f_n;\lambda,k))>0.\end{equation} 
Assume that $L(f_i,k_i/2)\neq 0$ for all $i=1,\ldots, n$ and put
\begin{align}\label{eq:either}e=\sum_{i=1}^n\ev_p\left(L(f_i,k_i/2)/\Omega_{f_i}^+\right)<\infty, \end{align}
where $\Omega_{f_i}^+$ denotes a plus period of $f_i$ as in (\ref{eq:normaddtwist}). Then for any integer $B\geq 1$ it holds that 
\begin{align}\label{eq:propquant}
\#\left\{ (\chi \modulo D)\in \mathcal{K}_{p^m}(X)\text{ even}: (D,B)=1, \sum_{i=1}^n\ev_p\left(L(f_i,\chi,k_i/2)/\Omega_{f_i}^+\right)\leq e \right\}\gg \frac{X}{(\log X)^{1-\alpha}}, 
\end{align}
as $X\rightarrow \infty$. Here the implied constant is allowed to depend on $f_1,\ldots, f_n,p,m, B$.
\end{theorem}
\begin{proof}
Let $\mathcal{L}=(\ell_i)_{i\in \N}$ be a sequence consisting of all the orderly primes not dividing $B$, i.e.
$$\bigsqcup_{i\in \N}\{\ell_i\}=\mathcal{P}_\mathrm{or}(f_1,\ldots, f_n;\lambda)\setminus \{\ell\text{ prime}: \ell|B\},$$ 
and as above, for $i\in \N$ put  $m_i=\ev_p(\ell_i-1)$. Let
$$\nu:=\nu_{f_1,\ldots, f_n;  \mathcal{L},0}\in R\left\llbracket \prod_{i\in \N} \Z/p^{m_i} \right\rrbracket ,$$ 
be the horizontal $p$-adic $L$-function associated to $f_1,\ldots, f_n$ (and $\mathcal{L}$, $r=0$) as in equation (\ref{eq:productmeasure}) and Definition \ref{def:padicL} with periods $\Omega_{f_i}^+$. By the interpolation property from Corollary \ref{cor:productmeasure} and the assumption (\ref{eq:either}) we know that $\nu(\mathbf{1})\neq 0$. Note that for $\chi$ of order $p^m$ we have   $\chi' \in \langle \chi\rangle \setminus  \langle \chi^p\rangle $ exactly if $\chi'=\chi^\sigma$ for some $\sigma\in \Gal(\Q_p(\mu_{p^m})/\Q_p)$. Thus we conclude from  Theorem \ref{thm:nonvanishingpm} applied to the pushforward of $\nu$ along the canonical projection $\prod_{i\in \N}\Z/p^{m_i}\rightarrow \prod_{i\in \N}\Z/p^{\min(m,m_i)}$ that there exists a finite subgroup $M_0$ of continuous characters of $\prod_{i\in \N}\Z/p^{m_i}$ of order at most $p^m$ such that the following holds: for any order $p^m$ character $\chi: \prod_{i\in \N}\Z/p^{m_i}\rightarrow \C_p^\times$ there exists $\sigma\in \Gal(\Q_p(\mu_{p^m})/\Q_p)$ and a character $\chi_0\in M_0$ such that
\begin{equation}\label{eq:leqe}
    \ev_p(\nu(\chi^\sigma\chi_0))\leq v_p(\nu(\1)) ,
\end{equation} 
where $\ev_p$ is defined via the embedding $K_{f_1}\cdots K_{f_n}\hookrightarrow \C_p$ determined by $\lambda|p$. Note that by Corollary \ref{cor:productmeasure} and the interpolation formula (\ref{eq:interpolationmeasure}) the valuation  $e=v_p(\nu(\1))$ is given by the formula (\ref{eq:either}). By the formula (\ref{eq:interpolationfinal}) the algebraically normalized $L$-values divides the modified $L$-values and so by the interpolation formula (\ref{eq:interpoltation}) we conclude that for  $\chi^\sigma$ and $\chi_0$ satisfying (\ref{eq:leqe}) it holds that $\chi^\sigma\chi_0$ is even and  
$$ \sum_{i=1}^n \ev_p\left( L(f_i,\chi^\sigma\chi_0,k_i/2)/\Omega_{f_i}^+\right)\leq  \ev_p\left(\prod_{i=1}^n L^\ast_{f_i}(\chi^\sigma\chi_0)\right)\leq e ,$$  where we identity the characters of $\prod_{i\in \N}\Z/p^{m_i}$ with the corresponding Dirichlet characters as in Corollary \ref{cor:measure}. To estimate the number of such characters $\chi^\sigma\chi_0$, we denote the indices where $M_0$ has non-trivial restriction and the corresponding conductor by
$$\mathcal{I}(M_0):=\{i\in \N: \exists \chi \in M_0\text{ s.t. } \chi(\Z/p^{m_i})\neq \{1\}\},\quad \cond(M_0):=\prod_{i\in \mathcal{I}(M_0)}\ell_i<\infty.$$
Then for $\chi\in \mathcal{K}_{p^m}\left(\tfrac{X}{ \cond(M_0)}\right)$ the conductor of $  \chi^\sigma\chi_0$ is at most $X$.  We note that the equality $  \chi^\sigma\chi_0=  (\chi')^{\sigma'} \chi_0'$ with $\chi_0,\chi_0'\in M_0$ implies that $\chi'\in \langle \chi\rangle M_0 $. Since the size of $\langle \chi\rangle M_0$ is $\leq p^m \#M_0$ we conclude that the characters we want to count on the left-hand side of (\ref{eq:propquant}) has size at least
\begin{equation}\label{eq:countproof}  \frac{1}{p^m |M_0|}\cdot \#\left\{ (\chi\modulo D)\in \mathcal{K}_{p^m}\left(\tfrac{X}{ \cond(M_0)}\right): D| \prod_{i\in \N} \ell_i \right\}.  \end{equation}
By the assumption that $\lambda^m$ is $(f_1,\ldots,f_n)$-good we have 
$$\d(\mathcal{L}_{p^m,p^m})=\d(\mathcal{P}_\mathrm{or}(f_1,\ldots, f_n; \lambda,m))>0,$$
using the notation (\ref{eq:Ah}). Thus by Lemma \ref{lem:rprimefactors} the count in  (\ref{eq:countproof}) satisfies $\gg X /(\log X)^{1-\alpha} $ where $\alpha$ is given by the formula (\ref{eq:formulaalpha}) and the implied constant is allowed to depend on $p,m,\LL,M_0$. Finally, we note that for $k<m$
\begin{align*}
\d(\mathcal{L}_{p^k,p^m})&=\d(\mathcal{P}_\mathrm{or}(f_1,\ldots, f_n; \lambda,k)\setminus \mathcal{P}_\mathrm{or}(f_1,\ldots, f_n; \lambda,k+1))\\
&=\d(\mathcal{P}_\mathrm{or}(f_1,\ldots, f_n; \lambda,k))- \d(\mathcal{P}_\mathrm{or}(f_1,\ldots, f_n; \lambda,k+1)),\end{align*}
which yields the formula (\ref{eq:alphaformula}). This finishes the proof.
 \end{proof}
Combining the above ``propagation of non-vanishing''-result at good primes $p$ with the general criterion from Corollary \ref{cor:TWpropa} for $p$ to be good upon twisting, we arrive at the following.  
\begin{corollary}\label{cor:propa}
Let $f_1,\ldots, f_n$ be newforms of even weight $k_i$ and level $N_i$, respectively. Let $\eta\modulo D$ be an even  Dirichlet character of order $d_0\geq 2$. Let $d\geq 2$ be an integer. Assume that 
\begin{itemize}
\item $(D,N_1\cdots N_nd)=1$,
 \item  $(d_0,d)=1$,
  \item $L(f_i,\eta,k_i/2)\neq 0$ for $i=1,\ldots ,n$.
\end{itemize}
Then there exists a constant $\alpha=\alpha_{d,f_1,\ldots,f_n,\eta}>0$ such that 
\begin{align}\label{eq:corpropa}
\#\left\{ \chi\in \mathcal{K}_d(X): \prod_{i=1}^nL(f_i,\eta\chi,k_i/2)\neq 0 \right\}\gg_{f_i,\eta,d} \frac{X}{(\log X)^{1-\alpha}},\quad \text{as } X\rightarrow \infty.
\end{align}
\end{corollary}
\begin{proof}
Assume firstly that $d=p^m$ is a prime power. Then by Corollary \ref{cor:TWpropa} and the coprimality assumptions we conclude that $p^m$ is $(f_1\otimes \eta,\ldots, f_n\otimes \eta )$-good.
Now we apply Theorem \ref{thm:propa} to $f_1\otimes \eta,\ldots, f_n\otimes \eta$, with $B=N_1\cdots N_n D d$, which implies that there exists $\alpha>0$ such that the lower bound (\ref{eq:corpropa}) holds, as wanted. Assume now that $d=p_1^{m_1}\cdots p_r^{m_r}$ with $r>1$ in which case we will proceed prime by prime: By the preceding argument we can, in particular, find an even character $\chi_1$ of order $p_1^{m_1}$ so that the $L$-values twisted by $\eta \chi_1$ (i.e.\ those appearing on the left-hand side of (\ref{eq:corpropa})) are non-vanishing. Continuing like this, applying the same argument with $\eta\chi_1$ in place of $\eta$, etc., we end up with an even character $\chi'$ of order $p_1^{m_1}\cdots p_{r-1}^{m_{r-1}}$ so that the  $L$-values twisted by $\eta\chi'$ are non-vanishing. Finally we apply the same argument to $\eta \chi'$ and conclude that there exists $\alpha>0$ such that there are $\gg X/(\log X)^{1-\alpha}$ many characters $\chi$ of order $p_r^{m_r}$ and conductor $\leq X/\mathrm{cond}(\chi')$ so that the $L$-values twisted by $\eta \chi' \chi$ are non-vanishing (where $\alpha$ and the implied constant is allowed to depend on $f_1,\ldots, f_n, \eta, \chi',p_r,m_r$). Since in this case $ \chi' \chi$ has order exactly $d$ and conductor at most $X$ we conclude.   
\end{proof}

\subsection{Mod {\it p} Kurihara Conjecture and Horizontal Measures}\label{sec:Kolyvaginsection}
In this section we state the \emph{mod $p$ Kurihara Conjecture} on the non-vanishing of Kurihara numbers modulo $p$, as well as the recent results of Kim \cite{Kim} and Burungale-Castella-Grossi-Skinner \cite{BurungaleCastellaGrossiSkinner} which show that the mod $p$ Kurihara Conjecture is true in the majority of cases (Corollary \ref{cor:Kurihara}). This gives results on non-vanishing of horizontal measures (Corollary \ref{cor:nonvanishingQ}). We will throughout fix a prime $p$ and an elliptic curve $E/\Q$ of conductor $N$ with $L$-function coefficients $a_E(n),n\geq 1$.  Assume that $p$ is prime to the Manin constant $c_E$ of $E$ and that $E$ has no non-trivial rational $p$-torsion, i.e. $E[p](\Q)=0$.

\begin{definition}[Kato primes]
Let $\mathcal{N}_{\mathrm{Kato}}=\mathcal{N}_{\mathrm{Kato}}(E;p)$ be the set of squarefree integers which are a product of primes $q$ satisfying that $(q,N) = 1$, $q\equiv 1\modulo p$ and $a_{E}(q)\equiv 2\modulo p$. 
\end{definition}

Note that primes and squarefree products of such primes in $\mathcal{N}_{\mathrm{Kato}}$ are analogues of the Kolyvagin primes and squarefree products of Kolyvagin primes considered in \cite{WZhang}. Note also that the Kato primes are complementary to the orderly primes in the sense that $\mathcal{N}_\mathrm{Kato}(E;p)\cap \mathcal{P}_\mathrm{or}(E;p)=\emptyset$ and $\mathcal{N}_\mathrm{Kato}(E;p)\sqcup \mathcal{P}_\mathrm{or}(E;p)=\{\ell \text{ prime}: \ell\equiv 1\modulo p, \ell\nmid N\}$. For each $q\in \mathcal{N}_{\mathrm{Kato}}$ we pick a generator $\zeta_q\in (\Z/q)^\times$ and for $q|Q\in \mathcal{N}_{\mathrm{Kato}}$ we consider $\zeta_q$ as an element of $(\Z/Q)^\times$ by putting  $\zeta_q\equiv 1\modulo \tfrac{Q}{q}$. Let $Q=q_1\cdots q_r$ be a squarefree integer with $q_i\in \mathcal{N}_{\mathrm{Kato}}$ and recall the Kurihara number from (\ref{eq:Kolderiv}):
\begin{equation}\label{eq:kuriharanmb}\delta_Q=\delta_{Q,E,p} := \sum_{a_1=1}^{q_1-1}\cdots \sum_{a_r=1}^{q_r-1}\left(\prod_{i=1}^ra_i\right) \left\langle \tfrac{\prod_{i=1}^r(\zeta_{q_i})^{a_i}}{q_1\cdots q_r} \right\rangle_E^+ \in \mathbb{Z}_{(p)},\end{equation}
where the inclusion $\delta_Q\in\mathbb{Z}_{(p)}$ holds by the discussion following equation (\ref{eq:neron}).
\begin{Conjecture}[mod $p$ Kurihara Conjecture]\label{conj:Kurihara}
Assume that $E[p](\Q)=0$, $p\nmid c_E$, and that $p$ does not divide the product of all Tamagawa factors of $E$ over $\mathbb{Q}$.  Then there exists $Q \in \mathcal{N}_{\mathrm{Kato}}$ such that 
\begin{equation}\label{eq:deltaQ}\delta_Q \not\equiv 0 \pmod{p}.\end{equation}
\end{Conjecture}

Notice that the non-vanishing modulo $p$ of (\ref{eq:kuriharanmb}) does not depend on the choice of generators $\zeta_{q_i}$. C.-H. Kim established the following result toward this conjecture.

\begin{theorem}[Consequence of Theorem 1.11 of \cite{Kim}]\label{thm:Kimthm}Suppose $E$ is an elliptic curve over $\mathbb{Q}$ and $p \ge 5$.  Let $\bar{\rho}_{E,p} : \mathrm{Gal}(\overline{\mathbb{Q}}/\mathbb{Q}) \rightarrow \mathrm{Aut}(E[p])$ denote the associated mod $p$ Galois representation. Assume the following on $(E,p)$:
\begin{enumerate}
\item The residual representation $\bar{\rho}_{E,p}$ is surjective.
\item The Manin constant $c_E$ of $E$ is prime to $p$.
\item $E[p](\mathbb{Q}_p) = 0$. 
\item\label{item:nonanom} The product of all Tamagawa factors of $E$ over $\mathbb{Q}$ is prime to $p$. 
\end{enumerate}
Then the following statements are equivalent:
\begin{enumerate}
\item Conjecture \ref{conj:Kurihara} is true.
\item The Kolyvagin system attached to Kato's Euler system \cite[Sec. 2]{Kim} is nonzero modulo $p$. 
\item The cyclotomic Iwasawa main conjecture for $E$ at $p$ \cite[Conj. 1.3]{Kim} holds. 
\end{enumerate}
\end{theorem}

In recent work \cite{BurungaleCastellaGrossiSkinner}, A. Burungale, F. Castella, G. Grossi and C. Skinner verify that, under the assumptions of Theorem \ref{thm:Kimthm} plus a couple auxiliary assumptions, the Kolyvagin system attached to Kato's Euler system is nonzero modulo $p$ and hence, in view of Theorem \ref{thm:Kimthm}, that the mod $p$ Kurihara conjecture holds in great generality.

\begin{theorem}[Consequence of Theorem C and Remark 3.3.4 of \cite{BurungaleCastellaGrossiSkinner}]\label{thm:BGCSthm}Assume that $(E,p)$ satisfies the assumptions of Theorem \ref{thm:Kimthm}, that $E$ has ordinary good reduction at $p$, and that
\begin{equation}\label{eq:Rubinconstant}p\nmid \prod_{\ell|N}\left(1-\frac{a_E(\ell)}{\ell}\right).\end{equation}
Then the Kolyvagin system attached to Kato's Euler system is nonzero modulo $p$. 
\end{theorem}

\begin{proof} Under the assumptions of Theorem \ref{thm:BGCSthm}, by Theorem C and Remark 3.3.4 of \cite{BurungaleCastellaGrossiSkinner}, it suffices that $E$ has ordinary good reduction at $p$ and
\[p\nmid r_E\prod_{\ell|N}\left(1-\frac{a_E(\ell)}{\ell}\right)\]
for some nonzero $r_E \in \Z$ defined in terms of equation (3.6) of \emph{loc. cit.}, where in the notation of \cite[Section 3]{BurungaleCastellaGrossiSkinner}, $t_1 = \ev_p(r_E)$ and $t_3 = \ev_p\left(\prod_{\ell|N}\left(1-a_E(\ell)\ell^{-1}\right)\right)$. By \cite[Appendix A]{KimNakamura} one has that $r_E \in \Z_p^{\times}$, and so under our assumptions $h = t_1 + t_3 = 0$. Thus by Remark 3.4.3 of \emph{loc. cit.} we have the non-triviality mod $p$ of the Kolyvagin system attached to Kato's Euler system.
\end{proof}

\begin{corollary}\label{cor:Kurihara}Under the assumptions of Theorem \ref{thm:BGCSthm}, Conjecture \ref{conj:Kurihara} is true, the Kolyvagin system attached to Kato's Euler system is nonzero modulo $p$ and the cyclotomic Iwasawa main conjecture for $E$ at $p$ holds.
\end{corollary}

\begin{proof}This follows immediately from Theorems \ref{thm:Kimthm} and \ref{thm:BGCSthm}.
\end{proof}
The presence of the mod $p$ Kurihara conjecture yields the following non-vanishing result for horizontal $p$-adic $L$-functions. Recall the definition of the Kato--Kolyvagin derivatives $D^r$ and $D^r(\nu)$ as defined in Sections \ref{sec:horizontalAmice} and \ref{sec:KatoKoly}.
\begin{corollary}\label{cor:nonvanishingQ}Let $E/\Q$ be an elliptic curve and $p$ a prime number such that $E[p](\Q)=0$ and  $p\nmid c_E$. Assume that $E$ satisfies the mod $p$ Kurihara Conjecture (Conjecture \ref{conj:Kurihara}). In particular, this is satisfied under the assumptions of Theorem \ref{thm:BGCSthm} (by Corollary \ref{cor:Kurihara}). Let $Q = q_1\cdots q_r\in \mathcal{N}_{\mathrm{Kato}}$ be such that $\delta_Q \not\equiv 0 \pmod{p}$ and let $\mathcal{L}=(\ell_n)_{n\in \N}$ be a sequence of distinct primes with $\ell_i=q_i$ for $1\leq i\leq r$ and such that $\ell_n\in \mathcal{P}_\mathrm{or}(E;p)$ for $n\geq r+1$. 
Let 
$$\nu_{E}:=\nu_{f_E,\mathcal{L},r}\in \Z_p\left\llbracket \prod_{n\in \N} \Z/p^{m_n} \right\rrbracket,$$ 
be the horizontal $p$-adic $L$-function associated to the newform $f_E$ corresponding to $E$ (and $\mathcal{L}$ and $r$) as in Definition \ref{def:padicL}. Then we have 
\begin{equation}
 \label{eq:nonvanishingDr}   D^r(\nu_E) \not\equiv 0\modulo p,
\end{equation}
with the Kato--Kolyvagin derivative $D^r$ defined  by the formulas (\ref{eq:assfunc}) and (\ref{eq:defDr}),  and $D^r(\nu_E)$ given by (\ref{eq:nuDa}).  In particular, we have $\mu(\nu_E)=0$ and  $\lambda(\nu_E)\leq r$ defined as in Definition \ref{def:lambdamu}.
\end{corollary}
\begin{proof}
By the definition of the horizontal $p$-adic $L$-function $\nu_E$ from Definition \ref{def:padicL} in terms of the compatible system (\ref{eq:compatsys}) and the definition of the Kato--Kolyvagin derivative (\ref{eq:Dr}) we see that 
$$D^r(\nu_E)=\nu_E(\varphi_{D^r})=\tilde{\theta}^+_{f_E,Q}(\varphi_{D^r}).$$
Recall that $\tilde{\theta}^+_{f_E,Q}$  is the pushforward of $\theta^+_{f_E,Q}$ (defined in (\ref{eq:normrelationstheta})) along the projection 
$$\rho_r:\prod_{i=1}^r (\Z/q_i)^\times\twoheadrightarrow \prod_{i=1}^r \Z/p^{m_i},$$
determined by a choice of generator $\zeta_i\in (\Z/q_i)^\times$ for each $i=1,\ldots, r$. This implies that 
$$D^r(\nu_E)= \sum_{b_1=1}^{p^{m_1}}\cdots \sum_{b_r=1}^{p^{m_r}}\left(\prod_{i=1}^r b_i\right)\cdot \left(\sum_{\substack{a_1 \modulo q_1-1:\\ a_1\equiv b_1\modulo p^{m_1}}}\cdots \sum_{\substack{a_r \modulo q_r-1:\\ a_r\equiv b_r\modulo p^{m_r}}} \theta^+_{f_E,Q}\left(\prod_{i=1}^r (\zeta_i)^{a_i} \right)\right) $$
Since $m_i\geq 1$ for $i=1,\ldots, r$, we can interchange the sums modulo $p$: 
$$D^r(\nu_E)\equiv \sum_{a_1=1}^{q_1-1}\cdots \sum_{a_r=1}^{q_r-1}\left(\prod_{i=1}^r a_i\right) \theta^+_{f_E,Q}\left(\prod_{i=1}^r (\zeta_i)^{a_i}\right)\mod p. $$
Finally, expressing $\theta_{f_E,Q}^+$ in terms of additive twist $L$-series using  (\ref{eq:normrelationstheta}) and rewriting the expression in terms of (plus) modular symbols using (\ref{eq:addtwistMS}), we see that the right-hand side of the above congruence is exactly equal to the Kurihara number $\delta_Q$ from (\ref{eq:kuriharanmb}) multiplied by  $c_E\cdot(\#E_\mathrm{tor}(\Q))$. By the assumptions on $p$ we have $p\nmid c_E\cdot(\#E_\mathrm{tor}(\Q))$ and thus we get the wanted non-vanishing modulo $p$. The last conclusion on the $\mu$- and $\lambda$-invariants follows directly from (\ref{eq:nonvanishingDr}) by definition. 
\end{proof}
\begin{remark}\label{rem:augrank}
Note that under the assumptions of Theorem \ref{thm:BGCSthm}  it follows from 
Corollary C of \cite{BurungaleCastellaGrossiSkinner} that the minimal $r$ such that $\delta_Q\not\equiv 0\modulo p$ with $Q=q_1\cdots q_r\in \mathcal{N}_\mathrm{Kato}$ is equal to the ``strict $p$-Selmer rank'' $r_\mathrm{str}(E)$ of $E$. In particular, this implies that for $\nu_E=\nu_{f_E,\mathcal{L},r}$ as above we have $r_\mathrm{aug}(\nu_E)\leq r_\mathrm{str}(E)$. On the other hand, by Theorem 1.2 of \cite{Ota} if all the primes $\ell_n$ in $\mathcal{L}$ used to define the horizontal $p$-adic $L$-function $\nu_E$ satisfy that $E(\F_{\ell_n})[p]$ is either trivial or isomorphic to $\Z/p$ then  $r_\mathrm{aug}(\nu_E)\geq r_\mathrm{alg}(E)$. In particular, if we can find ``strict'' Kato primes, in the sense that $E(\F_{q_i})[p]\cong \Z/p$ for $i=1,\ldots, r$, such that Conjecture \ref{conj:Kurihara} is satisfied, then assuming finiteness of the $p$-part of the Tate--Shafarevich group of $E$ it holds that $r_\mathrm{aug}(\nu_E)=r_\mathrm{alg}(E)$. In this case, 
if $\overline{\nu_E}\in \Z_p\llbracket(\Z/p)^\N\rrbracket$ denotes the pushforward to the digit algebra, then by Proposition \ref{prop:augmentation} the minimal valuation of $\overline{\nu_E}$ is given by:
$$\ev_p(\overline{\nu_E})=\frac{r_\mathrm{alg}(E)}{p-1}, $$
which by Corollary \ref{cor:nonvanishingPS} is attained for a positive proportion of continuous characters of $(\Z/p)^\N$. 
\end{remark}
\subsection{Proofs of Main Results}\label{sec:final}
In this final section we collect all of the above results to obtain a variety of quantitative non-vanishing results for twisted $L$-values.  
\begin{corollary}\label{cor:d=2(4)}
Let $f$ be a newform of even weight $k$. Let $d\equiv 2\modulo 4$ be a positive integer greater than 2.  Then there exists a constant $\alpha=\alpha_{f,d}>0$ such that   
$$ \#\{\chi\in \mathcal{K}_{d}(X): L(f,\chi,k/2)\neq0\}\gg_{f,d} \frac{X}{(\log X)^{1-\alpha}},\quad \text{as }X\rightarrow \infty.$$
\end{corollary}
\begin{proof}
By the results of Friedberg--Hoffstein \cite{FriedbergHoffstein95} (or alternatively a first moment computation as in \cite{MurtyMurty91}) one can produce an even quadratic character $\chi_D$ with conductor coprime to the level of $f$ and coprime to $d$ such that $L(f,\chi_D,k/2)\neq 0$. Now the result follows directly from Corollary \ref{cor:propa} (since $(2,d/2)=1$).
\end{proof}

\begin{corollary}\label{cor:2^m}
Let $f$ be a newform of even weight $k$ such that $2^m$ is $f$-good for some $m\geq1$. Let $d\geq 1$ be an odd integer. Then there exists  $\alpha=\alpha_{f,m,d}>0$  such that
\begin{equation}\label{eq:2md} \#\{\chi\in \mathcal{K}_{2^m d}(X): L(f,\chi,k/2)\neq 0\}\gg_{f,m,d} \frac{X}{(\log X)^{1-\alpha}},\quad \text{as }X\rightarrow \infty.\end{equation}
Furthermore, if $d=1$ then there exists $e_{f,m}\geq 0$ such that the number of twisted algebraically normalized $L$-values with $2$-adic valuation bounded by $e_{f,m}$ satisfies the same lower bound as in (\ref{eq:2md}). If $\lambda| 2$ is a place of the Hecke field $K_f$ above $2$ and $\epsilon_f\equiv 1\modulo \lambda$ then one can take $\alpha_{f,m,1}=m\cdot \d(\mathcal{P}_\mathrm{or}(f;\lambda))$. 
\end{corollary}
\begin{proof}
As above using \cite{FriedbergHoffstein95} we can produce an even quadratic character $\chi_D$ such that the central value $L(f,\chi_D,k/2)$ is non-vanishing and conductor coprime to the level of $f$ and to $d$. The lower bound (\ref{eq:2md}) now follows directly by firstly applying Theorem \ref{thm:propa} to $f\otimes \chi_D$ and $p=2$  (note that $2^m$ is $f\otimes \chi_D$-good by Lemma \ref{lem:twistingTW}) and then Corollary \ref{cor:propa} when $d\neq 1$.  When $\epsilon_f\equiv 1\modulo \lambda$ and $d=1$ the claimed value of $\alpha_{f,m,1}$ follows by plugging the lower bound on the density (\ref{eq:TW2}) from Corollary \ref{cor:galpm} into the formula  (\ref{eq:alphaformula}) for $\alpha$ and using Lemma \ref{lem:twistingTW}.
\end{proof}
Observe that for $\epsilon_f\equiv 1\modulo \lambda$ and $\#\F_\lambda=2$ (e.g. for elliptic curves) we have the densities
$$\mathrm{d}(\mathcal{P}_\mathrm{or}(f;\lambda))=\begin{cases}2/3,& \im(\overline{\rho}_{f,\lambda})\cong \Z/3\\ 1/3,& \im(\overline{\rho}_{f,\lambda})\cong S_3\end{cases},$$
which yields the exponents claimed in Remark \ref{rmk:Ono} above.

Finally, we obtain the following general simultaneous non-vanishing result.
\begin{corollary}\label{cor:simultnonvgen}
Let $f_1,\ldots, f_n$ be newforms of even weight $k_i$,  respectively.  Let $d\geq 2$ be an integer and let $p$ be a prime dividing $d$ with $p^m|\!|d$.  Assume that $L(f_i,k_i/2)\neq 0$ for $i=1,\ldots, n$ and that $p^m$ is $(f_1,\ldots ,f_n)$-good. Then there exists a constant $\alpha=\alpha_{f_1,\ldots, f_n,d}>0$ such that
\begin{align}
\#\{ \chi\in \mathcal{K}_d(X): L(f_1,\chi,k_1/2)\cdots L(f_n,\chi,k_n/2)\neq 0 \}\gg_{d,f_i}
\frac{X}{(\log X)^{1-\alpha}},\quad \text{as }X\rightarrow \infty.\end{align} 
\end{corollary}
\begin{proof}
This follows for $d$ a prime power directly from Theorem \ref{thm:propa}. For composite $d$ we apply Corollary \ref{cor:propa} with $d/p^m$ and a non-vanishing twist $\eta$ of order $p^m$ obtained from Theorem \ref{thm:propa}.
\end{proof}
From this we get directly the non-vanishing statements in the introduction by specializing to weight $2$ forms associated to elliptic curves.
\begin{proof}[Proof of Theorem \ref{thm:non-vanishing}]
We apply Corollaries \ref{cor:d=2(4)}, \ref{cor:2^m}, and \ref{cor:simultnonvgen} to the newform $f_E$ of weight $2$ corresponding to the elliptic curve $E/\Q$ (via modularity). Note here that $\overline{\rho}_{E,p}$ being irreducible implies that $p^m$ is $E$-good for all $m\geq 1$ by Corollaries  \ref{cor:galpm} and \ref{cor:irred}. 
\end{proof}

\begin{proof}[Proof of Theorem \ref{thm:simultnonv}]
  We apply Corollary \ref{cor:simultnonvgen} to the weight $2$ newforms $f_{E_1},\ldots, f_{E_n}$ associated to the elliptic curves $E_1,\ldots,E_n$. Note here that $p$ being $(E_1,\ldots,E_n)$-good for elliptic curves $E_1,\ldots, E_n/\Q$ implies that $p^m$ is $(E_1,\ldots,E_n)$-good for all $m\geq 1$ by Corollary \ref{cor:galpmjoint}. 
\end{proof}
From this we deduce results on ``simultaneous Diophantine stability'' as claimed in the introduction.
\begin{proof}[Proof of Corollary \ref{cor:simultnonv}]
Let $K/\Q$ be a quadratic extension of discriminant $D$ and let $E/\Q$ be an elliptic curve. It is a classical fact (see e.g.  \cite[Exc. 10.16]{Silverman09}) that 
$$\rank_\Z E(K)=\rank_\Z E(\Q)+\rank_\Z E^D(\Q),$$
where $E^D/\Q$ denotes the quadratic twist of $E$ by $D$. Recall that if $(D, N)=1$ then we have $L(E^D,s)=L(E,\chi_D,s)$ where $\chi_D$ denotes the quadratic character associated to $K$ by class field theory and in general, the two $L$-functions differ by some Euler factors  not vanishing at $s=1$. By the results of Kolyvagin \cite{Kolyvagin88} (combined with the work of Gross--Zagier \cite{GrossZagier86} and that of Friedberg--Hoffstein \cite{FriedbergHoffstein95}, see \cite[Section 4]{Darmon97} for details) applied to the elliptic curves $E/\Q$ and $E^D/\Q$ we have
$$L(E,1)L(E,\chi_D,1)\neq0 \Leftrightarrow L(E,1)L(E^D,1)\neq 0\Rightarrow \rank_{\Z}E(\Q)+\rank_{\Z}E^D(\Q)=0.$$
Now the ``Diophantine stability''-result in Corollary \ref{cor:simultnonv} follows directly from that of Theorem \ref{thm:simultnonv} with $d=2$.
\end{proof}
We now prove the claimed non-vanishing result in the presence of the mod $p$ Kurihara conjecture.
\begin{corollary}\label{cor:Kolynonvangen}
Let $E/\Q$ be an elliptic curve and let $p$ be an odd  prime number satisfying $p\nmid c_E$ and $E[p](\Q)=0$. Assume that $E/\Q$ satisfies the mod $p$ Kurihara Conjecture with corresponding $r\geq 1$ (see Conjecture \ref{conj:Kurihara}) and put $\alpha=\d(\mathcal{P}_\mathrm{or}(E; p))>0$.  Then we have
\begin{equation}\label{eq:Kolynonvangen} \#\{\chi\in \mathcal{K}_{p}(X): \ev_p(L(E,\chi,1)/\Omega_E^+)\leq \tfrac{r}{p-1}\}\gg_{E,p} \frac{X}{(\log X)^{1-\alpha}},\quad \text{as }X\rightarrow \infty.\end{equation} 
Furthermore, assume that  $Q=q_1\cdots q_r\in \mathcal{N}_\mathrm{Kato}$ satisfies the congruence (\ref{eq:deltaQ}) and that $r\leq p-2$. Then there exists a primitive Dirichlet character $\chi_r\mod Q$ of order $p$ such that the following holds: for any order $p$ Dirichlet character $\chi$ with conductor dividing $\prod_{\ell\in \mathcal{P}_\mathrm{or}(E;p)}\ell$ there exists $\sigma\in \Gal(\Q_p(\mu_p)/\Q_p)$ (allowed to depend on $\chi$) such that
$$\tfrac{1}{p-1}\leq \ev_p(L(E,\chi(\chi_r)^\sigma,1)/\Omega_E^+)\leq \tfrac{r}{p-1}.$$ \end{corollary} 
\begin{proof}
Let $Q=q_1\cdots q_r\in \mathcal{N}_\mathrm{Kato}$ be such that  $\delta_Q\not\equiv 0\modulo p$, with $\delta_Q$ as in (\ref{eq:kuriharanmb}). Let $\mathcal{L}=(\ell_n)_{n\in \N}$ be a sequence of distinct primes with $\ell_i=q_i$ for $1\leq i\leq r$ and such that 
$$\bigcup_{n\geq r+1}\{\ell_n\}=\mathcal{P}_\mathrm{or}(E; p).$$
Let 
$$\nu_{E}:=\nu_{f_E,\mathcal{L},r}\in \Z_p\left\llbracket \prod_{n\in \N} \Z/p^{m_n} \right\rrbracket,$$ 
be the horizontal $p$-adic $L$-function defined from the weight $2$ modular form $f_E$ associated to $E$ (and $\mathcal{L},r$) as in Definition \ref{def:padicL}. Note that the algebraically normalized $L$-value $L(E,\chi,1)/\Omega_E^+$ divides $\nu_E(\chi)$  by the interpolation property from Corollary \ref{cor:measure} and the explicit choice of period (\ref{eq:neron}) for elliptic curves. Denote by $\overline{\nu}_{E}\in \Z_p\llbracket(\Z/p)^\N\rrbracket$ the pushforward of $\nu_E$ along the canonical projection $\prod_{n\in \N}\Z/p^{m_n}\twoheadrightarrow(\Z/p)^\N$. By the assumption on the Kato--Kolyvagin derivative it follows by combining Corollary \ref{cor:nonvanishingQ}, Lemma \ref{lem:Dalphacongruence} and Proposition \ref{prop:kolyvagin} that we can find a Dirichlet character $\chi_r\modulo Q$ of order $p$ such that 
$$\ev_p(\nu_E(\chi_r))\leq \frac{r}{p-1},$$
and furthermore if $r\leq p-2$ we can ensure that $\chi_r$ is primitive modulo $Q$ by the second part of Theorem \ref{thm:horWP}. Now the lower bound (\ref{eq:Kolynonvangen}) follows directly by applying Corollary \ref{cor:nonvanishingPS} to $\overline{\nu}_E$. The last claim follows by applying Theorem \ref{thm:nonvanishingPS} to the twist of $\overline{\nu}_E$ by the  character $\chi_r$ (as in Lemma \ref{lem:twisting}):  by the bound (\ref{eq:Mbound}) we can ensure that $|M_\nu|\leq p^{\lfloor r/(p-1)\rfloor }=1$. 
Since $\overline{\nu}_E$ has coefficients in $\Z_p$ we conclude that $$\ev_p(\nu_E(\chi^\sigma \chi_r))=\ev_p(\nu_E(\chi^\sigma\chi_r)^\sigma)=\ev_p(\nu_E(\chi (\chi_r)^\sigma)),\quad \sigma\in \Gal(\Q_p(\mu_p)/\Q_p)$$
which yields the wanted conclusion. Finally, since 
$$\nu_E(\chi(\chi_r)^\sigma)\equiv \nu_E(\mathbf{1})\equiv 0\modulo p^{1/(p-1)},$$
we conclude the wanted lower bound on the valuation, as well. \end{proof}
Note that Corollary \ref{cor:Kolynonvangen} implies Theorem \ref{thm:Kolynonvan} from the introduction since we have the uniform lower bound $\d(\mathcal{P}_\mathrm{or}(E;p))\geq \tfrac{1}{\#GL_2(\F_p)}$ when $p$ is $E$-good. Furthermore, we conclude that for elliptic curves without CM we get strong quantitative non-vanishing results for order $d$ twists for $100\%$ of $d$'s:
\begin{proof}[Proof of Corollary \ref{cor:Kolynonvan}]
Let $E/\Q$ be an elliptic curve without complex multiplication. Let $M_E$ be $6$ times the product of the following invariants associated with $E$; the conductor, the Tamagawa factors, the Manin constant, the primes such that the corresponding residual representation is not surjective (which is finite by Serre's Open Image Theorem), and the numerator of (\ref{eq:Rubinconstant}). By combining the progress towards the mod $p$ Kurihara Conjecture from Theorem \ref{thm:BGCSthm} and the non-vanishing result in Corollary \ref{cor:Kolynonvangen} above we see that the non-vanishing (\ref{eq:Kolynonvangen}) holds whenever $p$ does not divide $M_E$, $E$ has good ordinary reduction at $p$, and $p$ is non-anomalous for $E$ (i.e. (\ref{item:nonanom}) in Theorem \ref{thm:Kimthm} is satisfied). It is a classical fact that both the supersingular and anomalous primes for non-CM elliptic curves have density zero among all primes, see e.g. \cite{Murty97}. Thus by applying Theorem \ref{thm:propa} and Corollary \ref{cor:propa} to propagate the non-vanishing we see that there exists a set of primes $\mathcal{A}$ (depending on $E$) of density one among all primes such that whenever $d\geq 2$ is an integer with a prime divisor in $\mathcal{A}$ then there exists $\alpha_d>0$ such that 
\begin{equation} \#\{\chi\in \mathcal{K}_{d}(X): L(E,\chi,1)\neq 0 \}\gg_{E,d} \frac{X}{(\log X)^{1-\alpha_d}},\quad \text{as }X\rightarrow \infty.\end{equation}
This yields the wanted result since $\{d\in \N: \exists p\in \mathcal{A}\text{ s.t. }p|d\}$ has density one among all integers (which follows by a Tauberian argument as in Lemma \ref{lem:rprimefactors} since $\mathcal{A}$ has density one among all primes).
\end{proof}



\end{document}